\documentclass[preprint]{imsart}
\usepackage{amsmath,amsthm,amssymb,mathrsfs,dsfont} 
\usepackage{natbib}
\usepackage{bm,url,wasysym,multirow}
\usepackage{algorithm}
\usepackage{algorithmic}
\usepackage{enumerate}
\usepackage{ascmac}
\usepackage[usenames]{color}
\usepackage{ascmac}
\usepackage[pdftex]{graphicx}
\usepackage{wrapfig}
\usepackage{rotating}
\usepackage{latexsym}
\usepackage{multirow,array}
\usepackage{comment}
\usepackage{arydshln}
\usepackage[subrefformat=parens]{subcaption}
\newtheorem{theorem}{Theorem}[section]
\newtheorem{lemma}[theorem]{Lemma}
\newtheorem{col}[theorem]{Corollary}

\def\ba#1\ea{\begin{align*}#1\end{align*}} %\ba = \begin{algin*}, \ea = \end{align*}
\def\banum#1\eanum{\begin{align}#1\end{align}} %\banum = \begin{algin}, \eanum
\newcommand{\lsb}{\left[}
\newcommand{\rsb}{\right]}

\newcommand{\lpar}{\left(}
\newcommand{\rpar}{\right)}

\newcommand{\Xc}{\mathcal{X}}

\newcommand{\Bc}{\mathcal{B}}
\newcommand{\Rc}{\mathcal{R}}
\newcommand{\Tc}{\mathcal{T}}
\newcommand{\Sc}{\mathcal{S}}
\newcommand{\Fc}{\mathcal{F}}
\newcommand{\Pb}{\mathbb{P}}
\newcommand{\Rb}{\mathbb{R}}

\newcommand{\Nb}{\mathbb{N}}
\newcommand{\argmin}{{\rm arg}\mathop{\rm min\,}\limits}

\newcommand{\Psig}{P_{\sigma^2}}

\newcommand{\bPhi}{\bar{\Phi}}
\newcommand{\proj}{\mathop{\rm proj}\nolimits}
\renewcommand{\appendixname}{Supplementary Material}
\begin{document}

\begin{frontmatter}
\title{Selective inference for the problem of regions via multiscale bootstrap}
\runtitle{Selective inference via multiscale bootstrap}

\author{\fnms{Yoshikazu} \snm{Terada}\ead[label=e1]{terada@sigmath.es.osaka-u.ac.jp}\thanksref{t1}}
\thankstext{t1}{
Jointly affiliated at 
Mathematical Statistics Team,
RIKEN Center for Advanced Intelligence Project (AIP),
1-4-1 Nihonbashi, Chuo-ku, Tokyo 103-0027, Japan.
}
\address{
Graduate School of Engineering Science, Osaka University\\
1-3 Machikaneyama-cho, Toyonaka, Osaka 560-8531, Japan\\
\printead{e1}}
\and
\author{\fnms{Hidetoshi} \snm{Shimodaira}\ead[label=e2]{shimo@i.kyoto-u.ac.jp}\thanksref{t1}}
%\thankstext{t2}{
%???.}
\address{
Graduate School of Informatics, Kyoto University\\
Yoshida Honmachi, Sakyo-ku, Kyoto, 606-8501, Japan\\
\printead{e2}}

\runauthor{Y.~Terada and H.~Shimodaira}

\begin{abstract} % 181 words

A general approach to selective inference is considered for hypothesis testing of the null hypothesis represented as an arbitrary shaped region in the parameter space of multivariate normal model.
This approach is useful for hierarchical clustering where confidence levels of clusters are calculated only for those appeared in the dendrogram, thus subject to heavy selection bias.
Our computation is based on a raw confidence measure, called bootstrap probability, which is easily obtained by counting how many times the same cluster appears in bootstrap replicates of the dendrogram.
We adjust the bias of the bootstrap probability by utilizing the scaling-law in terms of geometric quantities of the region in the abstract parameter space, namely, signed distance and mean curvature.
Although this idea has been used for non-selective inference of hierarchical clustering, its selective inference version has not been discussed in the literature.
Our bias-corrected $p$-values are asymptotically second-order accurate in the large sample theory of smooth boundary surfaces of regions, and they are also justified for nonsmooth surfaces such as polyhedral cones.
The $p$-values are asymptotically equivalent to those of the iterated bootstrap but with less computation.

\end{abstract}


\end{frontmatter}

\maketitle
%
%%
%%%------------------------------------------------------------------------------------------------------------------------------
\section{Introduction}

With recent advances in computer and measurement technologies,
big and complicated data have been common in various application fields,
and thus the importance of exploratory data analysis has been recognized.
From collected data, 
exploratory data analysis is usually used 
to discover useful information and to formulate hypotheses for further data analysis.
For hypotheses obtained by exploratory data analysis, 
classical statistical inference is commonly performed.
However, in the phase of classical inference, 
the effects of hypothesis selection based on data are often ignored,
and thus classical inference will not provide valid tests of the hypotheses.

Inference handling the effects of hypothesis selection appropriately
is called \emph{selective inference} and have been attracted much attention
on inferences after model selection, particularly variable selection in regression settings such as Lasso 
\citep{LockhartEtAl14,LeeEtAl16,fithian2014optimal,RyanTibshiraniEtAl16,RyanTibshiraniEtAl17+}
as well as closely related ideas \citep{benjamini2005false,benjamini2014selective}.
\cite{TaylorTibshirani15} provides a general introduction of selective inference.
\cite{fithian2014optimal} consider a general setting of selective inference and 
\cite{TianTaylor17+} propose the use of randomized response, which implies valid and more powerful tests.
\cite{RyanTibshiraniEtAl17+} consider a bootstrap resampling for the regression problem of \cite{RyanTibshiraniEtAl16}.

In these existing literatures, 
we mainly consider the cases that it is easy to access the parameter space or 
that we know the explicit form of the region on data space which represents the selective event.
On the other hand, in real application problems,
there are situations in which we cannot directly apply these methods.
As a motivating example, let us consider the problem to assess uncertainty in hierarchical clustering
using \emph{bootstrap probability}, which is originally introduced in \cite{Felsenstein:1985:CLP} to the hierarchical clustering of molecular sequences, known as phylogenetic inference.
Bootstrap probability of a cluster is easily computed by counting how many times the same cluster appears in bootstrap replicates.
This is implemented in R package \emph{pvclust} \citep{Suzuki:Shimodaira:2006:PRA}, 
which is used in many application fields such as cell biology (e.g., \citealp{ben2008embryonic}).
There is another approach for accessing uncertainty by estimating the optimal number of clusters via the gap statistic \citep{tibshirani2001estimating}.
Unlike the gap statistic, pvclust suffers from heavy selection bias because frequentist confidence measure is computed for each obtained cluster in the dendrogram (Fig.~\ref{fig:pvclust-tree}).
Unfortunately, in general, it is difficult to know the explicit form of the selective event that the specific cluster is obtained.
There are no existing frameworks to address this kind of issues.

\begin{figure}
\centering
\includegraphics[width=1.0\textwidth]{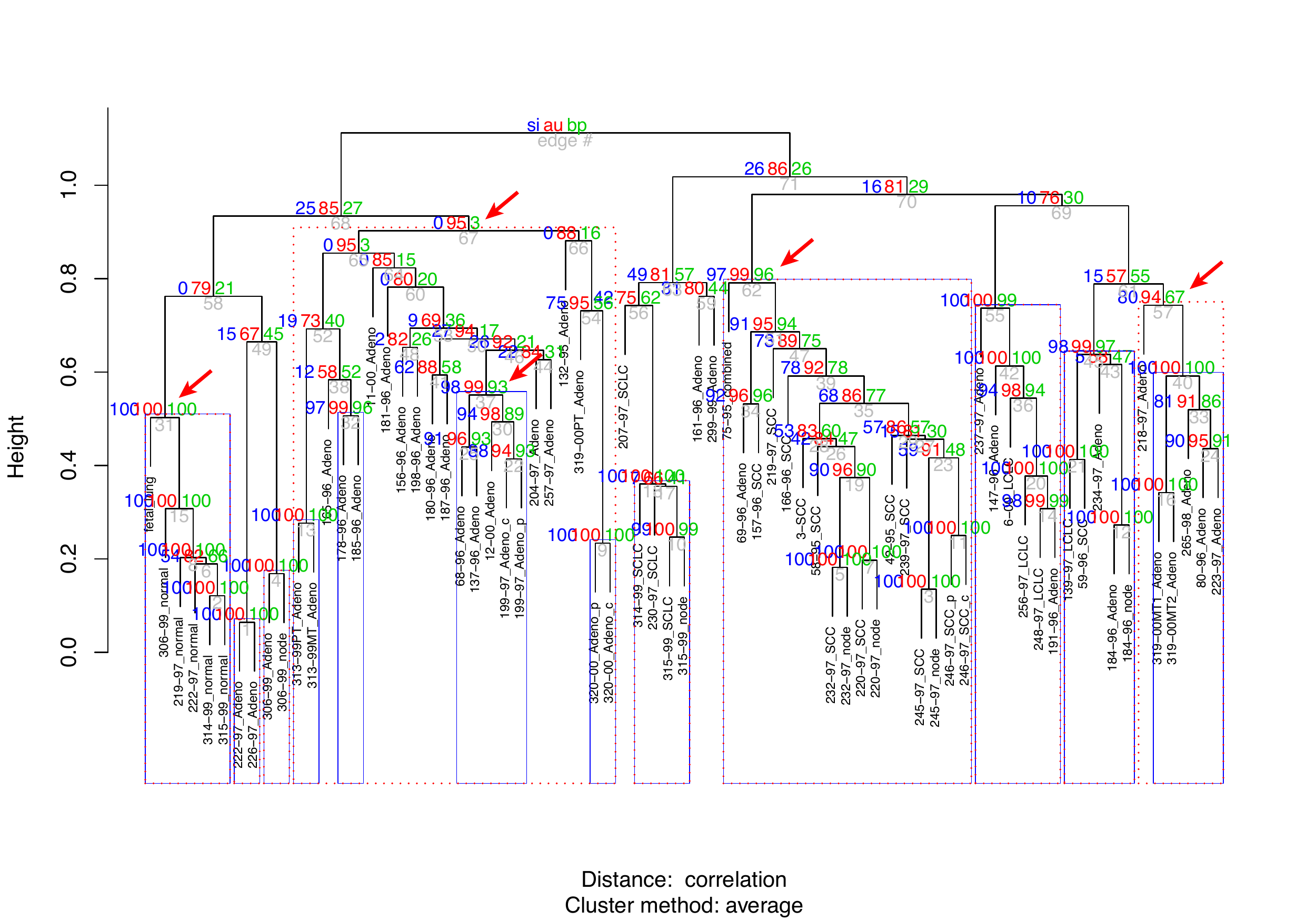}
\caption{Pvclust analysis of the lung dataset.
 Numbers at each branch are $(1-p)\times100$ for
 the raw confidence measure ($p_\mathrm{BP}$; right),
 the bias-corrected $p$-value for non-selective inference ($p_\mathrm{AU}$; middle),
 and the bias-corrected $p$-value for selective inference ($p_\mathrm{SI}$; left).
 Cluster id (shown below each branch) is numbered as $1,\ldots,71$ by the height of branch from bottom to the top.
 Boxes show the outmost clusters with $1-p>0.90$ for $p_\mathrm{SI}$ (solid line) and  $p_\mathrm{AU}$ (dotted line).
 The large difference between $p_\mathrm{AU}$ and $p_\mathrm{SI}$ indicates heavy selection bias.
 Branches with arrows will be examined later.
 See Section~\ref{sec:pvclust-experimenet} (supplementary material) for details.
}
\label{fig:pvclust-tree}
\end{figure}

Geometry plays important roles in the theory behind pvclust.
We consider testing the null hypothesis that the cluster is not ``true".
Hypotheses are represented as arbitrary shaped regions in a parameter space, and
geometric quantities, namely, signed distance and mean curvature, determine the confidence level.
This is the \emph{problem of regions} formulated in \cite{Efron:Halloran:Holmes:1996:BCL} and \cite{Efron:Tibshirani:1998:PR} for computing confidence measures for discrete decision issues such as clustering and model selection.
They argued that bootstrap probability is biased as a frequentist confidence measure, and it can be adjusted by knowing the geometric quantities.
The \emph{multiscale bootstrap} \citep{Shimodaira:2004:AUT} implemented in pvclust
is an idea to estimate the geometric quantities by changing the sample size of bootstrap replicates.
This method has been also used in phylogenetic inference \citep{Shimodaira:Hasegawa:2001:CAC,Shimodaira:2002:AUT}.
However, selective inference has not been considered so far in these works.
In this paper, we provide a general approach to selective inference for the problem of regions with a practical algorithm based on the multiscale bootstrap.

In Section~\ref{sec:overview}, 
we review the problem setting and describe our new method for a general selective inference problem.
A limitation of the method is that only a multivariate normal model is considered, whereas an extension to exponential family of distributions is mentioned in Section~\ref{sec:conclusion}.
However, the transformation invariant property of bootstrap probability leads to robustness to deviation from the normality.
Section~\ref{sec:numerical-results} presents numerical results, including a pvclust example, 
which indicates that our method reasonably works well.
Section \ref{sec:LST} provides the theoretical justification in the large sample theory 
by assuming that the boundary surfaces of the hypothesis and selective regions are smooth.
More specifically, it is shown that the selective $p$-value computed by our algorithm induces 
an unbiased selective test ignoring $O(n^{-1})$ terms.
Moreover, in order to provide a theoretical justification for the case 
that hypothesis and selective regions have possibly nonsmooth boundary surfaces, 
Section~\ref{sec:NFT} deals with the asymptotic theory of \emph{nearly flat surfaces} \citep{Shimodaira08}.
Note that, in the theoretical part of this paper, 
we deal with the case 
that the selection probability does not tend to $0$ or $1$, which corresponds to the third scenario of \cite{TianTaylor17+}.
All the technical details of experiments and proofs are found in Supplementary Material.

%%%------------------------------------------------------------------------------------------------------------------------------
%%
%

%
%%
%%%------------------------------------------------------------------------------------------------------------------------------
\section{Computing $p$-values via multiscale bootstrap} \label{sec:overview}

\subsection{Problem setting for selective inference} \label{sec:problem-setting}

We discuss the theory of the problem of regions by following the simple setting of \cite{Efron:Tibshirani:1998:PR} and \cite{Shimodaira:2004:AUT}.
Let $y\in\mathbb{R}^{m+1}$,  $m\ge0$, be an observation of random vector $Y$ following multivariate normal model
\begin{align}\label{eq:model}
Y\sim N_{m+1}(\mu,I_{m+1})
\end{align}
with unknown parameter $\mu\in \Rb^{m+1}$ and covariance identity $I_{m+1}$.

Given \emph{hypothesis regions} $H_i,\,i=1,2,\ldots$ in $\mathbb{R}^{m+1}$, we would like to know if $\mu$ belongs to $H_i$ or not.
Since $y$ is an unbiased estimate of $\mu$, a large distance between $y$ and $H_i$ is an evidence that $\mu$ does \emph{not} belong to $H_i$, leading to rejection of the null hypothesis $\mu\in H_i$ by hypothesis testing.
Instead of testing all the hypotheses, we are prone to select a part of hypotheses which may be easily rejected by the observed data.
For formulating this selection process, we introduce \emph{selective regions} $S_i,\,i=1,2,\ldots $ in $\mathbb{R}^{m+1}$, and see if $y$ belongs to $S_i$ or not.
If $y$ belongs to $S_i$ ($y\in S_i$), then $H_i$ is selected for hypothesis testing.
Otherwise $H_i$ is simply ignored and no decision will be made on $H_i$.

Our goal is to compute a non-randomized frequentist $p$-value $p(H|S,y)$ for the selective inference.
The index $i$ is omitted here, because only one hypothesis is considered at a time.
The $p$-value should control the selective rejection probability
 $P(Y\in R \mid Y\in S, \mu)$, where
 $R=\{y \mid p(H|S,y) < \alpha \} \subset S$ is the rejection region at a significance level $\alpha\in(0,1)$.
To control the selective type-I error,  $P(Y\in R \mid Y\in S, \mu)$ is not more than $\alpha$ at any $\mu\in H$.
Unbiased tests further request that it is not less than $\alpha$ at any $\mu\not\in H$, and thus
it equals $\alpha$ at any $\mu\in\partial H$.

A simple model for publication bias, which is called as the \emph{file drawer problem} by \cite{rosenthal1979file},
is easily solved for selective inference \citep{fithian2014optimal,TianTaylor17+}, where
$y, \mu \in \mathbb{R}$ ($m=0$), $H=\{\mu \mid \mu \le 0 \}$ and $S=\{y \mid y > c\}$ for some $c\in\mathbb{R}$.
Noticing $\partial H = \{ 0 \}$, an unbiased test is obtained by specifying $R=\{y \mid y > d\}$ with
${\bar\Phi(d)}/{\bar\Phi(c)} = \alpha$, and we have
\begin{equation} \label{eq:file-drawer-c}
	p(H|S,y) = \bar\Phi(y)/\bar\Phi(c),\quad y>c,
\end{equation}
where $\bar\Phi(x)=1-\Phi(x)$ is the upper tail probability of the standard normal distribution.

In this paper, we provide \emph{approximately unbiased $p$-values} for arbitrary shaped regions ($m\ge0$), which approximately satisfy
\begin{align}\label{eq:goal}
\frac{P(p(H | S,Y)<\alpha \mid \mu)}{P(Y\in S\mid \mu)} =  \alpha,\quad \forall\mu \in \partial H
\end{align}
up to specified asymptotic accuracy.
This will be solved by adjusting deviation from $p(H|S,y) = \bar\Phi(y_{m+1})/\bar\Phi(c)$,
where the file drawer problem (\ref{eq:file-drawer-c}) is considered for $y_{m+1}$ as
$H=\{\mu \mid \mu_{m+1}\le0 \}$ and $S=\{y \mid y_{m+1} >c \}$.
By setting $S=\mathbb{R}^{m+1}$, our argument reduces to the ordinary (non-selective) inference for computing $p(H|y)$, which approximately satisfies
\begin{align}\label{eq:goal-unselective}
P(p(H | Y)<\alpha \mid \mu) =  \alpha,\quad \forall\mu \in \partial H,
\end{align}
by adjusting deviation from $p(H|y) = \bar\Phi(y_{m+1})$.
Although unbiased tests are sometimes criticized for non-existence \citep{Lehmann:1952:TMH} and
unfavorable behavior \citep{perlman1999emperor}, we avoid these issues by considering only approximate solutions based on (\ref{eq:file-drawer-c}).

Algorithm~\ref{alg:simbp} shown in Section~\ref{sec:method} computes the $p$-values from the multiscale bootstrap probabilities of $H$ and $S$.
The bootstrap probability of region $H$ at scale $\sigma^2>0$ is defined by
\[
\alpha_{\sigma^2}(H | y)=P_{\sigma^2}(Y^\ast\in H \mid y),
\]
where $P_{\sigma^2}(\cdot | y)$ is the probability with respect to
\begin{align}\label{eq:bpmodel}
Y^\ast|y \sim N_{m+1}(y,\sigma^2I_{m+1}).
\end{align}
All we need for computing the $p$-value are bootstrap probabilities at several $\sigma^2$ values.
We consider the parametric bootstrap (\ref{eq:bpmodel}) in the theory, but we perform nonparametric bootstrap in real applications; the connection between the two resampling schemes is explained in Section~\ref{sec:pvclust-intro}.

Setting a good $S$ for a given $H$ would be an interesting issue.
For increasing the chance of rejecting $H$, setting $S \subset H^c$, i.e., a subset of the complement set $\Rb^{m+1}\setminus H$, is reasonable, because observing $y\in H$ would not be an evidence against the null hypothesis.
On the other hand, coarser selection would improve the power, according to the monotonicity of selective error in the context of ``data curving'' \citep{fithian2014optimal}.
A compromise would be $S = H^c$, because it is the largest (coarsest) region that does not overlap with $H$.
Taking a smaller (finer) selective region reduces the ``leftover information''.
Throughout this paper, we assume $S = H^c$, thus $\partial S = \partial H$, in illustrative examples and informal argument.

\subsection{Bootstrap probability in pvclust} \label{sec:pvclust-intro}

The argument below as well as Section~\ref{sec:pvclust-details} (supplementary material) explains
how the theoretical setting of the previous section is related to the nonparametric bootstrap implemented in pvclust. 
Let us consider hierarchical clustering of the lung dataset (available in pvclust) of micro-array expression profiles of $n=916$ genes for $p=73$ tissues \citep{garber2001diversity}.
In our setting, genes, instead of tissues, are independent samples (see Section~\ref{sec:pvclust-sampling} for a specific model).
Then the tree building process is quite similar to phylogenetic inference
\citep{Felsenstein:1985:CLP,Efron:Halloran:Holmes:1996:BCL},
where sites of aligned DNA sequences, instead of species, are independent samples.

Hierarchical clustering is formally described as follows.
The dataset of sample size $n$ is denoted as $\Xc_n=(x_1,\dots,x_n)$ with each $x_i \in \mathbb{R}^p$.
Euclidean distance $\frac{1}{n}\sum_{t=1}^n (x_{ti} - x_{tj})^2$ and sample correlation are commonly used for pairwise distances between tissues. 
Let $d_n = {\tt dist}(\Xc_n) \in\mathbb{R}^{p(p-1)/2}$ be the lower-triangular part of the $p\times p$ distance matrix, from which a tree building algorithm computes the dendrogram as shown in Fig.~\ref{fig:pvclust-tree}.
A cluster $G \subset \{1,\ldots,p\}$ is meant a subset of the $p$ tissues, and $k_\text{select}$ is the number of clusters appeared in the dendrogram, excluding trivial clusters $\{1\},\ldots, \{p\},\{1,\ldots,p\}$.
For $p=73$ tissues, only $k_\text{select} = p-2 = 71$ clusters are selected from $k_\text{all}=2^p-(p+2) = 9.44\times 10^{21}$ possible clusters.
We denote the dendrogram as $ {\tt hclust}(d_n) = \{G_1,\ldots, G_{k_\text{select}} \}$.

Before discussing hypothesis testing, we have to clarify what ``true'' clusters mean here.
We imagine a situation that  infinitely many genes can be collected by taking the limit of $n\to\infty$.
Now the dataset  $\Xc_\infty$ and the distance matrix $d_\infty={\tt dist}(\Xc_\infty) \in \mathbb{R}^{p(p-1)/2}$ can be interpreted as 
the population and the ``true" distance matrix, respectively.
By applying the tree building algorithm to $d_\infty$, we get ``true'' dendrogram ${\tt hclust}(d_\infty)$ as well as ``true'' clusters in it.
They could be very poor representations of reality, but simply what we would observe when the number of genes is very large.

For quantifying the random variation of ${\tt hclust}(\Xc_n)= {\tt hclust}({\tt dist}(\Xc_n))$, we generate bootstrap replicates by resampling $x_i$ with replacement from $\Xc_n$.
Let $\Xc_{n'}^*=(x_1^*,\dots,x_{n'}^*)$ be a bootstrap replicate of sample size $n'$.
In the ordinary bootstrap, the sample size is the same as the original data, and $n'=n$.
Similar to the subsampling or $m$-out-of-$n$ bootstrap \citep{Politis:Romano:1994:LSC}, we allow $n'
$ can be any positive integer in the multiscale bootstrap.
For each $\Xc_{n'}$, we compute $d_{n'}^* = {\tt dist}(\Xc_{n'}^*)$ and ${\tt hclust}(d_{n'}^*)$.
We repeat this process $B$ times, say $B=1000$, to generate ${\tt hclust}(\Xc_{n'}^{* b})$, $b=1,\ldots, B$.
For a cluster $G \in {\tt hclust}(\Xc_n)$, let $C(G)$ be the number of times that the same cluster appears in the $B$ instances of bootstrap replicates
\begin{equation} \label{eq:pvclust-count}
	C(G) = \#\{\Xc_{n'}^{* b} \mid  G \in {\tt hclust}(\Xc_{n'}^{* b}), \,b=1,\ldots,B  \}.	
\end{equation}
Then $C(G)/B$ is an estimate of the bootstrap probability of the cluster $G$ with error $O_p(B^{-1/2})$.
In Fig.~\ref{fig:pvclust-tree}, this value for $n'=n$ is shown at each branch as $1-p_\mathrm{BP}$,
which has been used widely as a confidence measure of the cluster in phylogenetic analysis too \citep{Felsenstein:1985:CLP}.

\subsection{Our proposed method} \label{sec:method}

\begin{figure}
\centering
\includegraphics[width=0.45\textwidth]{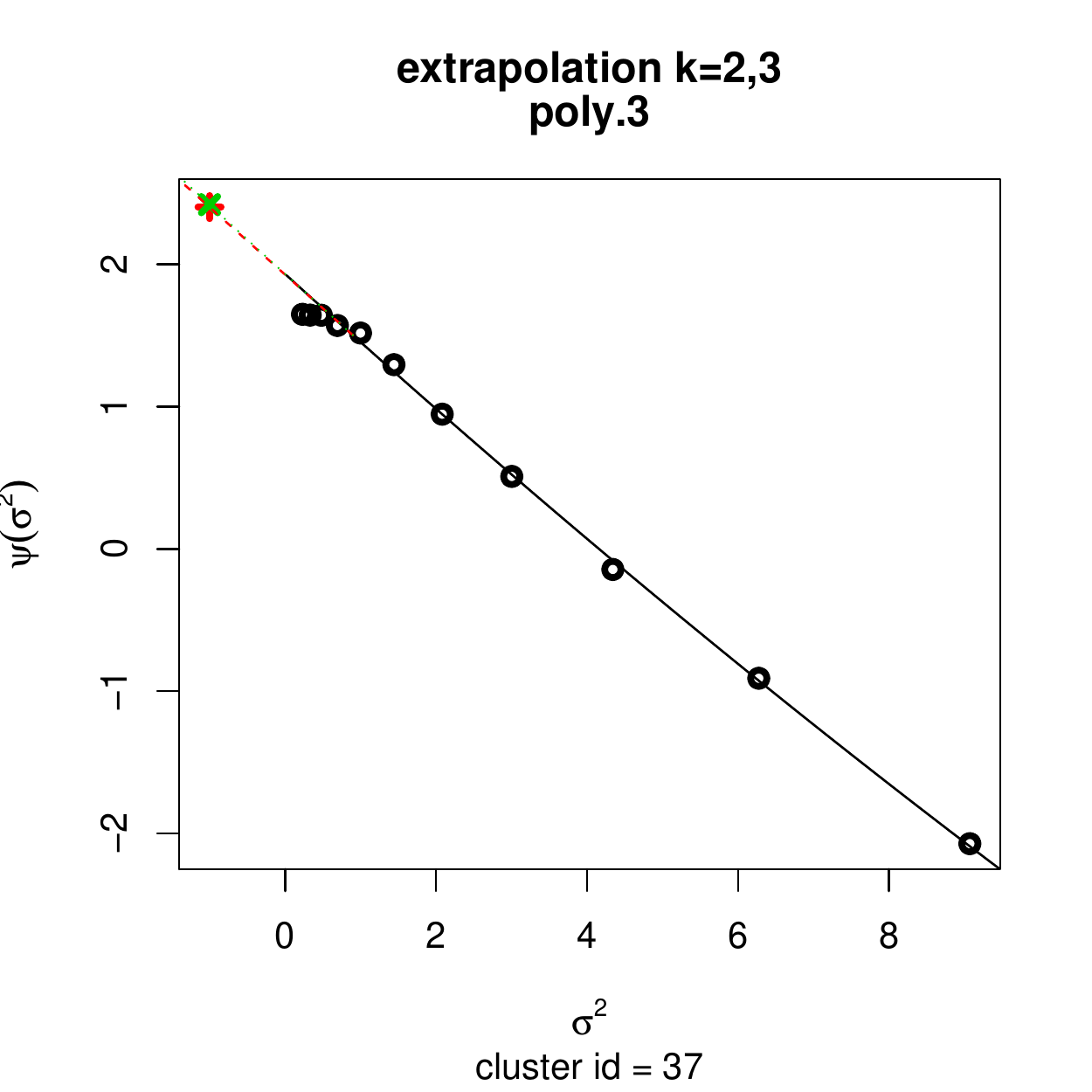}
\includegraphics[width=0.45\textwidth]{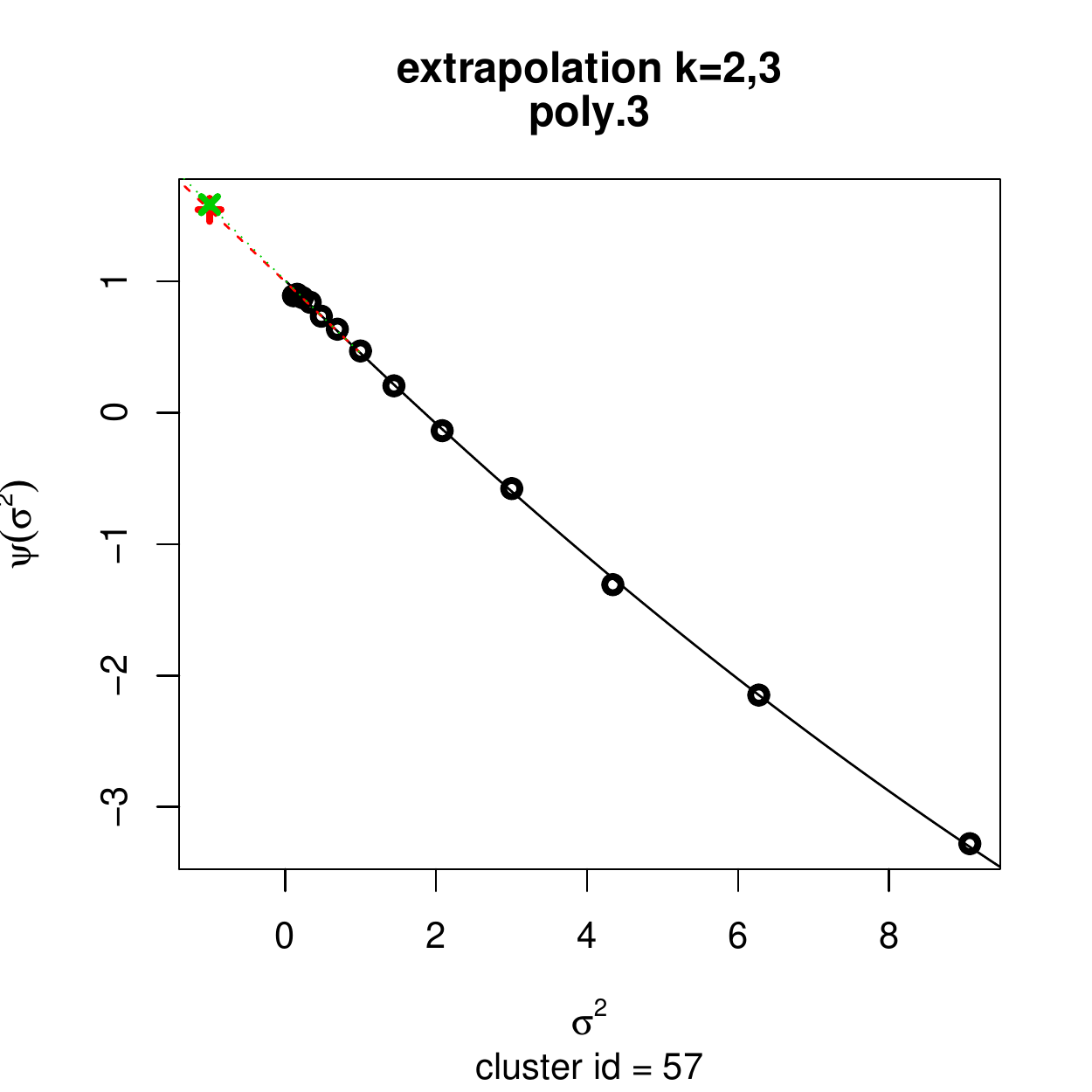}

\includegraphics[width=0.45\textwidth]{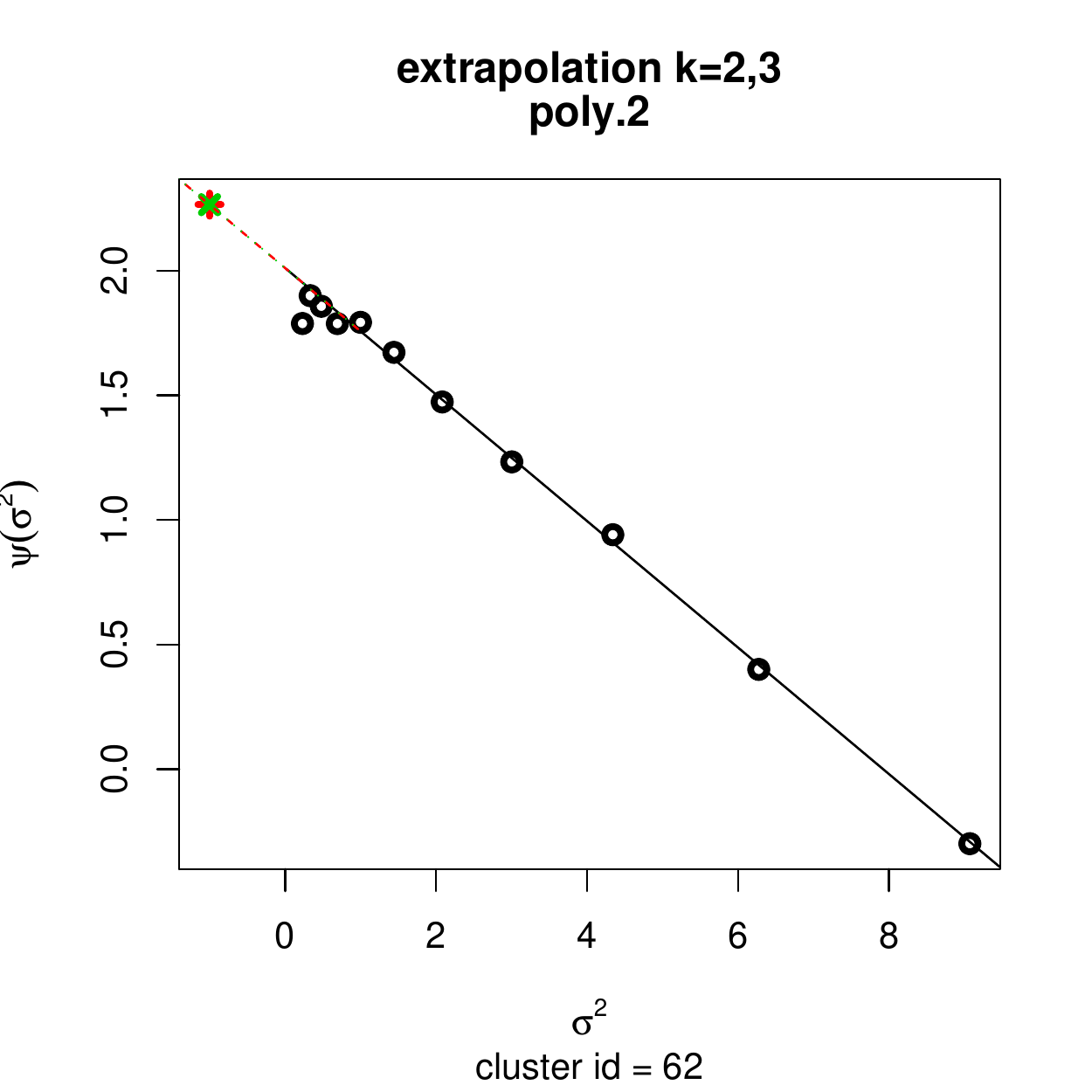}
\includegraphics[width=0.45\textwidth]{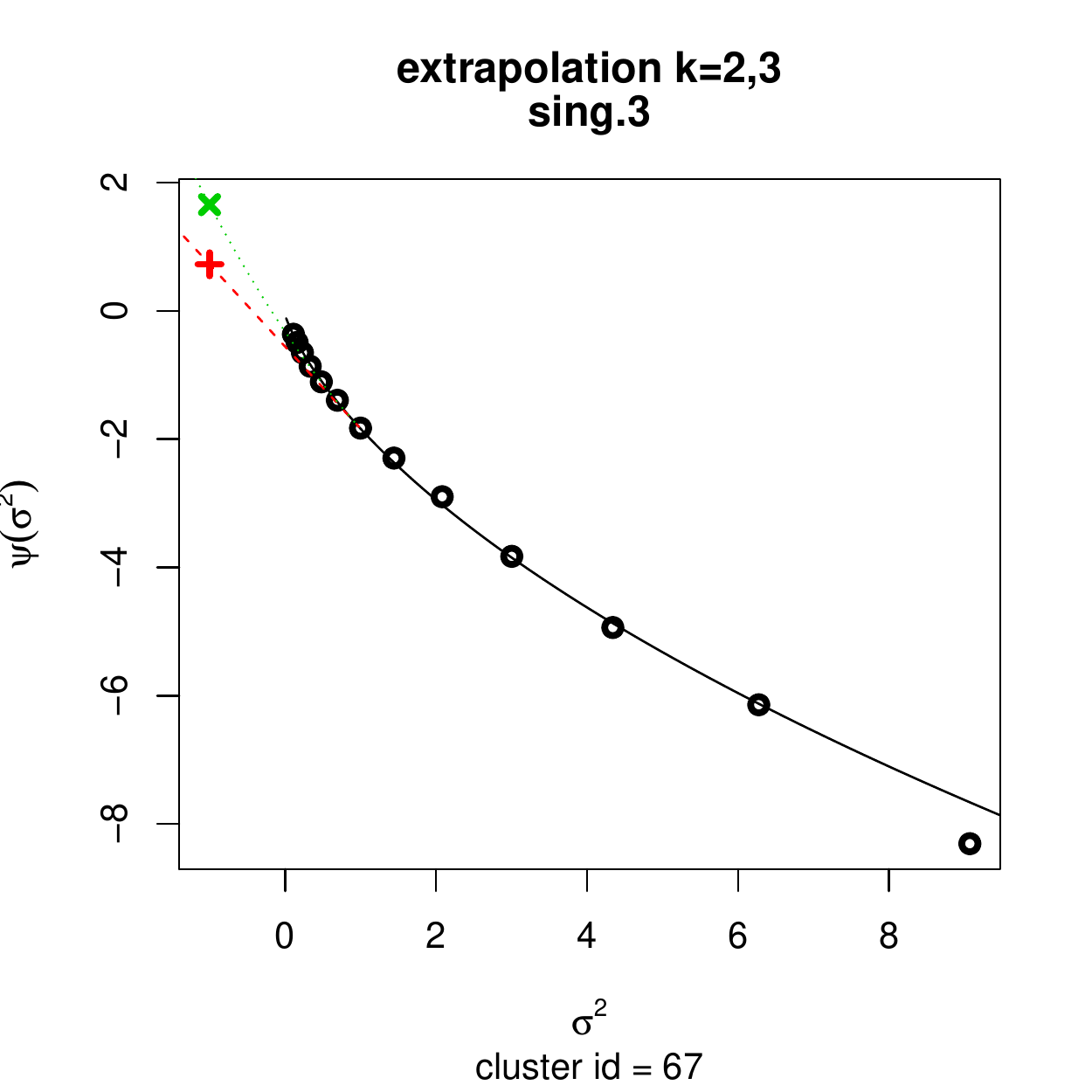}

\caption{Fitting of a parametric model $\varphi_H(\sigma^2|\beta)$
to the normalized bootstrap $z$-values $\psi_{\sigma^2}(H|y)$ at 13 values of $\sigma^2$
for cluster id = 37, 57, 62, and 67 of the lung dataset.
The best model is selected by AIC as indicated in each panel (poly.2, poly.3 and sing.3 which are defined in Section~\ref{sec:models}).
Extrapolation of $\psi_{\sigma^2}(H|y)$ to $\sigma^2=-1$ is computed by
$\varphi_{H,k}$ with $k=2$ $(+)$ and $k=3$ $(\times)$.
See Section~\ref{sec:pvclust-experimenet} (supplementary material) for details.}
\label{fig:pvclust-sbfit}
\end{figure}

In this paper, we propose a general multiscale bootstrap algorithm for computing approximately unbiased $p$-values of selective inference.
We assume that there exists a transformation $f_n$ so that (\ref{eq:model}) holds for $y = f_n(\Xc_n)$ and (\ref{eq:bpmodel}) holds for $y^* = f_n(\Xc_{n'}^*)$.
The nonparametric version of multiscale bootstrap changes the sample size $n'$ of $\Xc_{n'}^*$ and in effect changes the scale $\sigma^2=n/n'$ of $Y^*$ in (\ref{eq:bpmodel}).

For example, a realization of $f_n$ for pvclust is given in Section~\ref{sec:pvclust-fn} (supplementary material) so that
the event $G \in {\tt hclust}(\Xc_n)$ corresponds to the event $y\in S$ and the hypothesis region is specified as $H=S^c$.
In other words, for the clusters in the obtained dendrogram, we perform selective inference to test the null hypothesis that the cluster is \emph{not} true.
Then the bootstrap probability of $S$ is computed as
\[
	\alpha_{\sigma^2}(S | y) = C(G)/B + O_p(B^{-1/2})
\]
from the frequency $C(G)$ in (\ref{eq:pvclust-count}), and the bootstrap probability of $H=S^c$ is obtained as $\alpha_{\sigma^2}(H | y) = 1- \alpha_{\sigma^2}(S | y)$.

More generally, by assuming that we can tell if $y^* \in H$ and $y^* \in S$ from $\Xc_{n'}$, the bootstrap probabilities $\alpha_{\sigma^2}(H | y)$ and $\alpha_{\sigma^2}(S | y)$ at several $\sigma^2$ values are computed as frequencies with respect to $B$ instances of $\Xc_{n'}^*$ at several $n'$ values.
Since we actually work on $\Xc_{n'}^*$ for computing bootstrap probabilities, the transformation $f_n$ does not need to be identified in practice.
We define \emph{normalized bootstrap $z$-value} 
\citep{Shimodaira08,Shimodaira2014higherorder} as
\begin{equation} \label{eq:psi-H}
  \psi_{\sigma^2}(H|y) = 
   \sigma\bar\Phi^{-1} ( \alpha_{\sigma^2}(H | y)),\quad \sigma^2>0,
\end{equation}
and \emph{normalized bootstrap probability} as
\begin{equation} \label{eq:p-sigma}
  p_{\sigma^2}(H|y)=\bar\Phi(\psi_{\sigma^2}(H | y)).
\end{equation}
Given bootstrap probabilities at several $\sigma^2$ values, the idea is to estimate the functional forms of $\psi_{\sigma^2}(H|y)$ and $\psi_{\sigma^2}(S|y)$ with respect to $\sigma^2$ using an appropriate parametric model $\varphi(\sigma^2 | \beta)$ with parameter $\beta$.
Examples of model fitting are shown in Fig.~\ref{fig:pvclust-sbfit}.
The theory shows that a good model is the linear model
\begin{equation} \label{eq:phi-linear}
 \varphi(\sigma^2 | \beta) = \beta_0 + \beta_1 \sigma^2
\end{equation}
with respect to $\sigma^2$; denoted as poly.2 in Section~\ref{sec:models}.
Using the estimated parameter $\hat\beta=(\hat\beta_0, \hat \beta_1)$, we extrapolate (\ref{eq:psi-H}) and (\ref{eq:p-sigma}) to $\sigma^2\le 0$, from which we can calculate an approximately unbiased $p$-value for selective inference as well as that for non-selective inference.
Our method is summarized in Algorithm~\ref{alg:simbp}.
The procedure (A) is justified for smooth boundary surfaces of the regions in Section~\ref{sec:LST},
and (B) is justified for both smooth and nonsmooth surfaces in Section~\ref{sec:NFT}.

\begin{algorithm}                      
\caption{Computing approximately unbiased $p$-values}         
\label{alg:simbp}
\begin{algorithmic}[1]
\STATE Specify several $n^\prime\in \Nb$ values, and set $\sigma^2=n/n'$ for each $n'$. Set the number of bootstrap replicates $B$, say, 1000.
\STATE For each $n'$, perform bootstrap resampling to generate $Y^*$  for $B$ times and compute $\alpha_{\sigma^2}(H | y)=C_H/B$ and $\alpha_{\sigma^2}(S | y)=C_S/B$ by counting the frequencies
$C_H = \#\{Y^\ast\in H\}$ and $C_S = \#\{Y^\ast\in S\}$.
We actually work on $\Xc_{n'}^\ast$ instead of $Y^\ast$.
Compute $\psi_{\sigma^2}(H|y) = \sigma\bar\Phi^{-1} ( \alpha_{\sigma^2}(H|y))$ and
$\psi_{\sigma^2}(S|y) = \sigma\bar\Phi^{-1} ( \alpha_{\sigma^2}(S|y))$.
Note that, for the case of $S=H^c$, we only need to count $C_S$, because
$C_H = B-C_S$, $\alpha_{\sigma^2}(H|y) = 1-\alpha_{\sigma^2}(S|y)$ and $\psi_{\sigma^2}(H|y) = - \psi_{\sigma^2}(S|y)$.
\STATE Estimate parameters $\beta_H(y)$ and $\beta_S(y)$ by fitting models
\[
\psi_{\sigma^2}(H | y)=\varphi_{H}(\sigma^2 | \beta_H) \text{ and }
\psi_{\sigma^2}(S | y)=\varphi_{S}(\sigma^2 | \beta_S),
\]
respectively.
The parameter estimates are denoted as $\hat{\beta}_H(y)$ and $\hat{\beta}_S(y)$.
If we have several candidate models, apply above to each and choose the best model based on AIC value.

\STATE Approximately unbiased $p$-values of non-selective inference ($p_\mathrm{AU}$) and of
selective inference ($p_\mathrm{SI}$) are computed by one of (A) and (B) below.
\begin{description}
\item{(A)} Extrapolate $\psi_{\sigma^2}(H|y)$ and $\psi_{\sigma^2}(S|y)$ to $\sigma^2=-1$ and 0, respectively, by
\[
z_H = \varphi_{H}(-1 | \hat{\beta}_H(y)) \text{ and } z_S =  \varphi_{S}(0 | \hat{\beta}_S(y)),
\]
and then compute $p$-values by
\[
p_{\mathrm{AU}}(H|y) = \bar{\Phi}(z_H) \text{ and } 
p_{\mathrm{SI}}(H | S, y) = \frac{\bar{\Phi}(z_H)}{\bar{\Phi}( z_H + z_S )}.
\]

\item{(B)}  Specify $k\in \Nb$, $\sigma_0^2,\sigma_{-1}^2>0$ (e.g., $k=3$ and $\sigma_{-1}^2=\sigma_0^2=1$).
Extrapolate $\psi_{\sigma^2}(H|y)$ and $\psi_{\sigma^2}(S|y)$ to $\sigma^2=-1$ and 0, respectively, by
\[
z_{H,k} = \varphi_{H,k}(-1 | \hat{\beta}_H(y),\sigma_{-1}^2) \text{ and } 
z_{S,k} =  \varphi_{S,k}(0 | \hat{\beta}_S(y),\sigma_0^2),
\]
where the Taylor polynomial approximation of $\varphi_H$ at $\tau^2>0$ with $k$ terms is:
\[
  \varphi_{H,k}(\sigma^2 | \hat{\beta}_H(y),\tau^2)
=\sum_{j=0}^{k-1}
\frac{(\sigma^2-\tau^2)^j}{j!}
\frac{\partial^j \varphi_{H}(\sigma^2 | \hat{\beta}_H(y))}{\partial (\sigma^2)^j}\Biggr|_{\sigma^2=\tau^2},
\]
and $\varphi_{S,k}$ is defined similarly. Then compute $p$-values by
\[
p_{\mathrm{AU},k}(H|y) = \bar{\Phi}(z_{H,k})  \text{ and } 
p_{\mathrm{SI},k}(H | S, y) = \frac{\bar{\Phi}(z_{H,k})}{\bar{\Phi}( z_{H,k} + z_{S,k} )}.
\]
\end{description}
\end{algorithmic}
\end{algorithm}

The proposed method satisfies the following two properties.
\begin{description}
\item{(a)} Using only binary responses whether $Y^\ast\in H$ and $Y^\ast \in S$.
\item{(b)} Resampling only from $y$ instead of the null distribution.
\end{description}
With these properties, it is not necessary to know the dimension $m$, the transformation $f_n$, and the shapes of $H$ and $S$ in the parameter space, thus leading to wide applications and robustness to deviation from the multivariate normal model.

There could be several $f_n$ (possibly different $m$) exist, and someone may wonder that $p$-values depend on it.
However, the bootstrap probabilities as well as the $p$-values computed from them are transformation invariant, and they are in fact computed in the original space of $\Xc_n$ without even defining the transformations.
This property may be referred to as \emph{bootstrap trick} by analogy with the kernel trick of the support vector machine.

\subsection{Preview of the large sample theory} \label{sec:motivation}

Why does this method work?
To explain the reason for the case of $S=H^c$, let us introduce two geometric quantities of \cite{Efron:1985:BCI}.
First note that projection is the point on $\partial H$ closest to $y\in \Rb^{m+1}$ defined as
\[
  \proj(H|y)= \argmin_{\mu\in\partial H} \|y - \mu\|.
\]
\emph{Signed distance} from $y$ to $\hat\mu=\proj(H|y)$, denoted as $t=\eta(H|y)$, is
$t=\|y - \hat\mu\|>0$ for $y\not\in H$ and $t=-\|y - \hat\mu\|\le0$ for $y\in H$.
\emph{Mean curvature} of $\partial H$ at $\hat\mu$, denoted as $\hat\gamma = \gamma(H|y)$, is half the trace of Hessian matrix of the surface at $\hat\mu$ with sign $\hat\gamma>0$ when curved towards $H$ (e.g., convex $H$) and $\hat\gamma\le0$ otherwise (e.g., concave $H$).
For $\proj(H|\mu)$, the signed distance and the mean curvature are $\eta = \eta(H|\mu)$ and  $\gamma = \gamma(H|\mu)$.
Then signed distance $T=\eta(H|Y)$ follows the normal distribution
\[
  T \sim N(\eta+\gamma, 1)
\]
by ignoring the error of $O_p(n^{-1})$, where $\eta=O(1)$ under the local alternatives and $\gamma=O(n^{-1/2})$.
Our large sample theory has second order asymptotic accuracy correct up to $O_p(n^{-1/2})$ by ignoring $O_p(n^{-1})$ terms, and the equality with this accuracy will be indicated by $\doteq$ below.
For example, $\hat \gamma = \gamma + O_p(n^{-1})$ will be denoted as $\hat\gamma \doteq \gamma$.

Hypothesis testing is now a slight modification of the file drawer problem.
The null hypothesis $\mu\in H$ is expressed as $\eta \le0$ and the selective event $y\in H^c$ is expressed as $t>0$.
Since $T - \gamma(H|Y)  \sim N(\eta, 1)$, ignoring $O_p(n^{-1})$, is the pivot statistic distributed as $N(0,1)$ at $\eta=0$,
the $p$-value for the ordinary (non-selective) inference is 
\[
	p(H|y) \doteq \bar\Phi(t-\hat\gamma).
\]
The selective $p$-value (\ref{eq:file-drawer-c}) for the selective event $t-\hat\gamma > - \hat\gamma$  ( $\doteq -\gamma$) becomes
\[
	p(H|H^c,y) \doteq \bar\Phi(t-\hat\gamma)/\bar\Phi(-\hat\gamma),\quad t>0.
\]
Since $p(H|H^c,y)  \doteq p(H|y)/\bar\Phi(-\gamma)$,
the selective $p$-value adjusts the non-selective $p$-value by the selection probability $\bar\Phi(-\gamma)$ in the denominator.
These $p$-values are particularly simple when $\gamma  \equiv 0$, i.e., the boundary surface $\partial H$ is flat.
$p(H|y) \doteq \bar\Phi(t)$ is the $p$-value for one-tailed $z$-test of the null hypothesis $\eta=0$, and
$p(H|H^c,y) \doteq \bar\Phi(t)/\bar\Phi(0) \doteq 2 p(H|y)$ is the $p$-value for two-tailed $z$-test.
The selective inference considers the fact that we do not know which of $t>0$ and $t<0$ is observed in advance, thus doubling the non-selective $p$-value.

Bootstrap probability is also expressed in terms of the geometric quantities.
From the argument of \cite{Efron:Tibshirani:1998:PR} and \cite{Shimodaira:2004:AUT},
signed distance $T^* = \eta(H|Y^*)$ for the bootstrap replicate $Y^*$ follows
\[
  T^* | y \sim N(t+\hat\gamma \sigma^2, \sigma^2).
\]
Therefore, the bootstrap probability is 
\[
	\alpha_{\sigma^2}(H|y) = P_{\sigma^2}(T^*\le 0 \mid y) \doteq \bar\Phi(t \sigma^{-1} + \hat\gamma \sigma),
\]
which will be shown rigorously in (\ref{eq:bp-at-origin}).
In particular for $\sigma^2=1$, it becomes the ordinary Bootstrap Probability (BP)
\[
  p_{\mathrm{BP}}(H|y)=\alpha_1(H|y) \doteq \bar\Phi(t+\hat\gamma).
\]
This provides naive estimation of the $p$-values as
$p_\text{BP}(H|y) = p(H|y)+O_p(\gamma)$ and $2p_\text{BP}(H|y) = p(H|H^c,y)+O_p(\gamma)$.
The bias caused by $\gamma=O(n^{-1/2})$ will be adjusted by multiscale bootstrap.
Note that the scaling-law of $\alpha_{\sigma^2}(H|y)$ is intuitively obvious by rescaling (\ref{eq:bpmodel}) with the factor $\sigma^{-1}$ so that  $t$ and $\hat\gamma$ in $\alpha_1(H|y)$ are replaced with $t\sigma^{-1}$ and $\hat\gamma\sigma$, respectively.

We can compute the $p$-values from the multiscale bootstrap probabilities.
By fitting the linear model (\ref{eq:phi-linear}) to the observed values of
\[
  \psi_{\sigma^2}(H|y) = \sigma \bar\Phi^{-1} (\alpha_{\sigma^2}(H|y)) \doteq t + \hat\gamma \sigma^2
\]
at several $\sigma^2$ values, the regression coefficients are estimated as $\hat\beta_0 \doteq t$ and $\hat\beta_1 \doteq \hat\gamma$, from which we can extrapolate $\psi_{\sigma^2}(H|y)$ to $\sigma^2\le0$.
In particular for $\sigma^2=-1$, we have the pivot statistic
\[
	\psi_{-1}(H|y) \doteq t - \hat\gamma.
\]
Therefore, $p(H|y)$ is computed as 
\[
  p_{\mathrm{AU}}(H|y) = \bar\Phi(\psi_{-1}(H|y)) \doteq \bar\Phi(t-\hat\gamma)
\]
for the Approximately Unbiased (AU) test of non-selective inference \citep{Shimodaira:2002:AUT}.
The error in (\ref{eq:goal-unselective}) for $ p_{\mathrm{AU}}(H|y)$ is in fact $O(n^{-3/2})$ \citep{Shimodaira:2004:AUT}, which was originally shown for the third-order pivot statistic \citep{Efron:1985:BCI,Efron:Tibshirani:1998:PR}. 
Noticing (\ref{eq:p-sigma}), we may state that $p_{\mathrm{AU}}(H|y)=p_{-1}(H|y)$ adjusts the bias of $p_{\mathrm{BP}}(H|y)=p_{1}(H|y)$ by formally changing $\sigma^2=1$ ($n'=n$) in $p_{\sigma^2}(H|y)$ to $\sigma^2=-1$ ($n'=-n)$.

The selective $p$-value $p(H|H^c,y)$ is computed similarly.
Our idea for approximately unbiased test of Selective Inference (SI)
is to calculate $\psi_{0}(H|y)\doteq t$ as well as $\psi_{-1}(H|y)\doteq t-\hat\gamma$ by the multiscale bootstrap, from which we define
\[
  p_{\mathrm{SI}}(H|H^c,y)  = \frac{\bar\Phi(\psi_{-1}(H|y))}{\bar\Phi(\psi_{-1}(H|y) - \psi_{0}(H|y))}
  \doteq \frac{\bar\Phi(t - \hat\gamma)}{\bar\Phi(- \hat\gamma)}.
\]
This $p$-value satisfies (\ref{eq:goal}) with error $O(n^{-1})$.
For simplifying the notation, we may write $p_{\mathrm{SI}}(H|y)$ for $p_{\mathrm{SI}}(H|H^c,y)$ by omitting $H^c$.
It follows from $\psi_0(H|y) = -\psi_0(H^c|y)$ that $ p_{\mathrm{SI}}(H|y)$ is $p_{\mathrm{SI}}(H|S,y)$ of Algorithm~\ref{alg:simbp} for the case $S=H^c$.

\subsection{Bias correction by resampling from the null distribution} \label{sec:intro-nulldist}

Multiscale bootstrap generates $Y^*$ from $y$ in (\ref{eq:bpmodel}).
Replacing $y$ with $\proj(H|y)$,
\begin{equation} \label{eq:yboot-zero}
	Y^* | \proj(H|y) \sim N_{m+1}(\proj(H|y), I_{m+1})
\end{equation}
simulates the null distribution of $Y$ generated from $\mu\in\partial H$.
By letting $t=0$ in $p_{\mathrm{BP}}(H|y)$, we have
$p_{\mathrm{BP}}(H|\proj(H|y)) \doteq \bar\Phi(\hat\gamma)$,
and therefore
\[
	 z_\text{proj}(H|y) = \bar\Phi^{-1}(p_{\mathrm{BP}}(H|\proj(H|y))) \doteq \hat\gamma.
\]
This is the idea of \cite{Efron:Tibshirani:1998:PR} to estimate $\hat\gamma$ for adjusting $p_{\mathrm{BP}}(H|y)$ by
$p_\mathrm{ET}(H|y) = \bar\Phi( \psi_1(H|y) - 2 z_\text{proj}(H|y) )  \doteq p(H|y)$.
It is easily extended to selective inference 
by $p_\text{ET-SI}(H|y) = p_\mathrm{ET}(H|y)/ \bar\Phi(- z_\text{proj}(H|y) ) \doteq p(H|H^c,y)$
for the case of $S=H^c$,
and an extension to general $S$ is given in Section~\ref{sec:iterated-bootstrap}.
We will show in Theorem~\ref{theorem:double-bootstrap}
that the double bootstrap \citep{Hall:1986:BCI,Efron:Tibshirani:1998:PR} also computes the adjusted $p$-value.
An advantage of multiscale bootstrap over these methods is that it does not require expensive computation of the null distribution.

%%%------------------------------------------------------------------------------------------------------------------------------
%%
%

%
%%
%%%------------------------------------------------------------------------------------------------------------------------------
\section{Numerical Results} \label{sec:numerical-results}

\subsection{An illustrative example of pvclust} \label{sec:pvclust-example}

\begin{figure}
\centering
% \begin{minipage}[t]{0.49\textwidth}
% \centering
% \includegraphics[width=\textwidth]{./fig_draft/20171024-lung3-pvalue2-best-k3-20171024-lung1-fit-wide.pdf}
% (a)~confidence level
% \end{minipage}
% \begin{minipage}[t]{0.49\textwidth}
% \centering
\includegraphics[width=0.7\textwidth]{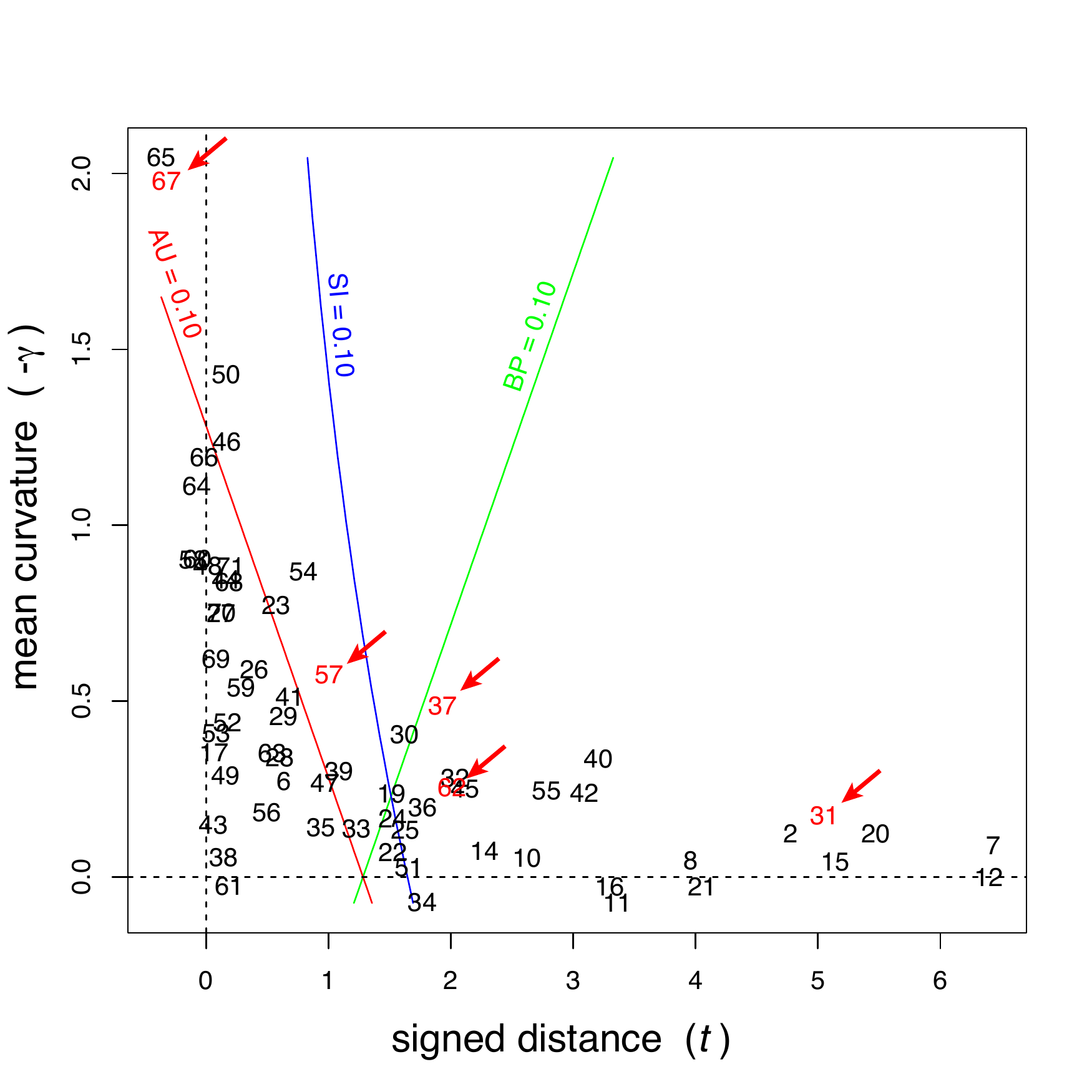}
%(b)~signed distance and mean curvature
%\end{minipage}
\caption{
%{\rm (a)}~Confidence level $1-p$ of the lung dataset for the 71 clusters.
%The number of clusters with $1-p>0.9$ is 44 for $p_\mathrm{AU}$ (A), which is larger than 33 for $p_\mathrm{SI}$ (S) and 34 for $p_\mathrm{BP}$ (B).
%$p_\mathrm{AU}$ and $p_\mathrm{SI}$ are connected by segments.
%{\rm (b)}~
Cluster id's are plotted at $(t,-\hat\gamma)$.
Contour lines are shown for $p_\mathrm{AU}$, $p_\mathrm{SI}$ and $p_\mathrm{BP}$ at $\alpha=0.10$.
Clusters with arrows are those in Fig.~\ref{fig:pvclust-tree}.
}
\label{fig:pvclust-pvalue}
\end{figure}

For each cluster in Fig.~\ref{fig:pvclust-tree}, we test the null hypothesis that the cluster is not true.
Therefore, clusters with $1-p>0.9$ are identified as true by rejecting their null hypotheses at significance level $\alpha=0.1$.
However, the decision depends on the type of $p$-values.
Which $p$-value should we use for making a decision?

Let us look at the branch of cluster id = 31 consisting of six non-tumor tissues
 (the left most box in Fig.~\ref{fig:pvclust-tree}).
All the other 67 tissues are lung tumors from patient.
For this cluster, it is very natural to use the non-selective $p$-value $p_\mathrm{AU}$ for controlling  (\ref{eq:goal-unselective}), because a scientist may hypothesize that the six tissues are different from the others before looking at the data.
For most of the other clusters, however, we should use the selective $p$-value $p_\mathrm{SI}$ for controlling (\ref{eq:goal}), because they are discovered only after looking at the data.
On the other hand, $p_\mathrm{BP}$ and $2p_\mathrm{BP}$ can be interpreted as naive estimates of $p_\mathrm{AU}$ and $p_\mathrm{SI}$, respectively, when $|\gamma|$ is small, but they are not quite good estimates here.

% Looking at Fig.~\ref{fig:pvclust-pvalue}~(a), long segments indicate large difference between $p_\mathrm{AU}$ and $p_\mathrm{SI}$.
% $p_\mathrm{AU}$ is smaller than $p_\mathrm{SI}$ and $p_\mathrm{BP}$, leading to more rejection of null hypotheses.
% This is confirmed in Fig.~\ref{fig:pvclust-pvalue}~(b), where the contour line of $p_\text{AU}$ is left to the other two contour lines, because null hypotheses are rejected on the right hand side of the contour lines.

Differences of $p$-values with respect to the geometric quantities are illustrated in Fig.~\ref{fig:pvclust-pvalue}.
We plotted $\psi_{0}(H|y) \doteq t$ for the $x$-axis and $\psi_{-1}(H|y) - \psi_{0}(H|y) \doteq -\hat\gamma$ for the $y$-axis.
On the $x$-axis ($\hat\gamma=0$), $p_\text{AU}=p_\text{BP}=\bar\Phi(t)$, $p_\text{SI}=2\bar\Phi(t)$,
and they are adjusted by $\hat\gamma$ as seen in the contour lines.
The contour line of $p_\text{AU}$ is left to the other two lines, indicating that $p_\text{AU}$ is smaller than $p_\text{BP}$ and $p_\text{SI}$, thus rejecting more null hypotheses.
Some clusters suggest estimation error, such as $t<0$ and $-\hat\gamma<0$, but the problem seems minor overall.

Computation of the $p$-values is examined in Fig.~\ref{fig:pvclust-sbfit}.
Looking at cluster id = 37, 57 and 62, fitting of poly.2, namely the linear model (\ref{eq:phi-linear}) or (\ref{eq:polyk}) with $k=2$,
is very good so that the extrapolation by substituting $\sigma^2=-1$ in the linear model is good enough, and the other sophisticated extrapolation methods may not be necessary.
This suggests the validity of the large sample theory of the second order asymptotic accuracy (Section~\ref{sec:motivation} and Section~\ref{sec:LST}).

However, the singular model sing.3, namely (\ref{eq:singk}) with $k=3$, fits much better for cluster id = 67.
Then the extrapolation to $\sigma^2=-1$ by the Taylor polynomial approximation $\varphi_{H,k}$ depends on $k$, giving
$1-p_\mathrm{AU}=0.77$ with $k=2$ (linear) and  $1-p_\mathrm{AU}=0.95$ with $k=3$ (quadratic).
The large $-\hat\gamma$ value for id = 67 indicates that the region $H^c$ is small; suggests small radius $r \doteq m/(-2\gamma)$ if it were a sphere in $\mathbb{R}^{m+1}$.
The last case is beyond the large sample theory of Section~\ref{sec:LST}, and it requests the need for the other asymptotic theory of Section~\ref{sec:NFT}.

\subsection{Simulation of convex and concave regions} \label{sec:convex-concave-simulation}

\subsubsection{Two dimensional examples}

%concave
\begin{table}
\caption{Concave hypothesis regions: Rejection probabilities in percent at significance level $\alpha=0.1$ 
and the average absolute bias defined by (\ref{eq:mean-absolute-bias})
with selection probabilities at the bottom (best two values in the sense of unbiasedness are in bold)}
\label{table:concave}
\scalebox{0.92}{
\begin{tabular}{l|rrrrrrrr|r}
\hline
Smooth            & $\theta=0.0$  & $0.5$   &  $1.0$  &   $1.5$ &   $2.0$ &   $2.5$ &   $3.0$ &   $3.5$ & Bias\\
\hline
BP          		& 13.32		& 13.66		& 14.57		& 23.57 		& 16.96		& 17.96		& 18.68 		& 19.15	& 6.26\\
AU ($k = 3$)     & 21.36  		& 21.44    		& 21.57    		& 21.50		& 21.18  		& 20.75		& 20.39  		& 20.17	& 10.90\\
\hdashline
2BP          		& 5.96    		&   6.15    		& 6.68    		&  7.39   		&   8.11    		& 8.73    		& 9.18    		&  9.47   		& 2.30\\
2AU ($k = 2$)     & $\bm{9.75}$ &$\bm{9.91}$  	& $\bm{10.26}$ & 10.60 		& 10.76 		& 10.72 		& 10.56 		& 10.38 		& 0.47\\
2AU ($k = 3$)     &$\bm{10.49}$ &$\bm{10.58}$ 	& 10.74 		& 10.79 		& 10.67 		& 10.43 		& $\bm{10.22}$&$\bm{10.08}$ & 0.53\\
SDBP 		& 8.70		& 8.87		& 9.29		& 9.76		& $\bm{10.11}$	& $\bm{10.28}$& 10.30		& 10.24		& 0.51\\
SI ($k = 2$)      & 8.87    		& 9.03    		& 9.44    		& $\bm{9.87}$ 	& 10.18		& 10.29 		& 10.27 		& 10.19 		& $\bm{0.43}$\\
SI ($k = 3$)      & 9.33    		& 9.45    		& $\bm{9.73}$ 	& $\bm{9.99}$ 	& $\bm{10.10}$ & $\bm{10.09}$ &$\bm{10.03}$ & $\bm{9.99}$ & $\bm{0.20}$\\
\hdashline
$P(Y\in S\mid\mu)$    & 44.54 &   45.00   & 46.09 &   47.28 &   48.24   & 48.89 &  49.29  & 49.53 & -\\
\hline \hline
Nonsmooth     & $\theta=0.0$  & $0.5$ & $1.0$ & $1.5$ & $2.0$ & $2.5$ & $3.0$ & $3.5$ & Bias\\
\hline
BP          		& $\bm{6.88}$	& $\bm{9.73}$	& 12.89		& 15.80 		& 17.93		& 19.16		& 19.72 		& 19.92	& 6.05\\
AU ($k = 3$)     & 17.14  		& 20.65  		& 22.51    		& 22.41		& 21.22 		& 20.15		& 19.72		& 19.73	& 10.56\\
\hdashline
2BP          		& 2.53   		 &   3.93    	&   5.62   		 & 7.3   		&  8.61   		& 9.41    		&   9.80    		& $\bm{9.94}$ 	& 2.73\\
2AU ($k = 2$)   & 5.86	 	& 8.11		& $\bm{10.11}$  & 11.3    		& 11.43 		& 11.01 		& 10.52 		& 10.21 		& 1.21\\
2AU ($k = 3$)   & $\bm{7.12}$  & $\bm{9.36}$	& 10.90 		& 11.3    		& 10.79 		& $\bm{10.16}$ & $\bm{9.86}$  & 9.84    		& $\bm{0.75}$\\
SDBP 		& 4.19		& 6.08		& 8.04		& 9.59		& $\bm{10.40}$	& 10.57		& 10.42		& 10.23		& 1.52\\
SI ($k = 2$)      & 4.94    		& 6.99    		& 8.96    		& $\bm{10.3}$ 	& 10.77		& 10.65 		& 10.37 		& 10.15 		& 1.24\\
SI ($k = 3$)      & 5.50   		& 7.55    		& $\bm{9.26}$	& $\bm{10.1}$ 	& $\bm{10.24}$ & $\bm{10.03}$& $\bm{9.89}$	& $\bm{9.89}$ 	& $\bm{0.85}$\\
\hdashline
$P(Y\in S\mid\mu)$    & 33.33 & 41.83 & 46.87 & 49.10 & 49.80 & 49.97 & 50.00 & 50.00 & -\\
\hline
\end{tabular}
}
\end{table}

\begin{figure}
 \begin{minipage}{0.49\textwidth}
  \begin{center}
   \includegraphics[width=0.9\textwidth]{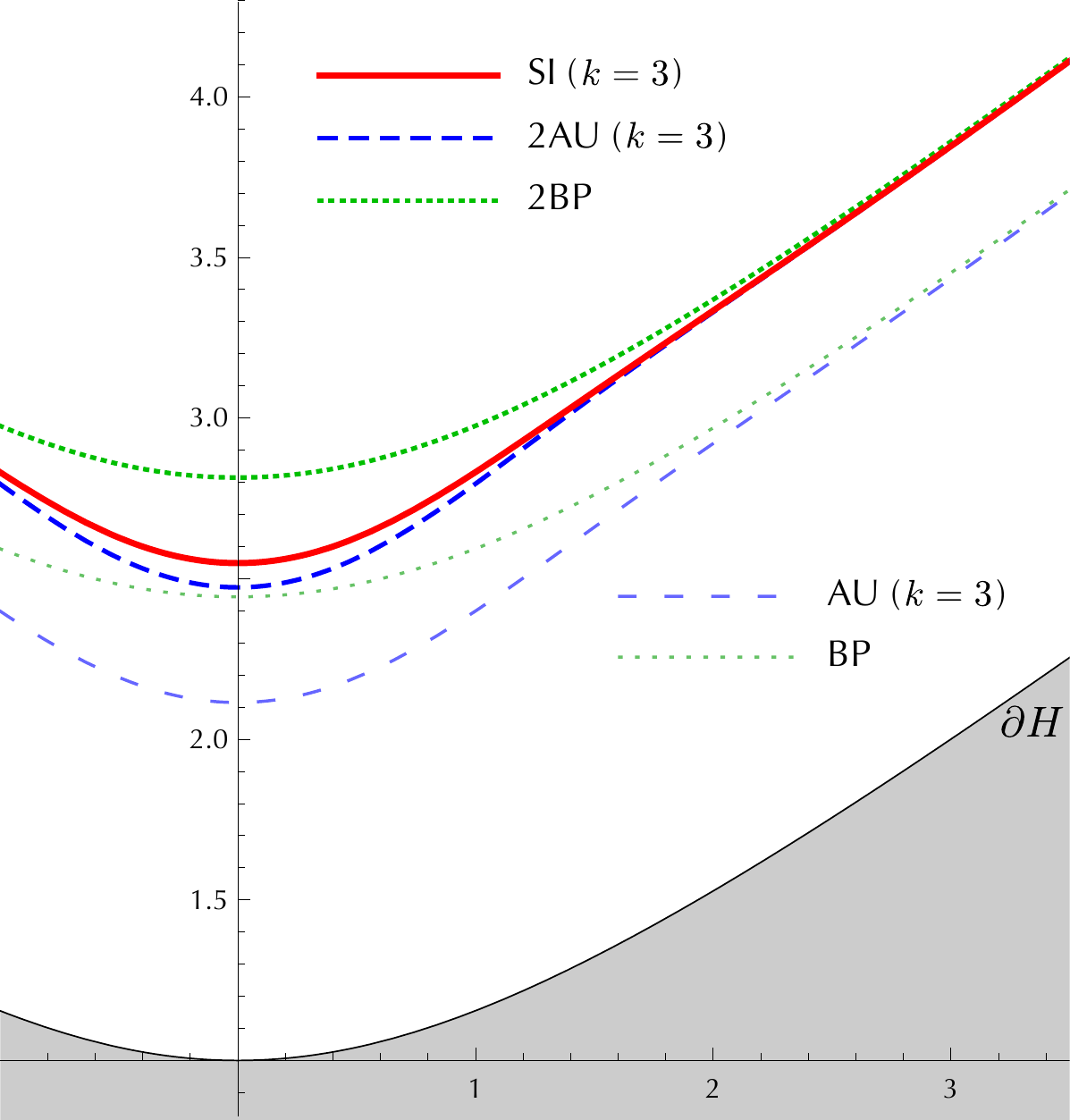}
   \hspace{1.6cm} (a) Smooth case : $a=1$
 \end{center}
 \end{minipage}
 \begin{minipage}{0.49\textwidth}
  \begin{center}
   \includegraphics[width=0.9\textwidth]{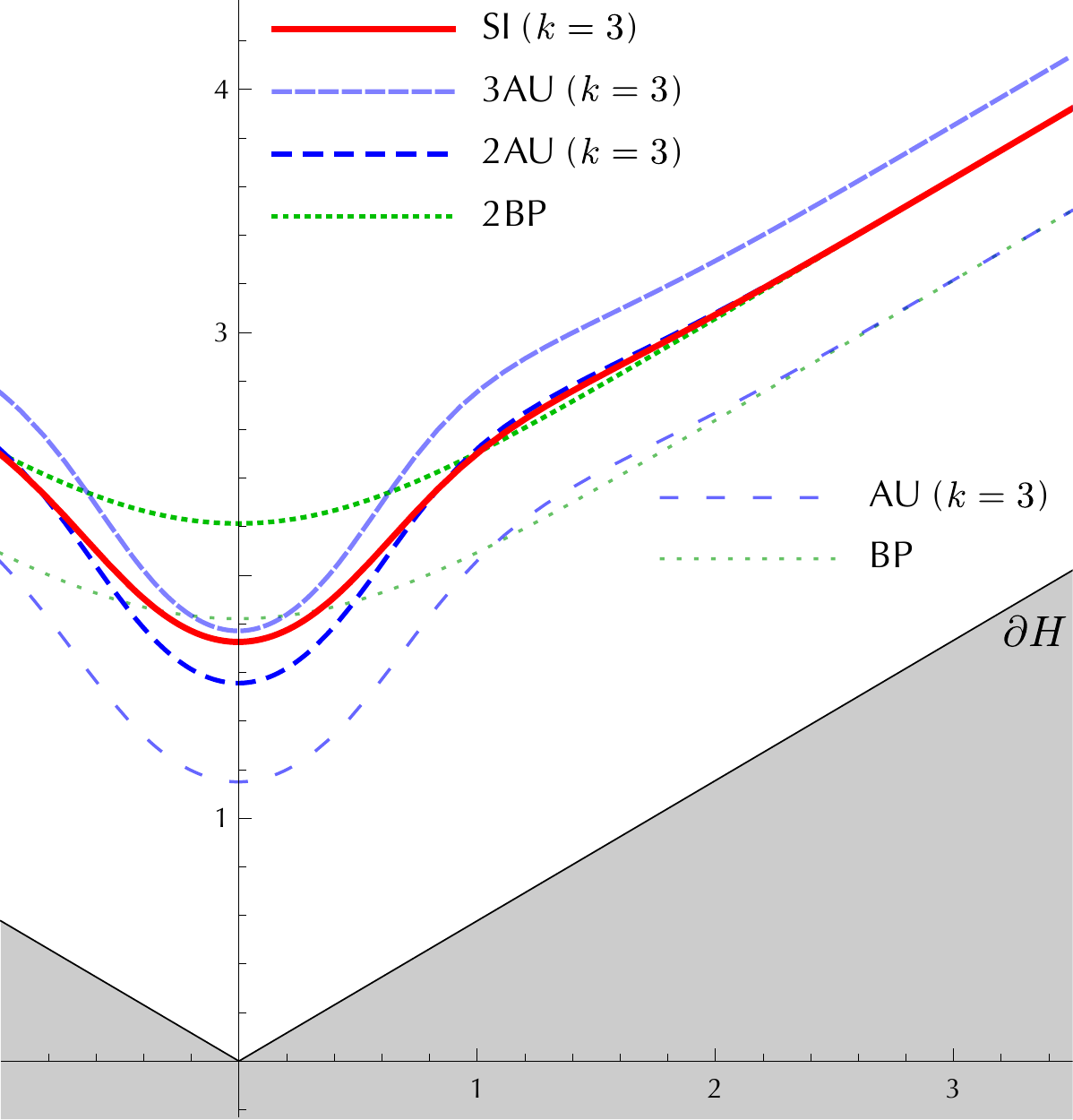}
   \hspace{1.6cm} (b) Nonsmooth case : $a=0$
 \end{center}
 \end{minipage}
   \caption{Concave hypothesis regions : Contour lines of $p$-values with $\alpha=0.1$.}
   % SI ($k=3$) : $p_{\mathrm{SI},3}(y)=\alpha$ (a solid line).
   % 3AU ($k=3$) : $3p_{\mathrm{AU},3}(y)=\alpha$  (a densely dashed line).
   % 2AU ($k=3$) : $2p_{\mathrm{AU},3}(y)=\alpha$ (a dashed line).
   % 2BP : $2p_\mathrm{BP}=\alpha$ (a dotted line).
   % AU ($k=3$) : $p_{\mathrm{AU},3}(y)=\alpha$ (a loosely dashed line).
   % BP ($k=3$) : $p_\mathrm{BP}=\alpha$ (a loosely dotted line).}
  \label{fig:sim_concave}
\end{figure}

\begin{figure}
 \begin{minipage}{0.49\textwidth}
  \begin{center}
   \includegraphics[width=0.9\textwidth]{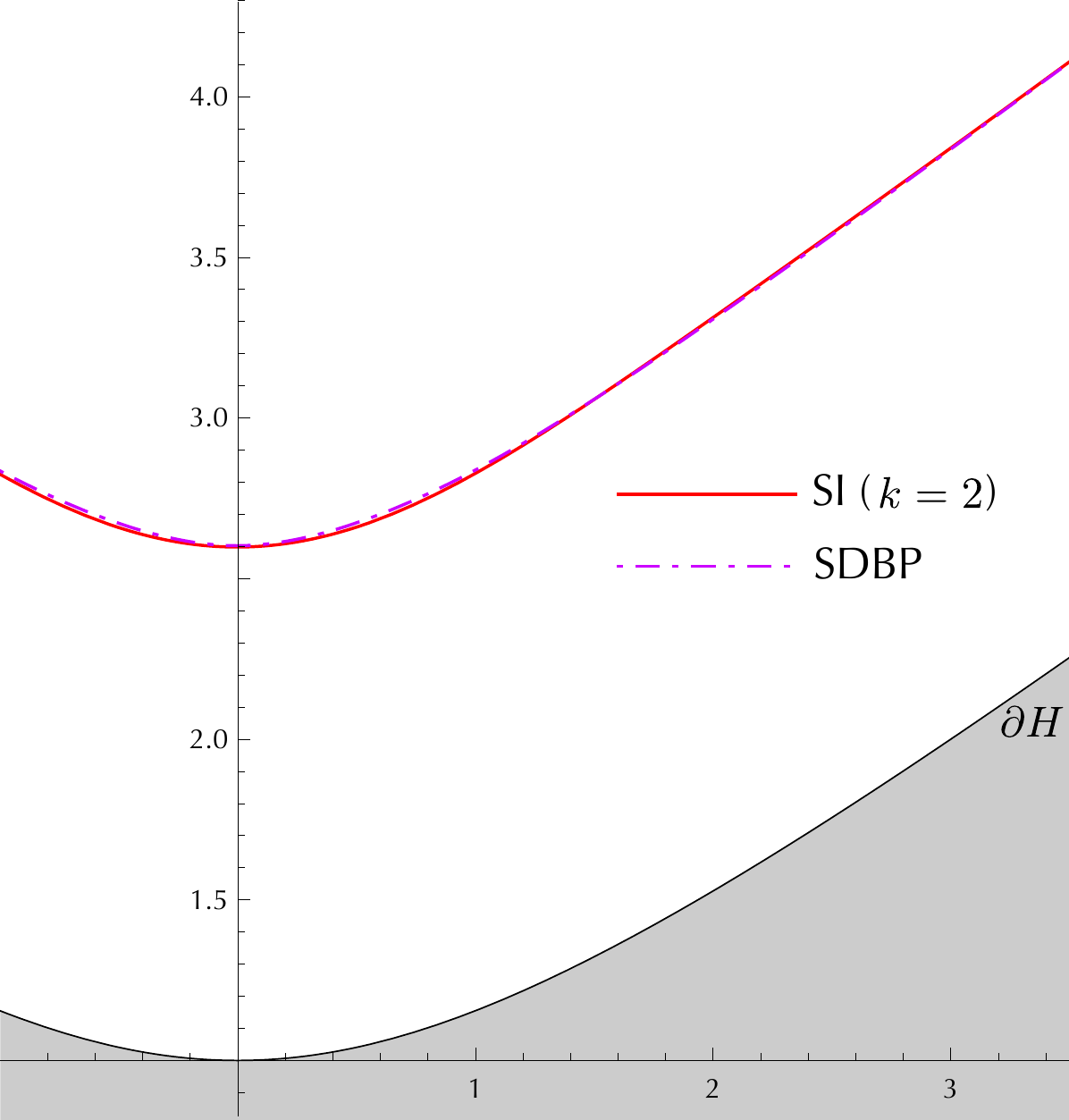}
   \hspace{1.6cm} (a) Smooth case : $a=1$
 \end{center}
 \end{minipage}
 \begin{minipage}{0.49\textwidth}
  \begin{center}
   \includegraphics[width=0.9\textwidth]{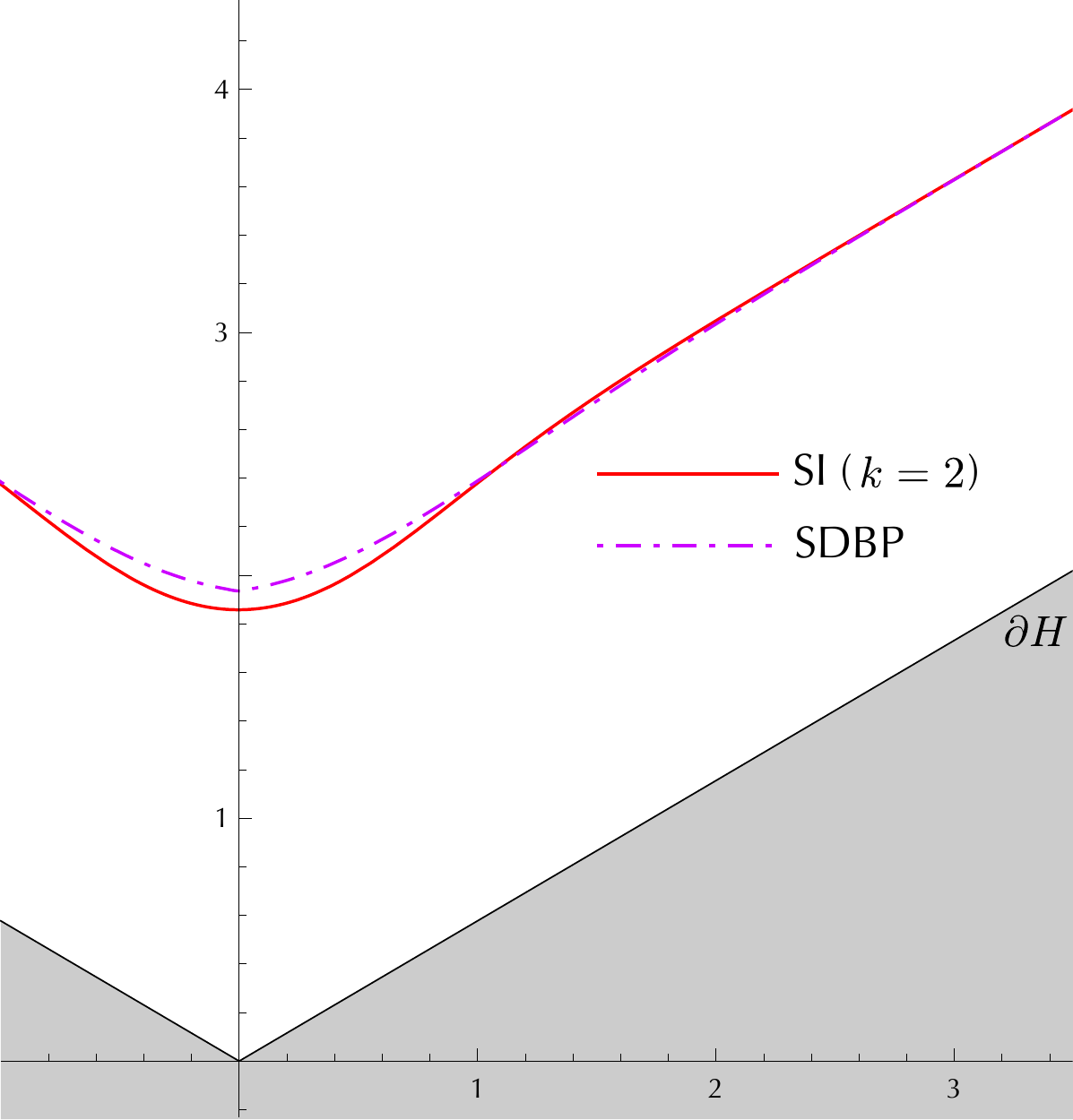}
   \hspace{1.6cm} (b) Nonsmooth case : $a=0$
 \end{center}
 \end{minipage}
   \caption{Multiscale bootstrap (SI) and double bootstrap (SDBP): Contour lines of $p$-values with $\alpha=0.1$.}
   % Relationship between SI ($k=2$) and SDBP : 
   %  SI ($k=2$) : $p_{\mathrm{SI},2}(y)=\alpha$ (a solid line).
   %  SDBP : $p_{\mathrm{BP},2}(y)=\alpha$  (a dashed-dotted line).}
  \label{fig:sim_concave_sdbp}
\end{figure}

Here, we verify that our method provides approximately unbiased selective inference through numerical simulations.
We consider the following hypothesis region in $\mathbb{R}^2$ as
\ba
h(u)= \pm\sqrt{a+u^2/3},\quad H=\{(u,v)\mid v \le -h(u)\},
\ea
and choose the selective region as $S=H^c$.
We consider four settings: the sign of function $h$ determines whether the hypothesis region is convex or concave, and
the value of $a$ determines the shape of the boundary surface $\partial H$ as smooth for $a=1$ or nonsmooth for $a=0$.
Here, we show only the results of concave cases and refer the reader to Section~\ref{sec:convex-concave-simulation-details} (supplementary material) for the convex cases.
As mentioned in Section~\ref{sec:motivation}, 
$2p(y)$ with non-selective $p$-values can be interpreted as naive selective $p$-values
if $\partial H$ is flat.
Thus, we compare our method with naive selective inferences 
using $2p(y)$ %, or equivalently setting significance level $\alpha'=\alpha/2$, 
for fair comparisons.
For the naive selective $p$-values, the rejection probability will be doubled if we use $p(y)$ instead.

We refer to the non-selective test with existing $p$-values $p_\mathrm{BP}$ and $p_{\mathrm{AU},k}$ as ``BP'' and ``AU ($k$)'', respectively.
We refer to the selective test with $2p_\mathrm{BP}$ as ``2BP'',
and that with $2p_{\mathrm{AU},k}$ as ``2AU ($k$)'',
where $p_{\mathrm{AU},k}$ \citep{Shimodaira08} is the non-selective version of $p_{\mathrm{SI},k}$.
The selective test with $p_{\mathrm{BP},2}$ in Section~\ref{sec:iterated-bootstrap} and $p_{\mathrm{SI},k}$ in (B) of Algorithm~\ref{alg:simbp} 
are denoted as ``SDBP'' and ``SI ($k$)'', respectively.
The bias of SI ($k$) is expected to reduce as $k$ increases.
By Theorem~\ref{theorem:double-bootstrap}, SI ($k=2$) and SDBP should exhibit the same behavior at least for smooth cases.
All results of the following simulations are computed accurately 
by numerical integration instead of Monte-Carlo simulation in order to avoid the effects of sampling error.

Table~\ref{table:concave} shows the selective rejection probabilities at significance level $\alpha=0.1$ 
and the selection probabilities for several $\mu=(\theta,-h(\theta))\in \partial H$, 
where we chose $\theta=0.0,0.5,\dots,3.5$.
The last column shows the average absolute bias for $p(y)$ computed by
\banum\label{eq:mean-absolute-bias}
\mathrm{Bias}(p)=\frac{1}{M+1} \sum_{j=0}^M \left| \frac{P(p(Y)<\alpha \mid \mu_j)}{P(Y\in S \mid \mu_j)}  - \alpha\right|,
\eanum
where $\theta_j$ ranges from 0 to 3.5 as $\theta_j = 0.05 \times j$, $j=0,\ldots,M$ with $M=70$.
%$\alpha' = \alpha$ for SI and $\alpha' = \alpha/2$ for BP and AU.
From the table,
we can see that the non-selective $p$-values induce serious bias in the context of selective inference.
Moreover, our method dominates the naive selective inferences in the sense of the unbiasedness in many cases,
and $k=3$ is better than $k=2$ for SI($k$).
In fact, for almost all points of $\theta$, selective rejection probabilities of our method are closer to significance level $\alpha=10\%$ than the naive inferences.
The average absolute bias (\ref{eq:mean-absolute-bias}) for our method is smaller than those for the naive selective inferences.
In addition, for each case, SI ($k=2$) and SDBP provide similar selective rejection probabilities in Table~\ref{table:concave} 
and similar rejection boundaries in Fig.~\ref{fig:sim_concave_sdbp}.
In accordance with Theorem~\ref{theorem:double-bootstrap}, the rejection surfaces of SI ($k=2$) and SDBP are almost the same in the smooth case.

For the concave hypothesis region with the nonsmooth boundary surface, 
Table~\ref{table:concave} shows that it is difficult to provide unbiased inferences in a neighborhood of the vertex.
Nevertheless, by using our method, the bias can be reduced more effectively with distance from the vertex.
%%%
%Here, it is worth noting that, for a concave hypothesis region with a nonsmooth surface, 
%a rejection probability of AU test is prone to take a small values in a neighborhood of the vertex 
%whereas the selection probability in such neighborhood is much less than 0.5.
%Thus, the underestimation behavior of the AU test is canceled by the overestimation effect due to small selective %probabilities in a neighborhood of the vertex.
%This superficial unbiasedness of the naive selective inference using $2p_{\mathrm{AU},k}(y)$ will disappear with increasing distance from the vertex.

Next, we look at contour lines of $p$-values in Fig.~\ref{fig:sim_concave} with the horizontal axis $\theta$.
We chose $\alpha=0.1$ again.
The shaded area represents the hypothesis region $H$, and the rejection regions are just above the contour lines in the subfigures.
For $p_{\mathrm{AU},k}$ and $p_{\mathrm{SI},k}$, we fixed $k=3$.
In all the settings, the three curves of 
$2p_{\mathrm{BP}}=\alpha$, $2p_{\mathrm{AU},3}=\alpha$ and $p_{\mathrm{SI},3}=\alpha$ coincide with each other at large $\theta$ values where $\partial H$ is flat.
This verifies that the use of $2p(y)$ of non-selective $p$-values leads to selective inference there.
Looking at $P(Y\in S|\mu)$ in Table~\ref{table:concave} at large $\theta$ values,
we confirm that the selection probabilities are actually 1/2.
However, the selection probability decreases as $\theta$ approaches zero in the concave cases. 
It is $1/3$ at the vertex for the nonsmooth concave case.
Then the curve of  $3p_{\mathrm{AU},3}=\alpha$ in Fig.~\ref{fig:sim_concave} almost touches the curve of $p_{\mathrm{SI},3}=\alpha$ near the vertex.
This shows that our selective inference method automatically adjusts the selection probability to provide a valid selective inference.

\subsubsection{Spherical examples}

Here, we consider a simple example which is considered in Example~1 of \cite{Efron:Tibshirani:1998:PR}.
Suppose that $H=\{\mu \mid \|\mu\|\ge \theta\}$ and $S=H^c$ as a concave hypothesis region.
That is, we consider the case that the selective region is a sphere of radius $\theta$ in $\Rb^{m+1}$.
The mean curvature of $\partial H$ defined in Section~\ref{sec:scaling-law} is given by $-\gamma=m/(2\theta)$.
For the fixed mean curvature, the number of dimensions $m+1$ was varied from $10$ to $1000$.
We chose $-\gamma=0.5,1.0,1.5$.
In this setting, by Theorem~1 of \cite{Shimodaira2014higherorder}, 
the third order term in the asymptotic expansion of the bootstrap probability goes to $0$ as $m\rightarrow \infty$.
Thus, from Theorem~\ref{theorem:au-pvalue}, 
we expect that the selective rejection probability of SI ($k$) goes to $\alpha$ as $m\rightarrow \infty$.
In addition, from the discussion in Section~\ref{sec:pvalue-si-LST},
it is expected that the selective rejection probabilities of 2BP and 2AU ($k$) go to $\Phi(\bar\Phi^{-1}(\alpha) - 2 \gamma)/\bar\Phi(- \gamma)$
and $\alpha/\bar\Phi(- \gamma)$, respectively.

Fig.~\ref{fig:sim_sphere_concave} illustrates
the change of the selective rejection probability for each $p$-value as the number of dimensions increases.
We can see that 2BP and 2AU ($k$) have serious bias related to the magnitude of mean curvature.
On the other hand, the selective rejection probabilities of SI ($k$) approach to $\alpha=10\%$, 
regardless of the mean curvature, as the number of dimensions increases.
Thus, when the number of dimensions is relatively large, 
our selective inference could be nearly unbiased whereas the naive selective inference using 2BP or 2AU may have serious bias.

\begin{figure}
 \begin{minipage}{0.32\textwidth}
  \centering
   \includegraphics[width=\textwidth]{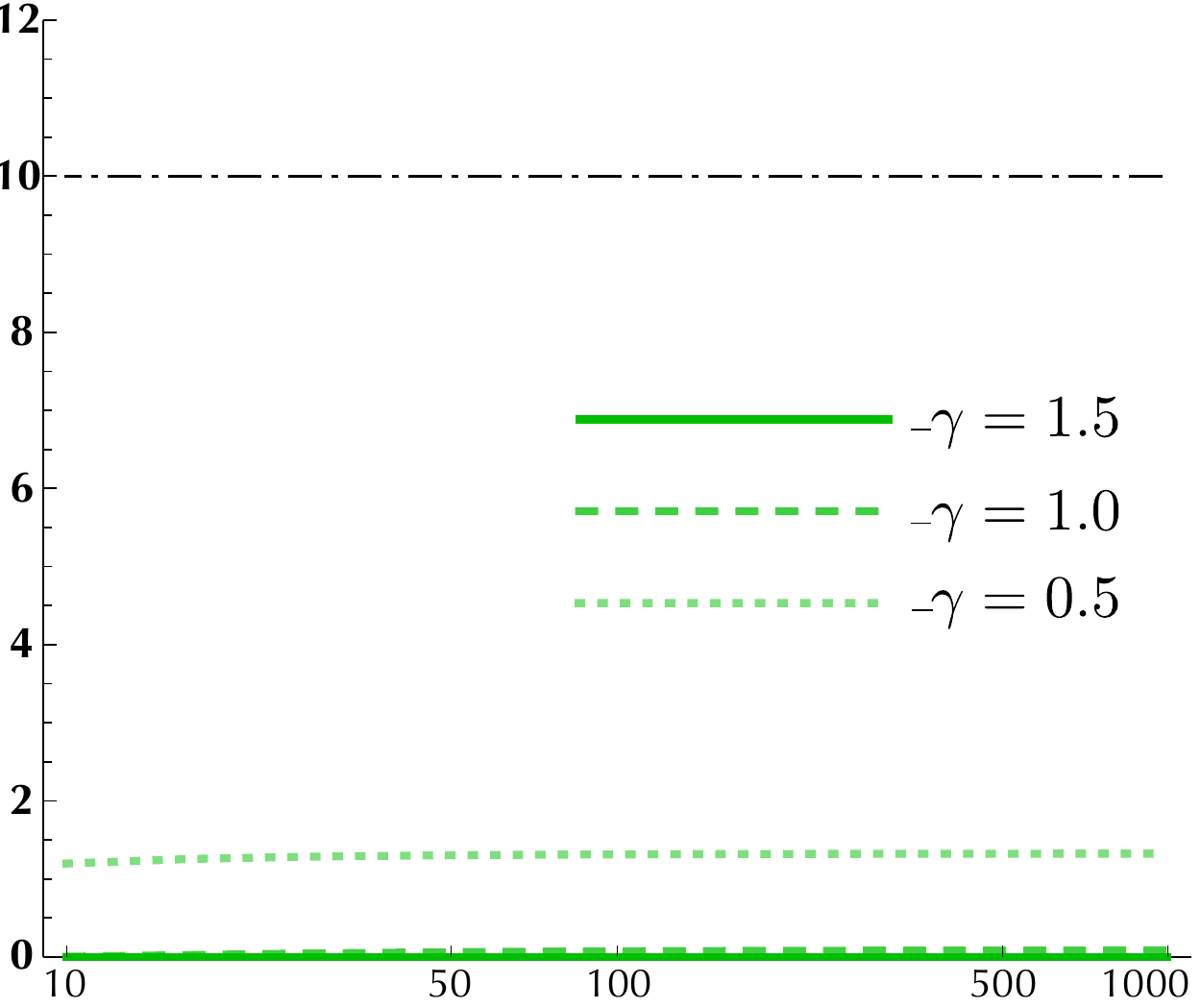}
   \subcaption{2BP}
 \end{minipage}
 \begin{minipage}{0.32\textwidth}
  \centering
   \includegraphics[width=\textwidth]{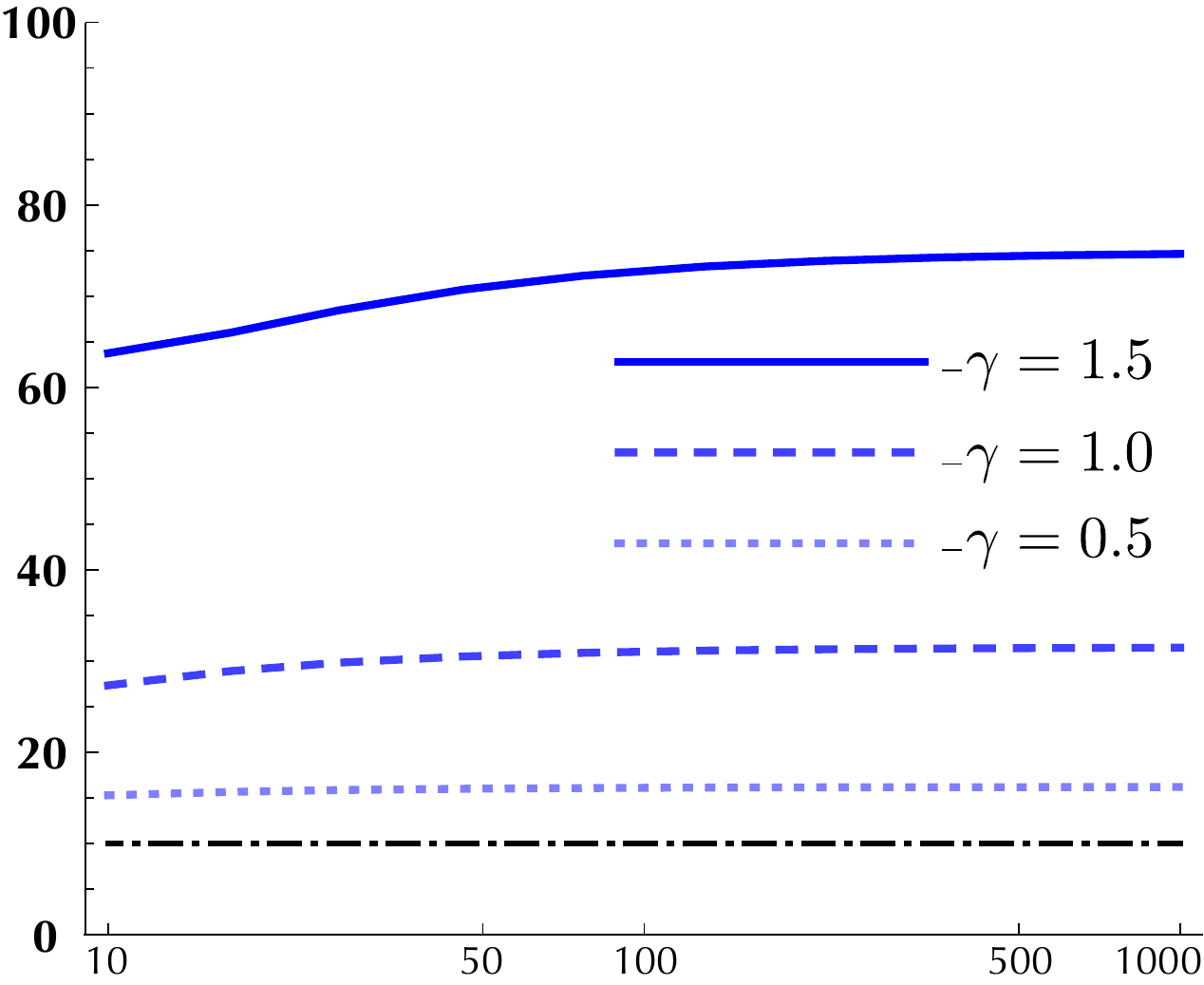}
   \subcaption{2AU ($k=3$)}
 \end{minipage}
  \begin{minipage}{0.32\textwidth}
  \centering
   \includegraphics[width=\textwidth]{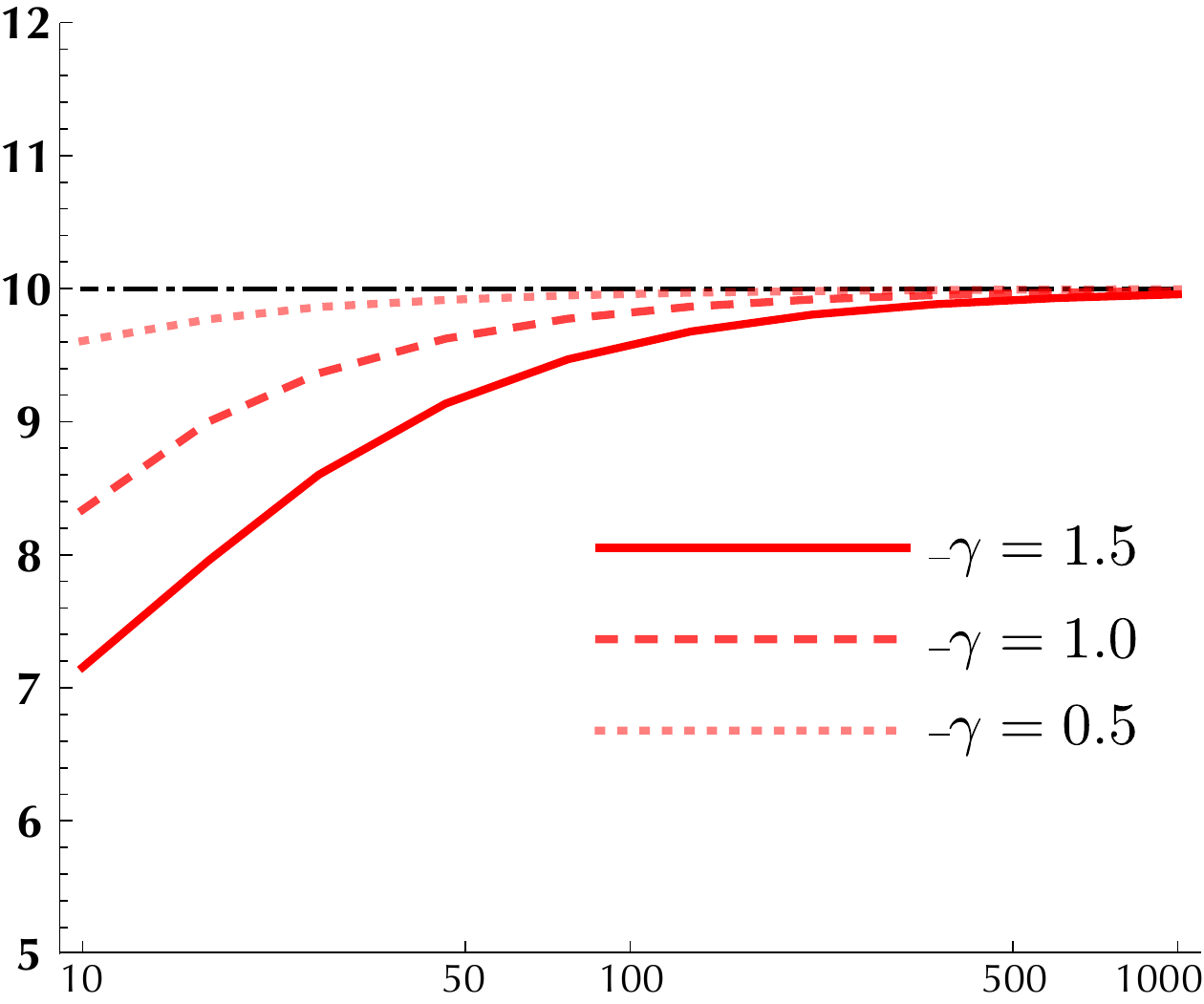}
   \subcaption{SI ($k=3$)}
    \end{minipage}
   \caption{Concave hypothesis regions : selective rejection probabilities as a function of the number of dimensions $m+1$. 
   The horizontal axis is the number of dimensions $m+1$ and the vertical axis is 
   the selective rejection probability in percent at significance level $\alpha=0.1$.}
   % ($-\gamma=0.5$ : a dotted line, $-\gamma=1.0$ : a broken line, $-\gamma=1.5$ : a solid line).
   % The dashed-dotted lines is the line of the ideal unbiased selective test of which selective rejection probability is always equal to the significance level $\alpha=0.1$.
   %}
  \label{fig:sim_sphere_concave}
\end{figure}

\subsection{Simulation analysis of pvclust} \label{sec:pvclust-simulation}

Here, we provide a numerical simulation of pvclust in accordance with Section~\ref{sec:pvclust-sampling} (supplementary material).
The average linkage hierarchical clustering with the (normalized) Euclidean distance $\frac{1}{n}\sum_{t=1}^n (x_{ti} - x_{tj})^2$ 
is considered as a tree building algorithm ${\tt hclust}$.
We denote by $\{i,j\}$ the cluster consisting of tissues $i$ and $j$.
We consider the simple setting in which
there are three tissues and $x_i\;(i=1,\dots,1000)$ are independent observations from the normal mixture 
\[
0.5N_3(\mu_1(a),I_3)+0.5(\mu_2(a),I_3),
\]
where $a\in \mathbb{R}$, $\mu_1(a)=(a,a,0)$, and $\mu_2(a)=(a,0,a)$.
In this setting, the ``true" distance matrix $d_\infty$ is given by 
\[
d_{\infty}=
\begin{bmatrix}
\text{-} &  a^2/2 + 2 & a^2/2 + 2 \\
\text{-} & \text{-} & a^2 + 2	\\
\text{-} & \text{-} & \text{-}
\end{bmatrix}.
\]
Thus, both clusters $\{1,2\}$ and $\{1,3\}$ are true. It is worth noting that the cluster $\{2,3\}$ is also true in the case that $a=0$.
We generated independently $10^4$ datasets of $\mathcal{X}_{1000}$ at each value of $a$ and applied pvclust for each dataset.
We choose values of $a$ such that the selection probabilities of the cluster $\{1,2\}$ (or $\{1,3\}$) at $a$ are 
almost equivalent to ones at $\theta=0,0.5,\dots,3.5$ in the nonsmooth case of Table~\ref{table:concave}.
Let $\hat{\theta}(a)$ be an estimated such transform from $a$ to $\theta$.
Obviously,  $a=0$ is corresponding to $\theta=0$.

For the details about the construction of $\hat{\theta}$, see Section~\ref{sec:pvclust-simulation-details} (supplementary material).
When we obtained the cluster $\{1,2\}$ (or $\{1,3\}$), 
we computed the following $p$-values for the null hypothesis that the cluster $\{1,2\}$ (or $\{1,3\}$) is not true: $p_{\mathrm{BP}}(H|y)$, $p_{\mathrm{AU},3}(H|y)$, 
$2p_{\mathrm{BP}}(H|y)$, $2p_{\mathrm{AU},3}(H|y)$, and $p_{\mathrm{SI},3}(H|S,y)$ by $B=10^4$ bootstrap replicates.
Note that $S$ is the selective region in which the cluster $\{1,2\}$ (or $\{1,3\}$) is true.
For $i=2,3$, let $H_i$ be the null hypothesis that the cluster $\{1,i\}$ is not true.
In pvclust, we consider the test for the null hypothesis $H_i$ only when the cluster $\{1,i\}$ appeared.
We also refer to tests at the significance level of $\alpha=0.1$ with these $p$-values as the same symbols in Section~\ref{sec:convex-concave-simulation}.
We count how many times, say $N_i$, the cluster $\{1,i\}$ appears in $10^4$ replications.
For each test, we also count how many times, say $R_i$, the null hypothesis $H_i$ is rejected and 
the selective rejection probability is estimated by $R_i/N_i$.
This scenario seems to correspond with the nonsmooth and concave setting in Section~\ref{sec:convex-concave-simulation} (see, Section~\ref{sec:pvclust-simulation-details}).
The subfigure (a) of Fig.~\ref{fig:sim_pvclust} shows the selective rejection probabilities of the nonsmooth case of Table~\ref{table:concave} related with $\theta$.
The subfigures (b,c) of Fig.~\ref{fig:sim_pvclust} show the selective rejection probabilities of tests against $H_2$ and $H_3$ related with $\hat{\theta}(a)$, respectively.
The shaded area around each line in the subfigures (b,c) of Fig.~\ref{fig:sim_pvclust} indicates the precision of plus minus two standard deviations.
From these results, we can see similar behaviors to the two-dimensional example in the simulation results of pvclust.
By using our method, the bias can be reduced more effectively in the practical situation of pvclust.

\begin{figure}
 \begin{minipage}{0.32\textwidth}
  \centering
   \includegraphics[width=\textwidth]{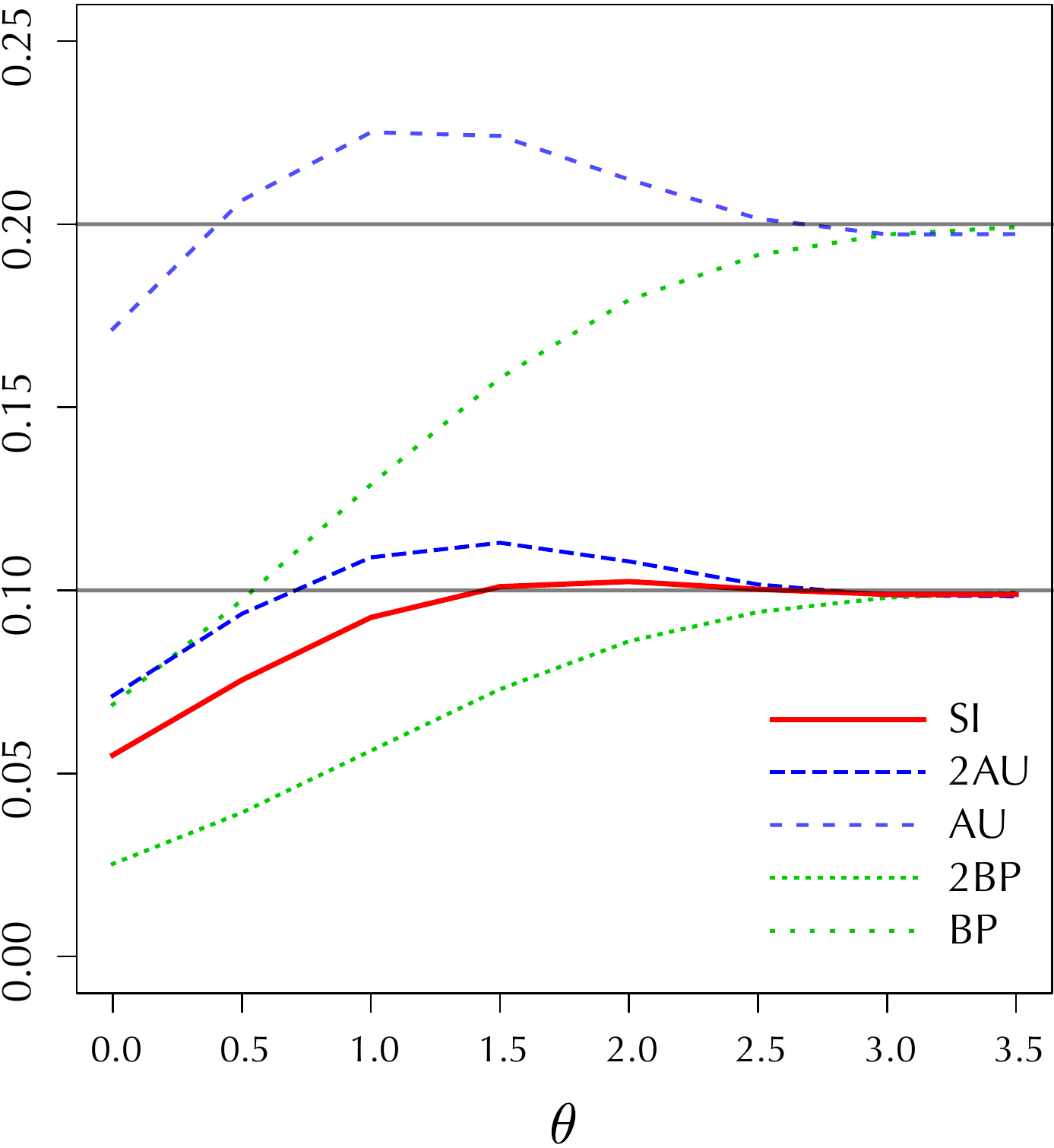}
   (a)%\subcaption{2-dimensional example}
 \end{minipage}
 \begin{minipage}{0.32\textwidth}
  \centering
   \includegraphics[width=\textwidth]{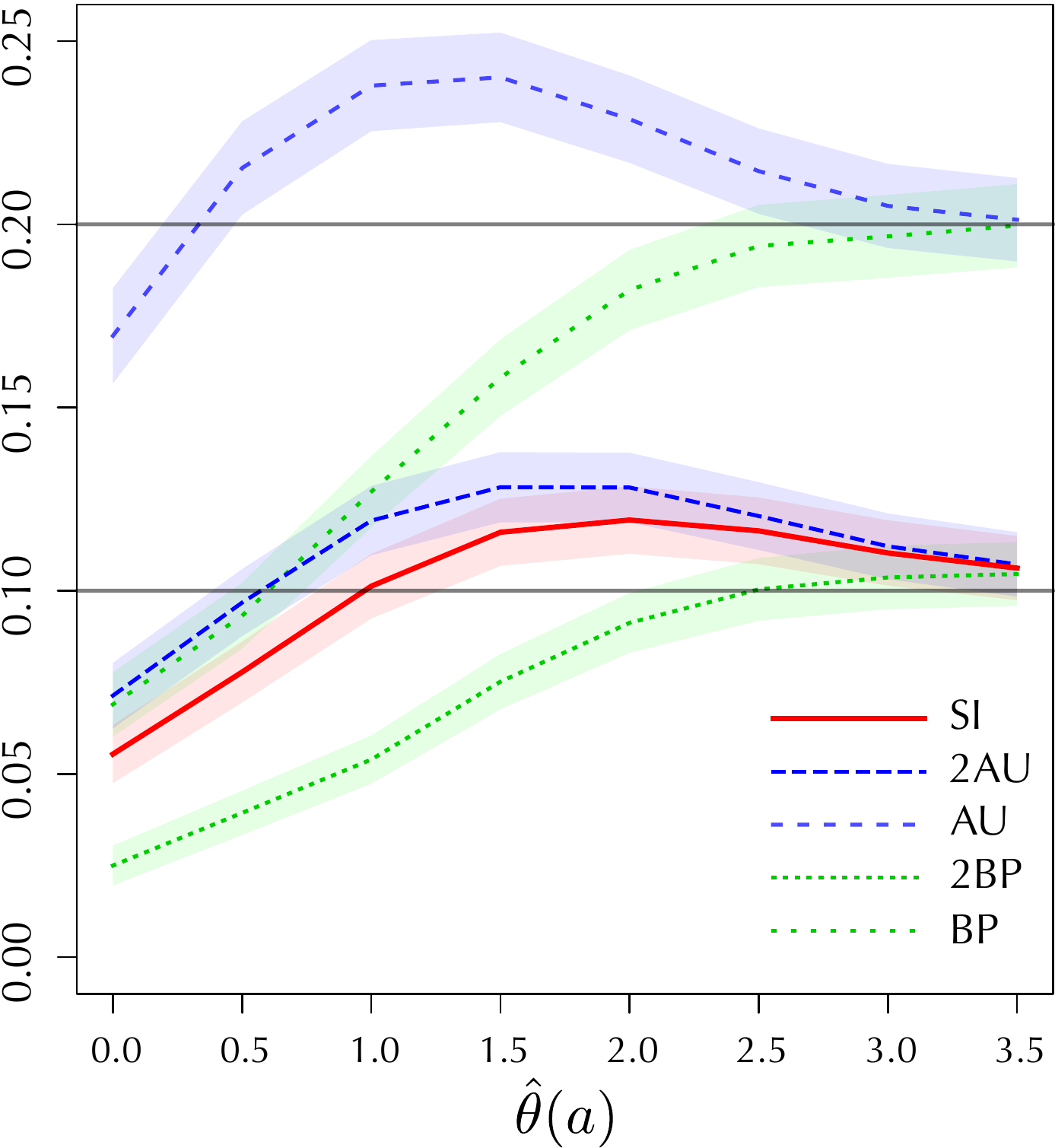}
   (b)%\subcaption{Cluster $\{1,2\}$}
 \end{minipage}
  \begin{minipage}{0.32\textwidth}
  \centering
   \includegraphics[width=\textwidth]{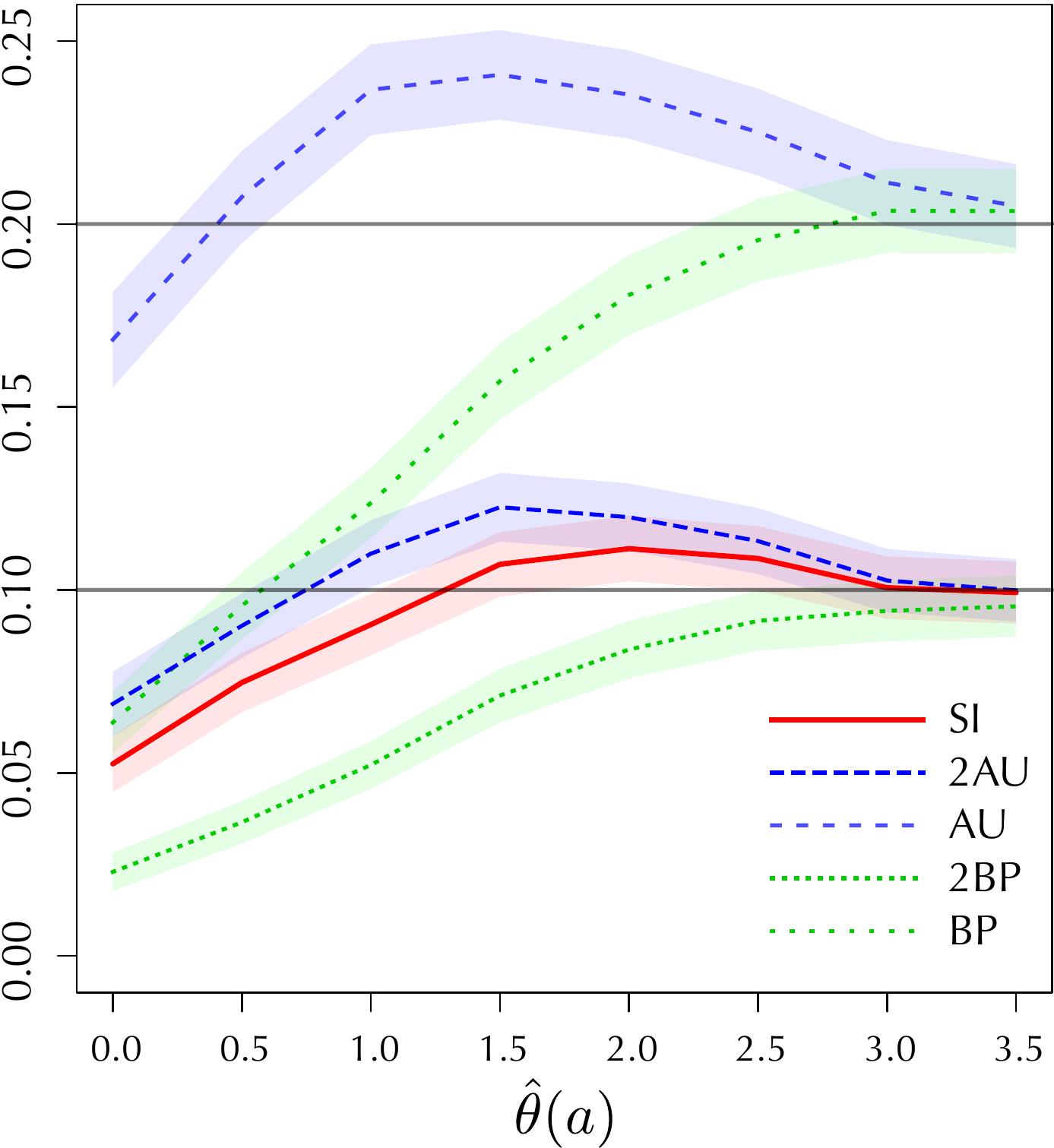}
   (c)%\subcaption{Cluster $\{1,3\}$}
    \end{minipage}
   \caption{Selective rejection probabilities as a function of $\theta$ or $\hat{\theta}(a)$.
   (a) The null hypothesis is the concave region with the nonsmooth boundary in Section~\ref{sec:convex-concave-simulation}.
   (b) The null hypothesis is $H_2$.
   (c) The null hypothesis is $H_3$.
   The horizontal axis of (a) is $\theta$ described in Section~\ref{sec:convex-concave-simulation}.
   The horizontal axes of (b), (c) are $\hat{\theta}(a)$.
   In each subfigure, the vertical axis is the selective rejection probability at significance level $\alpha=0.1$.}
   % SI ($k=3$) : $p_{\mathrm{SI},3}(y)=\alpha$ (a solid line).
   % 2AU ($k=3$) : $2p_{\mathrm{AU},3}(y)=\alpha$ (a dashed line).
   % 2BP : $2p_\mathrm{BP}(y)=\alpha$ (a dotted line).
   % AU ($k=3$) : $p_{\mathrm{AU},3}(y)=\alpha$ (a loosely dashed line).
   % BP ($k=3$) : $p_\mathrm{BP}=\alpha$ (a loosely dotted line)..
   % }
  \label{fig:sim_pvclust}
\end{figure}
%%%------------------------------------------------------------------------------------------------------------------------------
%%
%

%
%%
%%%------------------------------------------------------------------------------------------------------------------------------
\section{Large sample theory for smooth boundary surfaces}\label{sec:LST}

\subsection{Nearly parallel surfaces} \label{sec:nearly-parallel}

In this section, we consider the ordinary asymptotic theory of large $n$ by assuming that $\partial H$ and $\partial S$ are smooth surfaces.
We follow the geometric argument given for the problem of regions with multivariate normal model 
\citep{Efron:1985:BCI,Efron:Tibshirani:1998:PR},
and its extension to multiscale bootstrap \citep{Shimodaira:2004:AUT,Shimodaira2014higherorder}.
Here we introduce a new assumption for solving the selective inference.

For representing $y,\mu \in \Rb^{m+1}$ in a neighborhood of $\partial H$, we employ the coordinate system $(u,v)$
with $u\in\mathbb{R}^m, v\in\mathbb{R}$, 
and consider the hypothesis region $H=\Rc(h)$ for a smooth function $h(u)$ represented as
\[
	\Rc(h) = \{(u, v) \mid v \le -h(u), u\in \mathbb{R}^m \}.
\]
The selective region is defined similarly as $S=\Rc(s)^c$ for a smooth function $s(u)$.
Using the summation convention that an index appearing twice in a term implies summation over $1,\ldots,m$,
a smooth function $h$ is expressed as
\[
	h(u)  = h_0 + h_i u_i + h_{ij} u_i u_j + h_{ijk} u_i u_j u_k + \cdots,
\]
where $h_0= h(0), h_i=\partial h/\partial u_i|_0, h_{ij} = (1/2) \partial^2 h/\partial u_i \partial u_j |_0, \ldots$,  are the coefficients of the Taylor expansion at $u=0$.

In the large sample theory, each axis of $y=(u,v)$ is scaled by $\sqrt{n}$ as $Y = \sqrt{n} \bar X$ to keep the variance in (\ref{eq:model}) fixed.
Therefore the $k$-th derivatives of $h$ should be of order $O(n^{1/2}/ (n^{1/2})^k ) = O(n^{-(k-1)/2})$, giving
$h_{ij}=O(n^{-1/2})$, and higher order terms are $O(n^{-1})$. Thus we have
\[
	h(u) \doteq h_0 + h_i u_i + h_{ij} u_i u_j.
\]
%In the asymptotic theory, $p$-values are said to be $k$-th order accurate if it is correct up to terms of order $O(n^{-(k-1)/2})$ with error of order $O(n^{-k/2})$.
%Our computation here is second order accurate by computing up to $O(n^{-1/2})$ and ignore $O(n^{-1})$.
%The equality $\doteq$ is used again as in Section~\ref{sec:overview}.
In this paper, we consider a class of \emph{nearly parallel surfaces} $\Tc$ with the additional property that 
$h_0 = O(1)$ and $h_i=O(n^{-1/2})$ for all $h\in \Tc$ so that surfaces
defined by $\Bc(h) =  \{(u, v) \mid v = -h(u), u\in \mathbb{R}^m \} $ are nearly parallel to each other.
We then assume that $h,s \in \Tc$ for $\partial H=\Bc(h)$ and $\partial S=\Bc(s)$.
This setting is less restrictive than the class $h \in \Sc$ with $h_i=O(n^{-1})$
considered in \cite{Shimodaira2014higherorder} for representing non-selective rejection region $R=\Rc(r)^c$ with $r\in \Sc$.

Our motivation for introducing the class $\Tc$ is as follows.
Although we can set $h_0=0$, $h_i=0$, thus $h \in \Tc$, for any particular $h$ without losing generality by taking a point on surface as the origin, 
the $k$-th derivatives are $O(n^{-(k-1)/2})$ for another $s$ in general and thus $s_0 = O(\sqrt{n})$ and $s_i = O(1)$.
The first assumption for $s\in \Tc$, namely $s_0=O(1)$, comes from the \emph{local alternatives} setting,
where points, before applying the scaling of $\sqrt{n}$, are approaching zero at rate $O(1/\sqrt{n})$ so that
points in the space of $\mu$ are of order $O(\sqrt{n}/\!\sqrt{n}) = O(1)$.
The second assumption, namely $s_i = O(n^{-1/2})$, is newly introduced in this paper for solving the selective inference. It lets
$\partial s/\partial u \doteq \partial h/\partial u$ at $u=O(1)$ so that $\partial S$ is nearly parallel to $\partial H$ in the neighborhood of the origin. This assumption clearly holds for the case $S=H^c$, where $\partial S = \partial H$.

\subsection{The scaling law of the normalized bootstrap $z$-value} \label{sec:scaling-law}

The asymptotic expansion of the bootstrap probability \citep{Efron:Tibshirani:1998:PR} and its extension to multiscale bootstrap \citep{Shimodaira:2004:AUT} are obtained as follows.
By taking the origin at a point on $\partial H$, we can write
\[
	h(u) \doteq h_{ij} u_i u_j 
\]
with $h_0=0$, $h_i=0$.
We first work on the case that the observation is $y=(0,t)$, namely, $u=0\in\mathbb{R}^m$, $v=t\in\mathbb{R}$.
Then the signed distance from $y$ to $\partial H$ is $t$.
The mean curvature of $\partial H$ at the origin $(0,0)$ is defined as
\[
  \hat\gamma = h_{ii} = \sum_{i=1}^m h_{ii},
\]
which is half the trace of Hessian matrix.
The bootstrap probability $\alpha_{\sigma^2}(H | y)$ is, by noting $V^* |y \sim N(t,\sigma^2)$,
\[
P_{\sigma^2}(Y^\ast\in H\mid y) 
= P_{\sigma^2}(V^\ast \le -h(U^\ast)\mid y)
=E_{\sigma^2}\biggl[ \bar\Phi\biggl(\frac{t + h(U^\ast)}{\sigma}\biggr)\,\biggm|\, u=0 \biggr].
\]
We use the notation $E_{\sigma^2}(\cdot|y)$ for the expectation with respect to (\ref{eq:bpmodel}),
and $E_{\sigma^2}(\cdot | u)$, in particular, for the expectation with respect to 
$U^\ast\sim N_m(u,\sigma^2I_m)$.
We also interpret $E_{\sigma^2}$ as a operator to $h$, and use
the notation
\[
  E_{\sigma^2} h(u) = E_{\sigma^2}(h(U^*)\mid u).
\]
For calculating  $\alpha_{\sigma^2}(H | y)$, consider the Taylor expansion
\begin{equation} \label{eq:phi-taylor}
  \bar\Phi(x + \epsilon)=\bar\Phi(x) - \phi(x) \epsilon + O(\epsilon^2),
\end{equation}
and put $x=(t+E_{\sigma^2}h(0))/\sigma$, $\epsilon = (h(U^\ast)-E_{\sigma^2}h(0))/\sigma$.
$\phi(x)$ is the density function of $N(0,1)$.
Then we have $\alpha_{\sigma^2}(H | y)=E_{\sigma^2} (\bar\Phi(x+\epsilon)\mid u=0 ) = \bar\Phi(x) - \phi(x) E_{\sigma^2}(\epsilon|u=0) + O(n^{-1})$.
Since $E_{\sigma^2}(\epsilon\mid u=0)=0$ and $E_{\sigma^2}(h_{ij} U^*_i U^*_j\mid u=0)= h_{ij} \delta_{ij} \sigma^2 = h_{ii} \sigma^2$, we finally get
\begin{equation} \label{eq:bp-at-origin}
  \alpha_{\sigma^2}(H | y)
  \doteq \bar\Phi\biggl(\frac{t+E_{\sigma^2}h(0)}{\sigma}\biggr)
  \doteq \bar\Phi(t\sigma^{-1}+h_{ii}\sigma)
  = \bar\Phi(t\sigma^{-1}+\hat\gamma\sigma).
\end{equation}
Therefore, the normalized bootstrap $z$-value defined in (\ref{eq:psi-H}) is expressed as
\begin{equation} \label{eq:psi-at-origin}
  \psi_{\sigma^2}(H | y) \doteq t + \hat\gamma \sigma^2.
\end{equation}

Next, we work on the general case for any $h\in \Tc$ and $y\in\mathbb{R}^{m+1}$.
The expression for $\alpha_{\sigma^2}(H | y)$ is obtained by change of coordinates
with $\proj(H|y)$ being at the origin.
This has been done in \cite{Shimodaira2014higherorder} up to $O(n^{-2})$, meaning fourth order accuracy.
Here we need only the result with second order accuracy as shown in the following lemma.

\begin{lemma} \label{lemma:geometry-of-bp}
Let $H=\Rc(h)$ and $y=(\theta, -s(\theta))$ for any $h, s\in \Tc$ and $\theta\in \mathbb{R}^m$.
Then the bootstrap probability is expressed as
\begin{equation} \label{eq:bpsy}
	\alpha_{\sigma^2}(H | y) \doteq \bar\Phi(\hat\eta\sigma^{-1} + \hat\gamma \sigma),
\end{equation}
where $\hat \eta=\eta(H|y)$ is the signed distance from $y$ to $\partial H$ and $\hat\gamma=\gamma(H|y)$ is the mean curvature of $\partial H$ at $\proj(H|y)$.
These two geometric quantities are expressed by indicating the dependency on $h,s,\theta$ as
\begin{equation} \label{eq:hat-lambda}
	\hat\eta(h | s,\theta) \doteq h_0 - s_0 + (h_i - s_i) \theta_i + (h_{ij} - s_{ij}) \theta_i \theta_j
\end{equation}	
\begin{equation} \label{eq:hat-gamma}
		\hat\gamma(h | s, \theta) \doteq h_{ii}.
\end{equation}
We also denote $\psi_{\sigma^2}(H | y) = \psi_{\sigma^2}(h | s,\theta)$.
Then (\ref{eq:bpsy}) is expressed as
\begin{align}
\psi_{\sigma^2}(h | s,\theta) & \doteq \hat \eta(h | s,\theta) + \hat\gamma(h | s,\theta) \sigma^2 \nonumber\\
	 &\doteq h_0 - s_0 + (h_i - s_i) \theta_i + (h_{ij} - s_{ij}) \theta_i \theta_j + h_{ii} \sigma^2.\label{eq:psi-hsq}
\end{align}
\end{lemma}
\begin{proof}
Since $t$ and $\hat\gamma$ are geometric quantities which do not depend on the choice of coordinate system,
(\ref{eq:bp-at-origin}) gives (\ref{eq:bpsy}).
We will show (\ref{eq:hat-lambda}) and (\ref{eq:hat-gamma}) in Section~\ref{sec:proof-lemma-geometry-of-bp} (supplementary material).
Then (\ref{eq:psi-hsq}) is an immediate consequence of (\ref{eq:bpsy}).
\end{proof}

\subsection{Approximately unbiased $p$-value for selective inference} \label{sec:pvalue-si-LST}

The rejection region of approximately unbiased test for selective inference is given in the following theorem.
Only non-selective inference, i.e., the case of $S=\Rb^{m+1}$, has been discussed in the literature of the problem of regions, and we extend it to selective inference.

\begin{theorem} \label{theorem:au-surface}
Consider the hypothesis region $H = \Rc(h)$ and the selective region $S = \Rc(s)^c$ for any $h,s\in \Tc$.
For any $0<\alpha<1$, we can specify $r\in \Tc$ for the rejection region $R=\Rc(r)^c$ so that the selective rejection probability takes the constant value $\alpha$ for $\mu$ on $\partial H$;
\begin{equation} \label{eq:selective-rejection-probability}
	\frac{P(Y\in R \mid \mu )}{P(Y\in S\mid \mu)} \doteq \alpha,\quad \forall\mu\in\partial H.
\end{equation}
The coefficients of $r(u)$ are solved as
$r_0 \doteq  h_0 - h_{ii}  - \bar\Phi^{-1}( \alpha \bar\Phi(h_0 - s_0 - h_{ii}) )$,
$r_i \doteq h_i - \alpha C (h_i - s_i)$ and
$r_{ij} \doteq h_{ij} - \alpha C (h_{ij} - s_{ij})$, where
\[
  C = \frac{\phi(h_0 - s_0)}{\phi(\bar\Phi^{-1}(\alpha \bar\Phi(h_0 - s_0))   )}.
\]
For sufficiently large $n$, $r(u)\le s(u)$ and $R\subset S$ in the neighborhood of $u=0$,
and thus (\ref{eq:selective-rejection-probability}) is the conditional probability $P(Y\in R \mid Y\in S , \mu)$.
We also have an expression of $\alpha$ in terms of geometric quantities as
\begin{equation} \label{eq:au-surface-alpha}
	\frac{\bar \Phi(h_0 - r_0 - h_{ii})}{\bar \Phi(h_0 - s_0 -h_{ii})} \doteq \alpha.
\end{equation}
\end{theorem}
\begin{proof}
First, we proceed by assuming $r\in \Tc$.
For any $\mu\in\partial H$, we write
$\mu=(\theta,-h(\theta))$, $\theta\in \mathbb{R}^m$.
Then it follows from Lemma~\ref{lemma:geometry-of-bp} that
\begin{align*}
P(Y\in S \mid \mu) &= 1-\alpha_1(S^c | \mu) \doteq \bar\Phi( - \psi_1(s | h,\theta)),\\
P(Y\in R \mid \mu) &= 1-\alpha_1(R^c | \mu) \doteq \bar\Phi( - \psi_1(r | h,\theta)).
\end{align*}
Therefore, (\ref{eq:selective-rejection-probability}) is expressed as
$ \bar\Phi( - \psi_1(r | h,\theta)) \doteq \alpha \bar\Phi( - \psi_1(s | h,\theta))$.
Substituting (\ref{eq:psi-hsq}) into it,  we get
\begin{align}
  &\bar\Phi\bigl( h_0 - r_0 + (h_i - r_i)\theta_i + (h_{ij} - r_{ij}) \theta_i \theta_j - r_{ii} \bigr)\nonumber \\
  \doteq \label{eq:au-surface-proof1}
  &\alpha \bar\Phi\bigl( h_0 - s_0 + (h_i - s_i)\theta_i + (h_{ij} - s_{ij} ) \theta_i \theta_j - s_{ii} \bigr).
\end{align}
We then solve the equation for $r$. The rest of the proof is given in
 Section~\ref{sec:proof-theorem-au-surface} (supplementary material).
\end{proof}

Suppose we observed $y=(0,-r_0)\in \partial R$. Then the $p$-value should be $\alpha$ for $y$, and
we define $p_{\mathrm{SI}}(H | S, y) \doteq \alpha$ with (\ref{eq:au-surface-alpha}).
There are geometric quantities in (\ref{eq:au-surface-alpha}), namely,
the signed distance $h_0 - r_0$ from $y$ to $\partial H$,
the signed distance $h_0 - s_0$ from $\proj(H|y)$ to $\partial S$ and
the mean curvature $h_{ii}$ of $\partial H$ at $\proj(H|y)$.
All these geometric quantities can be estimated from bootstrap probabilities
$\alpha_{\sigma^2}(H | y)$ and $\alpha_{\sigma^2}(S | y)$.

Substituting $\theta=0$ in (\ref{eq:au-surface-proof1}), we get
\begin{equation} \label{eq:au-surface-alpha2}
	\frac{\bar \Phi(h_0 - r_0 - r_{ii})}{\bar \Phi(h_0 - s_0 -s_{ii})} \doteq \alpha.
\end{equation}
Interestingly, the mean curvatures $r_{ii}$ of $\partial R$ and $s_{ii}$ of $\partial S$ in (\ref{eq:au-surface-alpha2}) are replaced by the mean curvature $h_{ii}$ of $\partial H$ in (\ref{eq:au-surface-alpha}).
In Theorem~\ref{theorem:au-pvalue} below, $p$-value is computed from (\ref{eq:au-surface-alpha}),
while we are not able to use (\ref{eq:au-surface-alpha2}) for computing the $p$-value because
 $r_{ii}$ is not directly estimated by multiscale bootstrap before knowing $r$.

The following theorem justifies (A) of Algorithm~\ref{alg:simbp}
with 
$\varphi_{H}(\sigma^2 | \beta_H)=\beta_{H,1}+\beta_{H,2}\sigma^2$ and
$\varphi_{S}(\sigma^2 | \beta_S)=\beta_{S,1}+\beta_{S,2}\sigma^2$.

\begin{theorem} \label{theorem:au-pvalue}
Consider the hypothesis region $H = \Rc(h)$ and the selective region $S = \Rc(s)^c$ for any $h,s\in \Tc$. Define $p$-value as
\begin{equation} \label{eq:au-pvalue}
	p_{\mathrm{SI}}(H | S, y) = \frac{\bar\Phi(\psi_{-1}(H | y))} { \bar\Phi( \psi_{-1}(H | y) + \psi_0(S | y))}.
\end{equation}
For computing (\ref{eq:au-pvalue}), 
the values of $\psi_{\sigma^2}(H | y)$ and $\psi_{\sigma^2}(S | y)$ are extrapolated to $\sigma^2\le0$
by the linear regression on $\sigma^2$ for $\sigma^2>0$.
Then, this $p$-value is second order accurate.
For any $0<\alpha<1$, the rejection region $R=\{y \mid p_{\mathrm{SI}}(H | S, y) <\alpha \}$
satisfies (\ref{eq:selective-rejection-probability}) in Theorem~\ref{theorem:au-surface}.
\end{theorem}

\begin{proof}
The theorem is a direct consequence of Lemma~\ref{lemma:geometry-of-bp} and Theorem~\ref{theorem:au-surface}.
Let $y=(0,-r_0)$. From (\ref{eq:psi-hsq}), we have
\begin{equation} \label{eq:psi-h-regression}
	\psi_{\sigma^2}(H | y) \doteq \psi_{\sigma^2}(h | r,0) = h_0 - r_0 + h_{ii}\sigma^2,
\end{equation}
\begin{equation}\label{eq:psi-s-regression}
	\psi_{\sigma^2}(S | y) \doteq -\psi_{\sigma^2}(s | r,0) = -s_0 + r_0 - s_{ii}\sigma^2.
\end{equation}
By fitting the linear models (\ref{eq:psi-h-regression}) and (\ref{eq:psi-s-regression}) to
observed bootstrap probabilities for $\sigma^2>0$, we get $h_0-r_0$, $s_0 - r_0$, $h_{ii}$, $s_{ii}$ as regression coefficients.
Then extrapolating the models formally to $\sigma^2\le 0$, we have
$\psi_{-1}(H | y)\doteq h_0 - r_0 - h_{ii}$, $\psi_0(S | y) \doteq -s_0 + r_0$.
Substituting them into (\ref{eq:au-pvalue}), we get
\[
	p_{\mathrm{SI}}(H | S, y)  \doteq \frac{ \bar\Phi(h_0 - r_0 - h_{ii})} {\bar\Phi(h_0 - s_0 - h_{ii})},
\]
which coincides with the $\alpha$ in (\ref{eq:au-surface-alpha}).
The rest of the proof is given in Section~\ref{sec:proof-theorem-au-pvalue} (supplementary material).
\end{proof}

\begin{col}\label{col:non-selective-inference}
Consider the hypothesis region $H=\Rc(h)$ for any $h\in \Tc$.
For any $0<\alpha<1$, we can specify $r\in \Tc$ for the rejection region $R=\Rc(r)^c$
so that the non-selective rejection probability is $P(Y\in R | \mu) \doteq \alpha$, $\forall \mu\in \partial H$.
The coefficients of $r(u)$ are $r_0\doteq h_0 - h_{ii} - \bar\Phi^{-1}(\alpha)$, $r_i\doteq h_i$, $r_{ij} \doteq h_{ij}$.
This rejection region is expressed as $R = \{y \mid p_{\mathrm{AU}}(H|y)<\alpha\}$,
and the approximately unbiased $p$-value $p_{\mathrm{AU}}(H|y) = \bar\Phi(\psi_{-1}(H|y))$ is second order accurate;
in fact third order accurate as shown in \cite{Shimodaira:2004:AUT}.
\end{col}
\begin{proof}
All the second-order results follow by letting $S=\Rb^{m+1}$ in Theorem~\ref{theorem:au-surface} and Theorem~\ref{theorem:au-pvalue},
where $s_0=-\infty$, $C=0$, $\psi_0(S|y)=-\infty$.
\end{proof}

Let us verify that $p_{\mathrm{AU}}$ and $p_{\mathrm{BP}}$ are biased heavily for selective inference.
For a $p$-value $p(y)$, consider the rejection region $R=\{ y \mid p(y)<\alpha\}$ and
let us denote the selective rejection probability as $\alpha(p) = P(Y\in R \mid Y \in S, \mu)$ at $\mu=(0,-h(0))\in\partial H$, which is given by
the left hand side of (\ref{eq:au-surface-alpha2}). For $p=p_\text{SI}$, $\alpha(p_{\mathrm{SI}}) \doteq \alpha$.
%Consider the $p$-value of the form
%\begin{equation} \label{eq:p-sigma}
%  p_{\sigma^2}(H|y)=\bar\Phi(\psi_{\sigma^2}(H | y)).
%\end{equation}
Consider $p = p_{\sigma^2}$ of (\ref{eq:p-sigma}),
To get $\partial R$, $\psi_{\sigma^2}(h|r,\theta)=\bar\Phi^{-1}(\alpha)$ is solved for $r(u)$ by looking at the coefficients in (\ref{eq:psi-hsq}). Then $r(u)$ is given by $r_0 \doteq h_0 + h_{ii}\sigma^2 - \bar\Phi^{-1}(\alpha)$, $r_i \doteq h_i$, $r_{ij} \doteq h_{ij}$.
Substituting it into (\ref{eq:au-surface-alpha2}), $\alpha(p_{\sigma^2}) \doteq
\{\bar\Phi(\bar\Phi^{-1}(\alpha) - h_{ii}(1+\sigma^2))\}/\bar\Phi(h_0 - s_0 - s_{ii})$.
Therefore, 
$\alpha(p_{\mathrm{BP}}) \doteq \bar\Phi(\bar\Phi^{-1}(\alpha) - 2 h_{ii})/\bar\Phi(h_0 - s_0 - s_{ii})$ for $\sigma^2=1$ and
$\alpha(p_{\mathrm{AU}}) = \alpha/\bar\Phi(h_0 - s_0 - s_{ii})$ for $\sigma^2=-1$.
Due to the selection probability $P(Y\in S\mid \mu)$ in the denominator,
$\alpha(p_{\sigma^2}) $ is very much different from $\alpha$.

\subsection{Iterated bootstrap and related methods} \label{sec:iterated-bootstrap}

Iterated bootstrap is a general idea to improve the accuracy by applying bootstrap repeatedly.
It has been used for confidence intervals of parameters \citep{Hall:1986:BCI},
and for the problem of regions as well \citep{Efron:Tibshirani:1998:PR}.
The computational cost (time complexity) of $k$th-iterated bootstrap is $O(B^k)$ when each bootstrap uses $B$ bootstrap replicates, while that of multiscale bootstrap is only $O(B)$,
and thus it is often prohibitive even for the double bootstrap, i.e., the iterated bootstrap with $k=2$.
It also requires the computation of $\proj(H|y)$, which can be difficult in applications.
Here we show that multiscale bootstrap calculates $p$-values equivalent to double bootstrap with less computation.

Let $p_{\mathrm{BP},k}(H|S,y)$, $k=1,2,\ldots$, be the series of iterated bootstrap $p$-values.
At Step $k+1$, we compute
\begin{equation} \label{eq:iterated-bp-k}
  p_{\mathrm{BP},k+1}(H|S,y) = 
  \frac{P_1 \{ p_{\mathrm{BP},k}(H|S,Y^*) <  p_{\mathrm{BP},k}(H|S,y)  \mid \proj(H|y) \} }
  {P_1(Y^* \in S \mid \proj(H|y))},
\end{equation}
where the probability is with respect to the null distribution (\ref{eq:yboot-zero}).
The following theorem shows that the double bootstrap computes $p$-value equivalent to $p_{\mathrm{SI}}$.
The double bootstrap is robust to the computational error in the $u$-axis of $\proj(H|y)$.

\begin{theorem} \label{theorem:double-bootstrap}
Consider the hypothesis region $H = \Rc(h)$ and the selective region $S = \Rc(s)^c$ for any $h,s\in \Tc$.
Let $y=(0,-r_0)$ and $\hat\mu = (u_0, -h(u_0)+O(n^{-1}))$ for any $u_0 = O(n^{-1/2})$ by allowing the error of $ O(n^{-1/2})$ to the $u$-axis of $\proj(H|y)$, which is
$u_i = -(h_0-r_0)h_i=O(n^{-1/2})$ according to Lemma~\ref{lemma:change-of-coordinates} (supplementary material).
For $k=1$, we adjust (\ref{eq:p-sigma}) by the selection probability to define
\begin{equation} \label{eq:p-bp-1}
  p_{\mathrm{BP},1}(H|S,y) = \frac{p_{\sigma^2}(H|y)}
  {P_1(Y^* \in S \mid \hat \mu)}
\end{equation}
for some $\sigma^2\in\mathbb{R}$,
and apply (\ref{eq:iterated-bp-k}) for computing $ p_{\mathrm{BP},2}(H|S,y)$.
Then we have
\begin{equation} \label{eq:p-bp-1-2-lst}
  p_{\mathrm{BP},1}(H|S,y) \doteq\frac{ \bar\Phi(h_0 - r_0 + h_{ii}\sigma^2)} {\bar\Phi(h_0 - s_0 - s_{ii})},\,
  p_{\mathrm{BP},2}(H|S,y) \doteq\frac{ \bar\Phi(h_0 - r_0 - h_{ii})} {\bar\Phi(h_0 - s_0 - h_{ii})}.
\end{equation}
Therefore, $p_{\mathrm{BP},2}(H|S,y)\doteq p_{\mathrm{SI}}(H|S,y)$, i.e., equivalence in the second order accuracy,
and then $p_{\mathrm{BP},2}(H|S,y)$ is second order accurate.
The result does not depend on $\sigma^2$; 
the numerator of (\ref{eq:p-bp-1}) can be $\bar\Phi(t)=p_0(H|y)$ for $\sigma^2=0$, say.
%On the other hand, $p_{\mathrm{BP},1}(H|S,y) = p_{\mathrm{BP},2}(H|S,y)+O_p(n^{-1/2})$ is only first order accurate for $\sigma^2\neq -1$,
%but $p_{\mathrm{BP},1}(H|S,y)$ becomes second order accurate if $\sigma^2=-1$. => ミス
On the other hand, $p_{\mathrm{BP},1}(H|S,y) = p_{\mathrm{BP},2}(H|S,y)+O_p(n^{-1/2})$ is only first order accurate,
but $p_{\mathrm{BP},1}(H|S,y)$ becomes second order accurate if $S=H^c$ and $\sigma^2=-1$.
\end{theorem}

\begin{proof}
See Section~\ref{sec:proof-theorem-double-bootstrap} (supplementary material).
\end{proof}

In Section~\ref{sec:intro-nulldist}, we have introduced $p_\mathrm{ET}$ \citep{Efron:Tibshirani:1998:PR} as a bias correction method using the null distribution. Here we extend it to selective inference for general $S$. The $p$-value is defined as
\[
  p_\text{ET-SI}(H|S,y) = \frac{\bar\Phi( \psi_1(H|y) - 2 z_\text{proj}(H|y) )}{ \bar\Phi(  \psi_1(H|y) - 2 z_\text{proj}(H|y) + \psi_1(S|y) - z_\text{proj}(S|y)) }.
\]
For the case of $S=H^c$, $ p_\text{ET-SI}(H|H^c,y) = p_\text{ET-SI}(H|y)$ because
$z_\text{proj}(H^c|y)=-z_\text{proj}(H|y)$, $\psi_1(H^c|y)=-\psi_1(H|y)$.
Considering the setup of Theorem~\ref{theorem:double-bootstrap},
the four terms  in $p_\text{ET-SI}(H|S,y)$ are expressed as
$z_\text{proj}(H|y)=\psi_1(H|\proj(H|y)) \doteq h_{ii}$, $z_\text{proj}(S|y)=\psi_1(S|\proj(S|y)) \doteq -s_{ii}$,
$\psi_1(H|y) \doteq h_0-r_0+h_{ii}$ and $\psi_1(S|y) \doteq -s_0+r_0-s_{ii}$.
Therefore $ p_\text{ET-SI}(H|S,y) \doteq \bar\Phi( h_0-r_0+h_{ii} - 2 h_{ii} )/\bar\Phi( h_0-r_0+h_{ii} - 2 h_{ii} -s_0+r_0-s_{ii} + s_{ii}) = \bar\Phi(h_0-r_0-h_{ii} )/\bar\Phi( h_0-s_0-h_{ii} ) \doteq p_\mathrm{SI}(H|S,y)$.
Thus $p_\text{ET-SI}(H|S,y) \doteq p_\mathrm{SI}(H|S,y)$ again, and they are equivalent in the second order accuracy.

%%%------------------------------------------------------------------------------------------------------------------------------
%%
%

%
%%
%%%------------------------------------------------------------------------------------------------------------------------------

%
%%
%%%------------------------------------------------------------------------------------------------------------------------------
\section{Asymptotic theory for non-smooth boundary surfaces}\label{sec:NFT}

\subsection{Nearly flat surfaces} \label{sec:nearly-flat}

In the previous section, 
we consider asymptotic behavior as $n$ goes to infinity. The shape of $H$ in the normalized space is magnified by $\sqrt{n}$.
In this large sample theory, the key point is that 
the boundary surface $\partial H$ of the hypothesis region approaches a flat surface 
in a neighborhood of any point on $\partial H$ if the surface is smooth.
However, this argument cannot apply to nonsmooth surfaces.
For example, if $H$ is a cone-shaped region, it is scale-invariant;
the shape remains as cone in the neighborhood of the vertex.
In many real world problems such as clustering and variable selection,
hypothesis and selective regions are represented as polyhedral convex cones (or their complement sets) at least locally thus have nonsmooth boundaries.
Although the chi-bar squared distribution appears in this kind of statistical inference under inequality constraints \citep{shapiro1985asymptotic,lin1997projections}, computation of the coefficients seems not very easy for our setting.

To deal with general regions with possibly nonsmooth boundary surfaces, 
we employ the asymptotic theory of {\it nearly flat surfaces} \citep{Shimodaira08}, which is reviewed in Sections~\ref{sec:nearly-flat} and \ref{sec:models}.
%In this asymptotic theory, we provide some theoretical justifications for our algorithm.
We provide a theoretical justification for (B) in Algorithm~\ref{alg:simbp}.
Roughly speaking, we consider the situation that the magnitude of $h$, say $\lambda$, 
becomes small
so that the file drawer problem of (\ref{eq:file-drawer-c}) appears again as the limiting distribution.
The scale in the direction of the tangent space is fixed in this theory
so that any boundary surfaces approach flat surfaces.
Instead of $n\rightarrow \infty$, we introduce the artificial parameter $\lambda$ and 
let $\lambda \rightarrow 0$.
It is worth noting that 
this theory is analogous to the classical theory with the relation $\lambda = 1/\sqrt{n}$.
Although this theory does not dependent on $n$, 
we implicitly assume that $n$ is sufficiently large to ensure the multivariate normal model (\ref{eq:model}).
Instead of the notation $\doteq$ used in previous sections,
we use $\simeq$ for the equality correct up to $O(\lambda)$ erring only $O(\lambda^2)$ in this section.

As with the previous section, for $y=(y_1,\dots,y_{m+1})\in \Rb^{m+1}$, 
let $u=(y_1,\dots,y_{m})$ and $v=y_{m+1}$.
For a continuous function $h:\Rb^m\rightarrow \Rb$ and $v_h\in \Rb$, 
we define the region by
\[
\Rc(h,v_h)=\{(u,v)\in \Rb^{m+1}\mid v\le v_h - h(u)\}.
\] 
When we consider $v_h$ as $-h_0$, $\Rc(h,v_h)$ corresponds to $\Rc(h)$ introduced in Section~\ref{sec:LST}.
Let us denote $L^1$-norm and $L^\infty$-norm of $h$ by
$\|h\|_1=\int_{\Rb^m} |h(u)|\,du$ and
$\|h\|_\infty = \sup_{u\in \Rb^m}|h(u)|$,
respectively.
We say that $h$ is \emph{nearly flat} if $\|h\|_\infty=O(\lambda)$, 
and if $L^1$-norms of $h$ and its Fourier transform $\tilde{h}$ are bounded; $\|h\|_1 < \infty$ and $\|\tilde h\|_1 < \infty$.
However, polynomials and cones are unbounded, and they are obviously not nearly flat.
As mentioned in Section~5.4 and Appendix~A.4 of \cite{Shimodaira08}, the results can be generalized to continuous functions of slow growth;
$|g(u)|=O(\|u\|^k)$ as $\|u\| \to\infty$ for some $k$. We can take a nearly flat $h$ approximating $g$ arbitrary well in a sufficiently large window.
In practical situations, the magnitude of $h$ is not necessarily too small.
From the numerical examples, we may see that our theory works even for a moderate $\lambda$.

The hypothesis and selective regions are defined, respectively, by
\[
  H=\Rc(h,0),\mbox{ and } S=\Rc(s,v_s)^c
\]
for nearly flat functions $h$ and $s$, and $v_s\in\mathbb{R}$.
Note that, for $H=\Rc(h,v_h)$, 
we can redefine $H$ as $\Rc(h,0)$ in the coordinate taking the origin at $(0,v_h)$.
For $0<\alpha<1$, let $v_r$ be a constant satisfying 
\[
\bar\Phi(v_r)=\alpha\bar\Phi(v_s),
\]
and let $R=\Rc(r,v_r)^c$ be a rejection region. 
Here, $v_s$ and $v_r$ correspond to $h_0-s_0$ and $h_0-r_0$ in Section~\ref{sec:LST}, respectively.

We will denote the Fourier transform of a nearly flat function $h$ by
$\tilde{h}(\omega)=\mathcal{F}h(\omega)=\int_{\Rb^m}e^{-i\omega \cdot u}h(u)\,du$, where $\omega\in \Rb^m$ is a spatial angular frequency vector and $i=\sqrt{-1}$ is the imaginary unit.
Moreover, let $h(u)=(\mathcal{F}^{-1}\tilde{h})(u)=(2\pi)^{-m}\int_{\Rb^m}e^{i\omega\cdot u}\tilde{h}(\omega)\,d\omega$ be the inverse Fourier transform of $\tilde{h}$.
Using these notations, 
we can represent the expected value of $h(U^\ast)$ 
with respect to $U^\ast\sim N_m(u,\sigma^2I_m)$ 
as follows:
\[
E_{\sigma^2}h(u)=E_{\sigma^2}(h(U^\ast)\mid u)=\Fc^{-1}\bigl[ e^{-\sigma^2\|\omega\|^2/2}\tilde{h}(\omega)\bigr](u).
\]
This is an application of the Gaussian low-pass filter 
$\tilde{f}_{\sigma^2}(\omega) = e^{-\sigma^2\|\omega\|^2/2}$ to $\tilde{h}$.
The inverse filter of $\tilde{f}_{\sigma^2}(\omega)$ is defined by 
$\mathcal{F}[E_{\sigma^2}^{-1}h](\omega)=(1/\tilde{f}_{\sigma^2}(\omega))\tilde{h}(\omega)$.
Applying the inverse Fourier transform to it, we can define the expected value with a negative variance, at least formally,
by
\[
E_{\sigma^2}^{-1}h(u)=\mathcal{F}^{-1}[(1/\tilde{f}_{\sigma^2}(\omega))\tilde{h}(\omega)]=\mathcal{F}^{-1}[e^{\sigma^2\|\omega\|^2/2}\tilde{h}(\omega)](u)=E_{-\sigma^2}h(u).
\]
Note that $E_{-\sigma^2}h=E_{\sigma^2}^{-1}h$ may not be defined unless
$\|e^{\sigma^2\|\omega\|^2/2}\tilde{h}(\omega)\|_1<\infty$ 
even though $E_{\sigma^2}h$ with $\sigma^2>0$ is nearly flat.

First of all, we provide the fundamental result in the theory of nearly flat surfaces 
corresponding to Lemma~\ref{lemma:geometry-of-bp} in the large sample theory.

\begin{lemma}\label{lemma:geometry-of-bp:NFS}
For a nearly flat function $h$ and a constant $v_h\in \Rb$, 
let $H=\Rc(h,v_h)$.
For $y=(u,v)\in \Rb^{m+1}$ and $\sigma^2>0$, we have 
\banum\label{eq:bp-approx}
\alpha_{\sigma^2}(H | y) 
= \Psig\lpar Y^\ast\in H \mid y\rpar
\simeq
\bPhi\left( \frac{v-v_h+E_{\sigma^2}h(u)}{\sigma} \right),
\eanum
and the 
the normalized bootstrap $z$-value is expressed as
\banum\label{eq:nz-value}
\psi_{\sigma^2}(H | y)=\sigma \bPhi^{-1}(\alpha_{\sigma^2}(H | y))\simeq v-v_h+E_{\sigma^2}h(u).
\eanum
\end{lemma}
\begin{proof}
See Section~\ref{sec:proof-geometry-of-bp:NFS} (supplementary material).
\end{proof}

%---
\subsection{Models for normalized bootstrap $z$-value}\label{sec:models}
%---
The key point of our algorithm is that 
the functional form of $\psi_{\sigma^2}(H | y)$ is estimated 
from the observed bootstrap probabilities computed at several $\sigma^2=n/n^\prime$.
We need a good parametric model $\varphi_{H}(\sigma^2 | \beta(y))$ with parameter $\beta(y)$.
From the scaling-law (\ref{eq:nz-value}),
it is important to specify an appropriate parametric model for $E_{\sigma^2}h(u)$.
The following results are shown in Section~5.4 of \cite{Shimodaira08}.

For smooth $h$, we have
\[
E_{\sigma^2}h(u)=\sum_{j=0}^\infty \sigma^{2j}\beta_j(u),
\]
where $\beta_0(u)=h(u)$, $\beta_1(u)=(1/2)\sum_{i=1}^m\partial^2h/\partial u_i^2$, and
\[
	\beta_j(u) = \frac{1}{2^j j!} \sum_{j_1+\cdots + j_m = j} \frac{j!}{j_1!\cdots j_m!}
	\frac{\partial ^{2j} h}{ \partial u_1^{2j_1} \cdots \partial u_m^{2j_m}}, \quad j\ge0.
\]
When the boundary surface can be approximated by a polynomial of degree $2k-1$, 
we may consider the following model, denoted poly.$k$, by redefining $\beta_0=v+\beta_0(u)$:
\begin{equation} \label{eq:polyk}
\varphi(\sigma^2 | \beta)=\sum_{j=0}^{k-1}\beta_j\sigma^{2j},\quad k\ge 1.
\end{equation}
If $h$ is a polynomial of degree $2k-1$, 
the model poly.$k$ correctly specifies $\psi_{\sigma^2}(H | y)$ 
by ignoring $O(\lambda^2)$ term.
It is worth noting that the parameters $(\beta_0,\beta_1,\dots)$ are interpreted as geometric quantities; 
$\beta_0$ is the signed distance from $y$ to the surface $\partial H$, and
$\beta_1$ is the mean curvature of the surface.

For a nonsmooth $h$, the above model is not appropriate. In fact, 
for a cone-shaped $H$ with the vertex at the origin, we have
\[
E_{\sigma^2}h(u)=\sum_{j=0}^\infty \sigma^{1-j}\beta_j(u)
\]
in a neighborhood of the vertex, where $\beta_j(u)=O(\|u\|^j)$ as $\|u\|$ goes to $0$.
The following model, denoted sing.$k$, takes conical singularity into account.
\begin{equation} \label{eq:singk}
\varphi(\sigma^2 |\beta)=\beta_0+\sum_{j=1}^{k-2}\frac{\beta_j\sigma^{2j}}{1+\beta_{k-1}(\sigma-1)},\quad k\ge 3,
\end{equation}
where $0\le \beta_{k-1}\le 1$. 
In practical situations, we are not sure which parametric model is the reality.
Thus, we prepare several candidate models describing the scaling-law of bootstrap probability, 
and choose the model based on the AIC value.

\subsection{Approximately unbiased $p$-values for selective inference in the theory of nearly flat surfaces}\label{sec:au-p-value-nfs}

Now, we ensure the existence of the rejection region corresponding to 
an approximately unbiased selective inference.
This result corresponds to Theorem~\ref{theorem:au-surface} of the large sample theory.
\begin{lemma}\label{lemma:rej-surface-nfs}
For nearly flat functions $h$ and $s$, and a constant $v_s\in \Rb$, 
we set $H=\Rc(h,0)$ and $S=\Rc(s,v_s)^c$ as the hypothesis and the selective regions, respectively.
Suppose that $E_{-1}h$ exists and is nearly flat.
Then, for a given $0<\alpha<1$, 
there exists a nearly flat function $r$  such that
\banum\label{eq:selec-rej-prob-nfs}
\frac{P_1(Y\in R\mid \mu)}{P_1(Y\in S\mid \mu)} \simeq \alpha,\quad \forall\mu\in \partial H,
\eanum
where $R=\Rc(r,v_r)^c$ and $v_r = \bar\Phi^{-1}(\alpha \bar\Phi(v_s))$.
The function $r$ is solved as
\banum\label{eq:unbias-rej-surface}
r(u)\simeq E_{-1}h(u)+\alpha C\{ s(u)-E_{-1}h(u) \},
\eanum
where $C=\phi(v_s)/\phi(v_r)$.
We also have an expression of $\alpha$ as
\banum\label{eq:rej-surface-alpha}
 \frac{\bar\Phi(v_r-r(u)+E_{-1}h(u))}{ \bar\Phi(v_s-s(u)+E_{-1}h(u))} \simeq \alpha.
\eanum
\end{lemma}
\begin{proof}
See Section~\ref{sec:proof-rej-surface-nfs} (supplementary material).
\end{proof}

When $y=(0,v_r-r(u))$ is observed, we have $y\in \partial R$ and an approximately unbiased $p$-value for $y$ is set as $\alpha$.
We define a selective $p$-value $p_{\mathrm{SI}}(H | S, y)$ by using $\alpha$ in (\ref{eq:rej-surface-alpha}), that is, 
$p_{\mathrm{SI}}(H | S, y)\simeq \alpha$.
Although several unknown quantities $v_r-r(u)$, $v_s-s(u)$ and $E_{-1}h(u)$ appear in the definition of $p_{\mathrm{SI}}(H | S, y)$,
we can compute these quantities by using bootstrap probabilities $\alpha_{\sigma^2}(H | y)$ and $\alpha_{\sigma^2}(S | y)$.
The following theorem shows that the $p$-value computed by (A) of Algorithm~\ref{alg:simbp} 
is unbiased ignoring $O(\lambda^2)$ terms.

\begin{theorem} \label{theorem:au-pvalue-nfs}
Suppose the assumptions in Lemma~\ref{lemma:rej-surface-nfs} hold.
Also suppose that
the functional forms of $\psi_{\sigma^2}(H | y)$ and $\psi_{\sigma^2}(S | y)$ 
can be extrapolated to $\sigma^2=-1$ and $\sigma^2=0$, respectively.
We define a selective $p$-value by
\begin{equation} \label{eq:au-pvalue-nfs}
	p_{\mathrm{SI}}(H | S, y) = \frac{\bar\Phi(\psi_{-1}(H | y))} { \bar\Phi( \psi_{-1}(H | y) + \psi_0(S | y))}.
\end{equation}
For given significance level $\alpha$, we set the rejection region by $R=\{y \in \Rb^{m+1}\mid p_{\mathrm{SI}}(H | S, y) <\alpha\}$.
Then, this $R$ is equivalent to that in Lemma~\ref{lemma:rej-surface-nfs} erring only $O(\lambda^2)$, and thus $R$ satisfies (\ref{eq:selec-rej-prob-nfs}).
\end{theorem}
\begin{proof}
From Lemma~\ref{lemma:geometry-of-bp:NFS}, 
normalized bootstrap $z$-values $\psi_{\sigma^2}(H | y)$ and $\psi_{\sigma^2}(S | y)$ for $y=(u,v)$
can be expressed by
\banum\label{eq:psi-H&S-rep}
\psi_{\sigma^2}(H | y)\simeq v + E_{\sigma^2}h(u),\;
\psi_{\sigma^2}(S^c | y)\simeq v - v_s + E_{\sigma^2}s(u),
\eanum
respectively.
Then $\psi_{-1}(H | y)\simeq v + E_{-1}h(u)$ by extrapolating it to $\sigma^2=-1$.
By noting $E_0s(u)=s(u)$ and $\psi_{\sigma^2}(S^c | y)=-\psi_{\sigma^2}(S | y)$,
we have $\psi_{0}(S | y) \simeq -v+v_s-s(u)$.
By substituting them into (\ref{eq:au-pvalue-nfs}), we get an expression
\begin{equation}\label{eq:au-pvalue2-nfs}
   p_{\mathrm{SI}}(H | S, y)\simeq \frac{\bar\Phi(v + E_{-1}h(u))} { \bar\Phi( v_s-s(u) + E_{-1}h(u))}.
\end{equation}
Let $r$ and $R$ be those defined in Lemma~\ref{lemma:rej-surface-nfs}.
For $y\in \partial R$, $v=v_r-r(u)$ and then (\ref{eq:au-pvalue2-nfs}) coincides with (\ref{eq:rej-surface-alpha}).
Therefore $ p_{\mathrm{SI}}(H | S, y) \simeq \alpha$ on  $y\in \partial R$.
For $y = (u,v)\in\mathbb{R}^{m+1}$, by looking at the numerator of (\ref{eq:au-pvalue2-nfs}),
we get
$p_{\mathrm{SI}}(H | S, y) < \alpha \Leftrightarrow v + E_{-1}h(u) > v_r - r(u) + E_{-1}h(u)  \Leftrightarrow v > v_r - r(u) \Leftrightarrow y \in R$, where  $O(\lambda^2)$ terms are ignored.
\end{proof}

\subsection{A class of approximately unbiased tests for selective inference}

In Section~\ref{sec:au-p-value-nfs}, we assumed that the functional forms of $\psi_{\sigma^2}(H | y)$ and $\psi_{\sigma^2}(S | y)$ can be extrapolated to $\sigma^2=-1$ and $\sigma^2=0$, respectively.
Unfortunately, however, parametric models for cone-shaped regions, e.g., sing.$k$,
can only be defined for $\sigma^2>0$. This is in parallel with the argument of \cite{Lehmann:1952:TMH} 
that an unbiased test does not exist for a cone-shaped hypothesis region;
see also \cite{perlman1999emperor} for counter-intuitive illustrations.
On the other hand, Stone-Weierstrass theorem argues that
any continuous functions $h$ and $s$ can be approximated arbitrary well by polynomials within a bounded window on $u$.
From this point of view, the selective $p$-value using poly.$k$ (\ref{eq:polyk}) in (A) of Algorithm~\ref{alg:simbp}
becomes unbiased as $k\rightarrow \infty$ ignoring $O(\lambda^2)$ terms by taking a sufficiently large window,
although fitting of high-degree polynomials of $\sigma^2$, namely large $k$ in (\ref{eq:polyk}), can become unstable especially outside the range of fitted points for extrapolation.

In the same manner as \cite{Shimodaira08},
the method (B) in Algorithm~\ref{alg:simbp} considers truncated Taylor series expansion of $\psi_{\sigma^2}(H | y)$ with $k$ terms
at a positive $\sigma_{-1}^2>0$ as
\ba
\psi_{\sigma^2,k}(H | y,\sigma_{-1}^2)
=\sum_{j=0}^{k-1}
\frac{(\sigma^2-\sigma_{-1}^2)^j}{j!}
\frac{\partial^j \psi_{\sigma^2}(H | y)}{\partial (\sigma^2)^j}\Biggr|_{\sigma^2=\sigma_{-1}^2},
\ea
and similarly for  $\psi_{\sigma^2}(S | y)$ at $\sigma_{0}^2>0$.
Then we compute the selective $p$-value by
\banum\label{eq:taylor-pval-si}
p_{\mathrm{SI},k}(H | S,y)=
\frac{\bPhi(\psi_{-1,k}(H | y,\sigma_{-1}^2))}
{\bPhi( \psi_{0,k}(S | y,\sigma_0^2)+\psi_{-1,k}(H | y,\sigma_{-1}^2) )}.
\eanum
Although this method can be interpreted as the polynomial fitting, namely (A) in Algorithm~\ref{alg:simbp} with (\ref{eq:polyk}),
in small neighborhoods of $\sigma_{-1}^2$ and $\sigma_{0}^2$, 
it is more stable than the polynomial fitting when a wider range of $\sigma^2$ is used for model fitting.

The following theorem provides the theoretical justification for a class of general $p$-values 
including (\ref{eq:taylor-pval-si}).
\begin{theorem}\label{theo:gen-pval-si}
For nearly flat functions $h$ and $s$, and a constant $v_s\in \Rb$, 
we set $H=\Rc(h,0)$ and $S=\Rc(s,v_s)^c$ as the hypothesis and the selective regions, respectively.
For a given $0<\alpha<1$, let $v_r=\bar{\Phi}^{-1}(\alpha \bar{\Phi}(v_s))$ and 
$A(v)=A(v,v_s)=(\bar{\Phi}(v)\phi(v_s))/(\phi(v)\bar{\Phi}(v_s))$.
Let $I_k(\omega)$ and $J_k(\omega)$ denote functions satisfying the following three conditions:
\begin{description}
\item{(i) } $\lim_{k\rightarrow \infty}I_k(\omega)=0$ and $\lim_{k\rightarrow \infty}J_k(\omega)=0$ for each $\omega\in \Rb^m$,
\item{(ii) } $\exists C>0;\;\forall k\in \Nb;\;\|I_k(\omega)\|_\infty,\;\|J_k(\omega)\|_\infty < C$, and
\item{(iii) } $\forall k\in \Nb;\;\|e^{\|\omega\|^2/2}I_k(\omega)\|_\infty,\;\|(1-A(v_r)-J_k(\omega))e^{\|\omega\|^2/2}\|_\infty<\infty$.
\end{description}
Then $r_k(u)=r_k(u,v_r)$ exists, where it is defined by
\[ %\label{eq:gen-rej}
  r_k(u,v) = \Fc^{-1}\lsb
\{1-A(v)-J_k(\omega)\}e^{\|\omega\|^2/2}\tilde{h}(\omega) 
+
\{A(v)-e^{\|\omega\|^2/2}I_k(\omega)\}\tilde{s}(\omega)
\rsb.
\]
We consider a general $p$-value $p_{k}(H | S, y)$ which can be represented by
\banum\label{eq:gen-pval-si}
p_{k}(H | S, (u,v)) \simeq \frac{\bar\Phi\left(v+r_k(u)\right)}{\bar\Phi\left(v_s\right)} 
\;\text{ for }\; v=v_r+O(\lambda).
\eanum
Note that $r_k(u)$ in (\ref{eq:gen-pval-si}) can be replaced by $r_k(u,v)\simeq r_k(u)$ for $v = v_r + O(\lambda)$.
Then, we have, for $\mu=(\theta,-h(\theta))\in \partial H$,
\banum\label{eq:conv-gen-pval-si}
\frac{P(p_{k}(H | S, Y)<\alpha \mid \mu)}{P(Y\in S\mid \mu)}\rightarrow \alpha+O(\lambda^2)\;\text{ as }\; k\rightarrow \infty.
\eanum
at each $\theta\in \Rb^m$.

In addition to the conditions (i), (ii), (iii), 
we assume that $I_k$ and $J_k$ can be expressed by
\begin{description}
\item{(iv) } $I_k(\omega)=\sum_{j=k}^\infty a_{k,j}\|\omega\|^{2j},\;J_k(\omega)=\sum_{j=k}^\infty b_{k,j}\|\omega\|^{2j}$, respectively.
\end{description}
Then, if $h$ and $s$ are polynomials of degree less than or equal to $2k-1$, 
$p_{k}(H | S, y)$ is unbiased ignoring $O(\lambda^2)$ term.
\end{theorem}
\begin{proof}
See Section~\ref{sec:proof-gen-pval-si} (supplementary material).
\end{proof}

Using this theorem, we can establish theoretical guarantees for 
our approach using the truncated Taylor series expansion, and also for the iterated bootstrap described in Section~\ref{sec:iterated-bootstrap}.

\begin{col}\label{col:taylor-ibp-pval-si}
For nearly flat functions $h$ and $s$, define $H$, $S$ and $A(v_r)$ as in Theorem~\ref{theo:gen-pval-si}.
Then,
the $p$-value $p_{\mathrm{SI},k}(H | S,y)$ defined by (\ref{eq:taylor-pval-si}) satisfies
(\ref{eq:conv-gen-pval-si}).
We also assume that $h$ is Lipschitz continuous with Lipschitz constant $K(\lambda)=O(\lambda)$ for 
the $p$-value $p_{\mathrm{BP},k}(H | S,y)$ defined by (\ref{eq:iterated-bp-k}) and (\ref{eq:p-bp-1}) for $\sigma^2>0$.
Then, $p_{\mathrm{BP},k}(H | S,y)$ satisfies (\ref{eq:conv-gen-pval-si}).
In addition to above conditions, we further assume that
$h$ and $s$ can be represented by polynomials of degree less than or equal to $2k-1$.
Then $p_{\mathrm{SI},k}(H | S,y)$ and $p_{\mathrm{BP},k}(H | S,y)$ are
unbiased ignoring $O(\lambda^2)$ term.
\end{col}
\begin{proof}
The results are immediate consequences of Theorem~\ref{theo:gen-pval-si} by knowing
that $p_{\mathrm{SI},k}(H | S,y)$ and $p_{\mathrm{BP},k}(H | S,y)$ satisfy the conditions (i), (ii), (iii), (iv)
as shown below in Lemma~\ref{lemma:taylor-pval-si} and Lemma~\ref{lemma:ibp-pval-si}, respectively.
\end{proof}

The following lemma shows the correspondence 
between $p_{\mathrm{SI},k}(H | S,y)$ and a general $p$-value $p_{k}(H | S,y)$ in Theorem~\ref{theo:gen-pval-si}.

\begin{lemma}\label{lemma:taylor-pval-si}
Assume the same conditions as in Corollary~\ref{col:taylor-ibp-pval-si}.
Then the $p$-value $p_{\mathrm{SI},k}(H | S,y)$ can be represented as (\ref{eq:gen-pval-si}) 
using the following $I_k$ and $J_k$:
\ba
J_k(\omega)
=
(1-A(v_r))G_{k}(\omega| -1,\sigma_{-1}^2),\quad
I_k(\omega)
=
A(v_r)e^{-\|\omega\|^2/2}G_{k}(\omega|0,\sigma_{0}^2),
\ea
where, for $\sigma_a^2 \le \sigma_b^2$,
\[
G_{k}(\omega | \sigma_a^2,\sigma_{b}^2)
=
\frac{\gamma(k,(\sigma_{b}^2-\sigma_{a}^2)\|\omega\|^2/2)}{\Gamma(k)}
=
\sum_{j=k}^\infty\frac{(-1)^{j-k}}{(k-1)!(j-k)!j} \frac{(\sigma_{b}^2-\sigma_{a}^{2})^j\|\omega\|^{2j}}{2^j},
\]
and $\gamma(n,z)=\int_{0}^z t^{n-1}e^{-t}\,dt$ is the lower incomplete gamma function.
The above $I_k$ and $J_k$ satisfy the conditions {\rm (i)-(iv)} in Theorem~\ref{theo:gen-pval-si}.
\end{lemma}
\begin{proof}
See Section~\ref{sec:proof-taylor-pval-si} (supplementary material).
\end{proof}

The iterated bootstrap in Corollary~\ref{col:taylor-ibp-pval-si} is discussed in parallel with
Theorem~\ref{theorem:double-bootstrap} in Section~\ref{sec:iterated-bootstrap}.
The next result provides the connection between $p_{\mathrm{BP},k}(H | S,y)$ and a general $p$-value $p_{k}(H | S,y)$ in Theorem~\ref{theo:gen-pval-si}.

\begin{lemma}\label{lemma:ibp-pval-si}
Assume the same conditions as in Corollary~\ref{col:taylor-ibp-pval-si}.
Then
$p_{\mathrm{BP},k}(H | S,y)$ defined by (\ref{eq:iterated-bp-k}) and (\ref{eq:p-bp-1})  for $\sigma^2>0$ can be represented as (\ref{eq:gen-pval-si})
using the following $I_k$ and $J_k$:
\ba
J_k(\omega)
&=
( 1-e^{-\|\omega\|^2/2} )^{k-1}
\Bigl\{1-e^{-(1+\sigma^2)\|\omega\|^2/2}- (1-e^{-\|\omega\|^2/2}) A(v_r) \Bigr\},\\
I_k(\omega)
&=
(1-e^{-\|\omega\|^2/2})^{k}e^{-\|\omega\|^2/2}  A(v_r).
\ea
The above $I_k$ and $J_k$ satisfy the conditions {\rm (i)-(iv)} in Theorem~\ref{theo:gen-pval-si}.
\end{lemma}
\begin{proof}
See Section~\ref{sec:proof-ibp-pval-si} (supplementary material).
\end{proof}
%%%%------------------------------------------------------------------------------------------------------------------------------
%%
%

\section{Concluding Remarks} \label{sec:conclusion}

The argument of multiscale bootstrap is generalized to the exponential family of distributions in \cite{Efron:Tibshirani:1998:PR} and \cite{Shimodaira:2004:AUT}, where the acceleration constant $\hat a$ of the ABC formula in \cite{Efron:1987:BBC} and \cite{Diciccio:Efron:1992:MAC} is considered.
$\hat a$ is interpreted as the rate of change of the covariance matrix in the normal direction to $\partial H$, and $\hat a =0$ in the fixed covariance matrix case.
In this paper, we ignore $\hat a$ by assuming the model (\ref{eq:model}) holds for $f_n(\Xc_n)$.
This assumption corresponds to $\Delta V=0$ in Section~\ref{sec:pvclust-fn} (supplementary material).
The value of $\hat a$ is relatively small in examples of \cite{Efron:Tibshirani:1998:PR} and \cite{Shimodaira:2004:AUT}, where several attempts have been made to estimate $\hat a$.
Computing $\hat a$ requires further theoretical and computational effort, and this is left as a future work.

There are several ways to deal with multiple testing, and the selective inference discussed in this paper is only one of them.
We have tested hypotheses separately for controlling the conditional rejection probability of each hypothesis.
Therefore, it would be interesting to consider other types of multiple testing, such as false discovery rate and family-wise error rate, together with our selective inference in a similar manner as \cite{benjamini2014selective}.

We have not discussed power of testing for comparing $p$-values, although the choice of $S$ is mentioned briefly in the last paragraph of Section~\ref{sec:problem-setting}.
Since our $p$-values are asymptotically derived by modifying (\ref{eq:file-drawer-c}) of the file-drawer problem,
their power curves should behave similarly at least locally.
However, $p$-values may differ by comparing higher-order terms of power curves.
Also, there could be possibilities to improve the power by relaxing the approximate unbiasedness.
They are interesting future topics.

\section*{Acknowledgments}

The authors greatly appreciate many comments from our seminar audience at Department of Statistics, Stanford University.
This research was supported in part by JSPS KAKENHI Grant (16K16024 to YT, 16H02789 to HS).

%
%%
%%%------------------------------------------------------------------------------------------------------------------------------
%

%%%------------------------------------------------------------------------------------------------------------------------------
%%
%

\clearpage
\bibliographystyle{imsart-nameyear}
\bibliography{simbp_bib,stat2017}

%%%----------------------------------------------------------------------------------------------

\clearpage
\appendix

\renewcommand{\thetable}{\Alph{section}.\arabic{table}}
\renewcommand{\thefigure}{\Alph{section}.\arabic{figure}}
\renewcommand{\theequation}{\Alph{section}.\arabic{equation}}
\setcounter{table}{0}
\setcounter{figure}{0}
\setcounter{equation}{0}

%%%----------------------------------------------------------------------------------------------

\section{Pvclust details} \label{sec:pvclust-details}

\subsection{Gene sampling} \label{sec:pvclust-sampling}

An example of generative model for gene sampling is specified as follows.
We consider that $x_i$, $i=1,\ldots,n$, are independent observations of a random vector $X$ in $\mathbb{R}^p$.
Assume that there are $K$ gene classes, and $X$ is distributed as a mixture model with probability $\pi_k$, $k=1,\ldots, K$, say, the normal mixture $\sum_{k=1}^K \pi_k N_p(\eta_k, \Sigma_k)$.
For class $k=1,\ldots,K$, 
$E(X)=\eta_k$ represents average gene expressions, and
$V(X)=\Sigma_k$ represents observation noise and gene variation.

Let us examine the ``true'' clusters in this model.
As a very simple setting, we assume $\Sigma_k = \tau_k^2 I_p$ and
Euclidean distance $d_{ij} = \frac{1}{n}\sum_{t=1}^n (x_{ti} - x_{tj})^2$.
Then $E(d_{ij}) = \sum_{k=1}^K \pi_k \{ (\eta_{ki}-\eta_{kj})^2 + 2\tau_k^2 \}$ and $V(d_{ij})=O(n^{-1})$.
By taking the limit $n\to\infty$, $d_\infty$ is given by $E(d_{ij})$ above, which determines the ``true'' dendrogram and ``true'' clusters.
They can be poor representations of reality when all the contribution of $\tau_k^2$ is just observation noise.
In this case, the reality is best represented by $\sum_{k=1}^K \pi_k (\eta_{ki}-\eta_{kj})^2$ by setting $\tau_k^2=0$.
This issue is not considered in our testing procedures.

\subsection{Construction of $f_n$} \label{sec:pvclust-fn}

To find a connection between $\Xc_n$ and $y$, we would like to consider a specific form of transformation $y=f_n(\Xc_n)$ with $m+1 = p(p-1)/2$.
Let us assume the asymptotic normality $\sqrt{n}(d_n - d_\infty) \sim N_{p(p-1)/2}(0,\Sigma(d_\infty))$ for sufficiently large $n$, where $\Sigma(d_\infty)$ expresses the dependency of the covariance matrix on the underlying distribution.
This holds for the Euclidean distance and the correlation, and more generally smooth functions of the first and second sample moments of $x_1,\ldots,x_n$ when the fourth moments of $x_i$ exist so that the central limit theorem applies to the sample moments.
By defining
\[
	y =  f_n(\Xc_n) = \sqrt{n} \, \Sigma(d_\infty)^{-1/2}({\tt dist}(\Xc_n) - d_\infty),
\]
we have the normal model (\ref{eq:model}) approximately holds with $\mu=0$.
For local alternatives $d_\infty' = d_\infty + O(n^{-1/2})$, the model becomes
\begin{equation} \label{eq:model-local-alternatives}
	Y\sim N_{m+1}(\mu, I_{m+1} + \Delta V(d_\infty'))
\end{equation}
with $\mu =  \sqrt{n} \, \Sigma(d_\infty)^{-1/2}(d_\infty' - d_\infty)  = O(1)$ and $\Delta V(d_\infty') = \Sigma(d_\infty)^{-1/2} (
\Sigma(d_\infty') - \Sigma(d_\infty)  )
\Sigma(d_\infty)^{-1/2} = O(n^{-1/2})$.
In this paper, we ignore $\Delta V$ by approximating $\Delta V(d_\infty')=0$ in (\ref{eq:model-local-alternatives}) for developing the theory based on (\ref{eq:model}).
For bootstrap replicates, the asymptotic normality becomes
$\sqrt{n'}(d_{n'}^* - d_n) \mid \Xc_n \sim N_{p(p-1)/2}(0,\Sigma(d_n))$. By approximating $\Delta V=0$ again,
the transformed vector
\[
	y^* = f_n(\Xc_{n'}^*) = \sqrt{n} \, \Sigma(d_\infty)^{-1/2}({\tt dist}(\Xc_{n'}^*) - d_\infty)
\]
follows model (\ref{eq:bpmodel}) with $\sigma^2=n/n'$.

The regions in $\mathbb{R}^{m+1}$ must be considered too. For each cluster $G$,
the event $G \in {\tt hclust}(\Xc_n)$ corresponds to the event $y\in S_n(G)$ by defining
\[
 S_n(G) = \{y \mid G \in {\tt hclust}(n^{-1/2} \Sigma(d_\infty)^{1/2}y + d_\infty)   \}.
 \]
Thus the bootstrap probability of $G$ is $\alpha_{\sigma^2}(S_n(G) | y) = C(G)/B + O_p(B^{-1/2})$.
In this paper, the selective region is $S=S_n(G)$ and the hypothesis region is $H=S_n(G)^c$.
Another interesting choice of selective region would be $S=\bigcap_{G\in {\tt hclust}(\Xc_n)} S_n(G)$ for the dendrogram ${\tt hclust}(\Xc_n)$, but the bootstrap probability of the dendrogram may be too small (could be almost zero) so that our algorithm does not work well.

\subsection{Pvclust analysis of lung dataset} \label{sec:pvclust-experimenet}

The lung data set \citep{garber2001diversity} available in pvclust consists of micro-array expression profiles of $n=916$ genes for $p=73$ lung tissues.
The lung tissues include five normal tissues, one fetal tissue and 67 tumors from patient.
The original data had 918 genes, but two duplications (the last two genes) were removed.
We resample columns of $73 \times 916$ matrix $\Xc_n$ for generating $73 \times n'$ matrix $\Xc_{n'}^*$.
Sample sizes are $n' = $ 8244, 5716, 3963, 2748, 1905, 1321, 916, 635, 440, 305,  211, 146,  101; they are chosen so that $\sigma^2=n/n'$ values are placed evenly in log-scale from $1/9$ to $9$.
The number of bootstrap repetition is $B=10^4$.
Then Algorithm~\ref{alg:simbp} is performed on each cluster. 
The best fitting model from 4 candidates (poly.1, poly.2, poly.3 and sing.3 in Section~\ref{sec:models}) is selected by AIC.
For example, poly.2 is $\varphi(\sigma^2|\beta) = \beta_0 + \beta_1 \sigma^2$,
and poly.3 is $\varphi(\sigma^2|\beta) = \beta_0 + \beta_1 \sigma^2 + \beta_2 (\sigma^2)^2$.
Then $\psi_{\sigma^2}(H|y)$ is extrapolated to $\sigma^2\le 0$ using the selected model,
and $p_{\mathrm{AU},k}$ and $p_{\mathrm{SI},k}$ ($k=3$, $\sigma_0 = \sigma_{-1}=1$), as well as $p_\mathrm{BP}$, are computed by (B) of Step~4.
These $p$-values are denoted as $p_\mathrm{AU}$ and $p_\mathrm{SI}$ by omitting $k$ in Section~\ref{sec:pvclust-example}.
Computation of model fitting and extrapolation is based on the maximum likelihood estimation implemented in the scaleboot package of R.

Model fitting is shown for cluster id =  37, 57, 62, and 67 in Fig.~\ref{fig:pvclust-sbfit}.
Selected model is indicated in each panel. 
For each cluster, observed frequencies are given as follows.
$C_{S_{37}} = $ 10000, 10000,  9997,  9978,  9911,  9704,  9355,  8597,  7443,  6157,  4724,  3583,  2457.
$C_{S_{57}} = $ 9962,  9878,  9657,  9271,  8551,  7773,  6807,  5676, 4622,  3695,  2650 , 1955, 1381.
$C_{S_{62}} = $ 10000, 10000,  9999,  9995,  9963,  9841,  9635,  9181,  8464,  7616,  6742,  5635, 4605.
$C_{S_{67}} = $ 1374, 1095,   871,  674,   553,   471,   338,   280,   223,   136,    89,    71,   29.
$\psi_{\sigma^2}(H|y)=-\psi_{\sigma^2}(S|y)$ is extrapolated by $\varphi_{H,3}$ as follows (id =  37, 57, 62, and 67).
$\psi_{-1}(H|y) = $2.401, 1.583, 2.265, 1.657.
The signed distance $t \approx \psi_{0}(H|y) = $1.934, 1.008, 2.011, $-0.322$.
Therefore, the mean curvature is estimated as $-\hat\gamma \approx \psi_{-1}(H|y) - \psi_{0}(H|y) =$
0.487, 0.575, 0.254, 1.979.

Although $t$ should be positive for the selection event $y\in S$,
$t$ is wrongly estimated as negative for cluster id = 67.
In this case, the algorithm calculates $p_\mathrm{SI}>1$, and we set $p_\mathrm{SI}=1$.
For cluster id = 67, as indicated in the very small $C_{S_{67}}$ values as well as the large $-\hat\gamma$ value, the region $S$ is very small, and both the theories of Sections~\ref{sec:LST} and \ref{sec:NFT} do not work perfectly well.
Nevertheless, the large value of $p_\mathrm{SI}$ safely avoids rejecting the null hypothesis.

\subsection{Pvclust simulation details} \label{sec:pvclust-simulation-details}
\begin{figure}
 \begin{minipage}{0.32\textwidth}
  \centering
   \includegraphics[width=\textwidth]{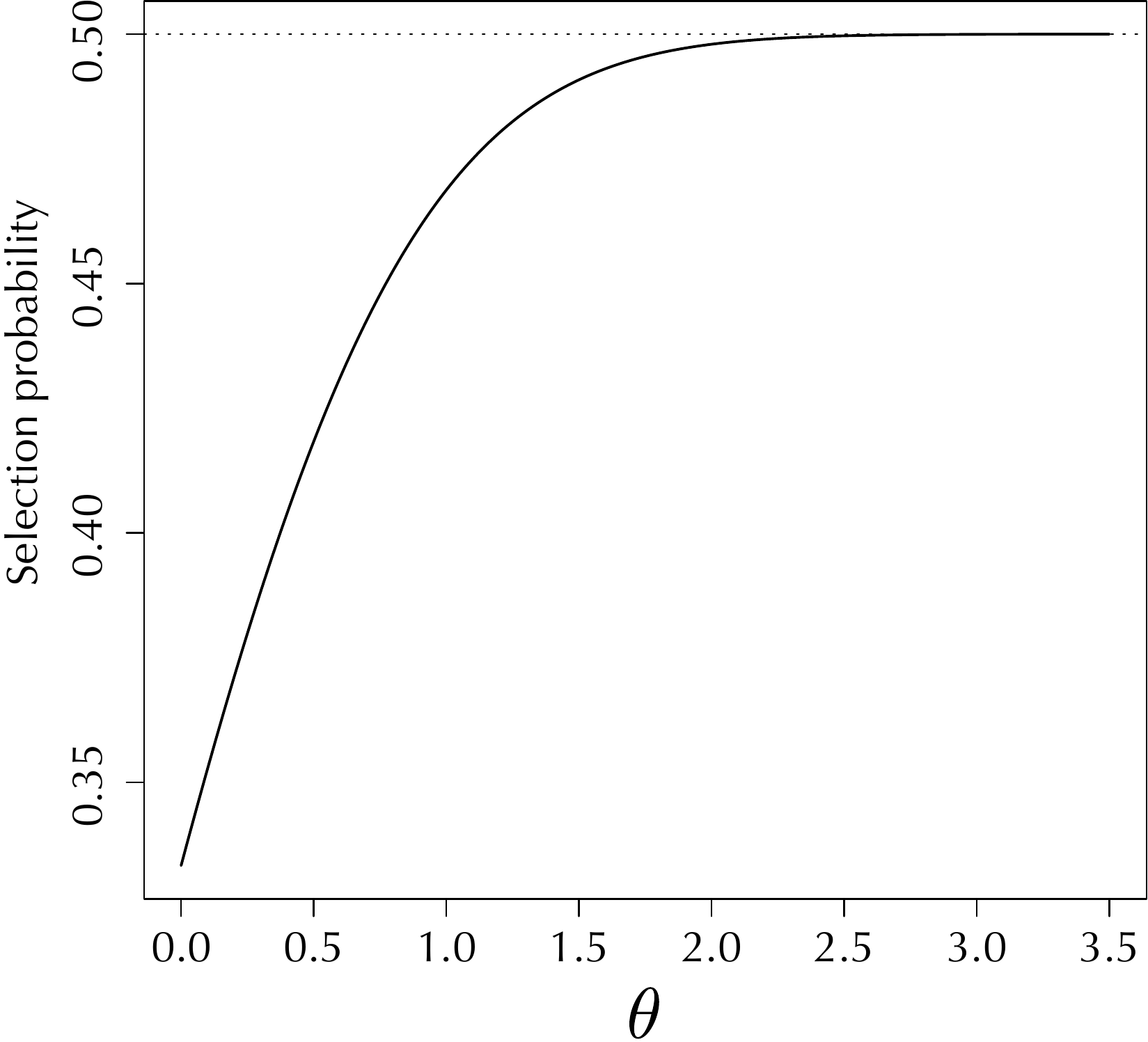}
   (a)
   %\subcaption{}%{Selection probability related with $\theta$}
 \end{minipage}
 \begin{minipage}{0.32\textwidth}
  \centering
   \includegraphics[width=\textwidth]{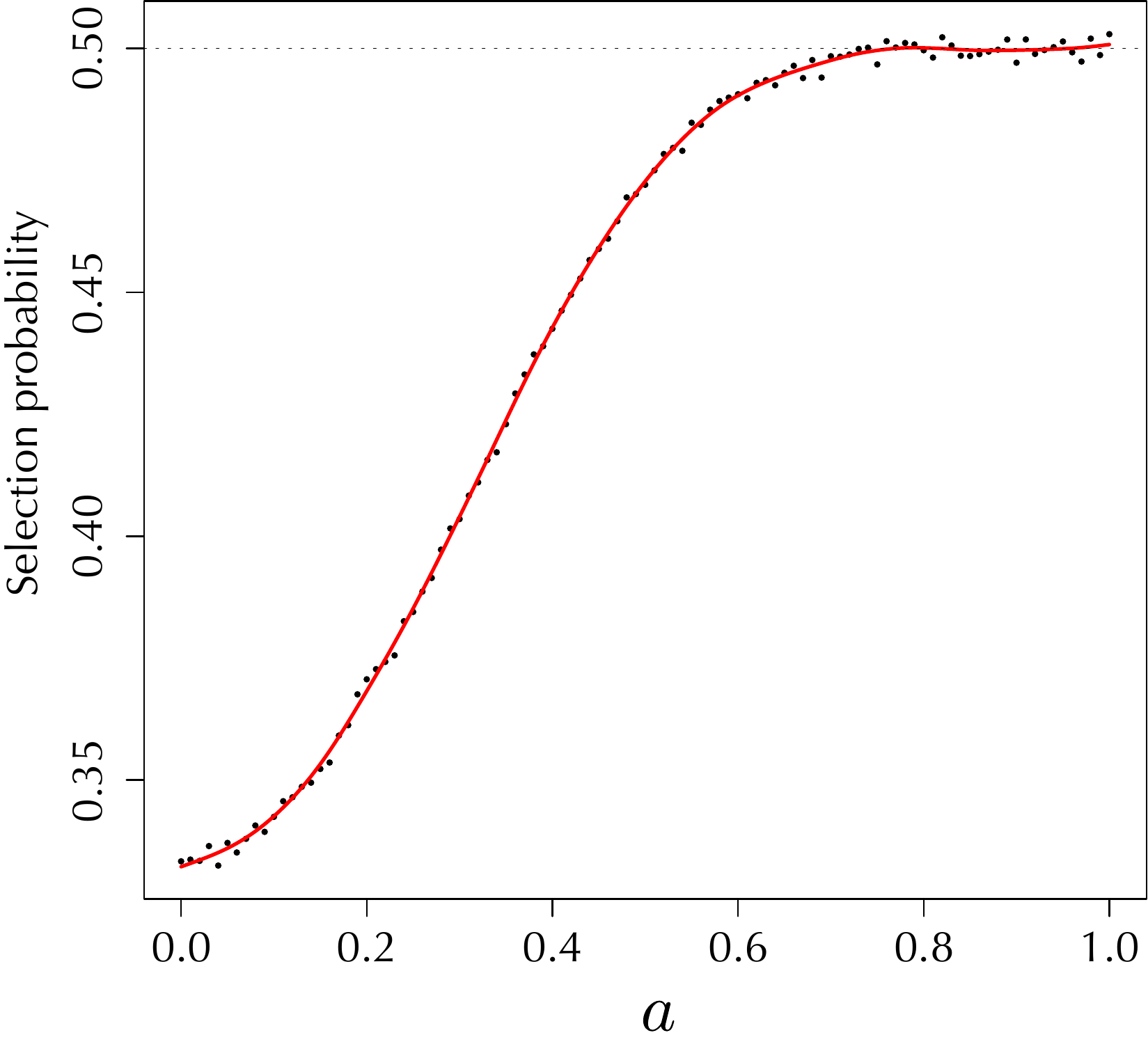}
   (b)
   %\subcaption{}%{Selection probability of Cluster $(1,2)$ related with $a$}
 \end{minipage}
  \begin{minipage}{0.32\textwidth}
  \centering
   \includegraphics[width=\textwidth]{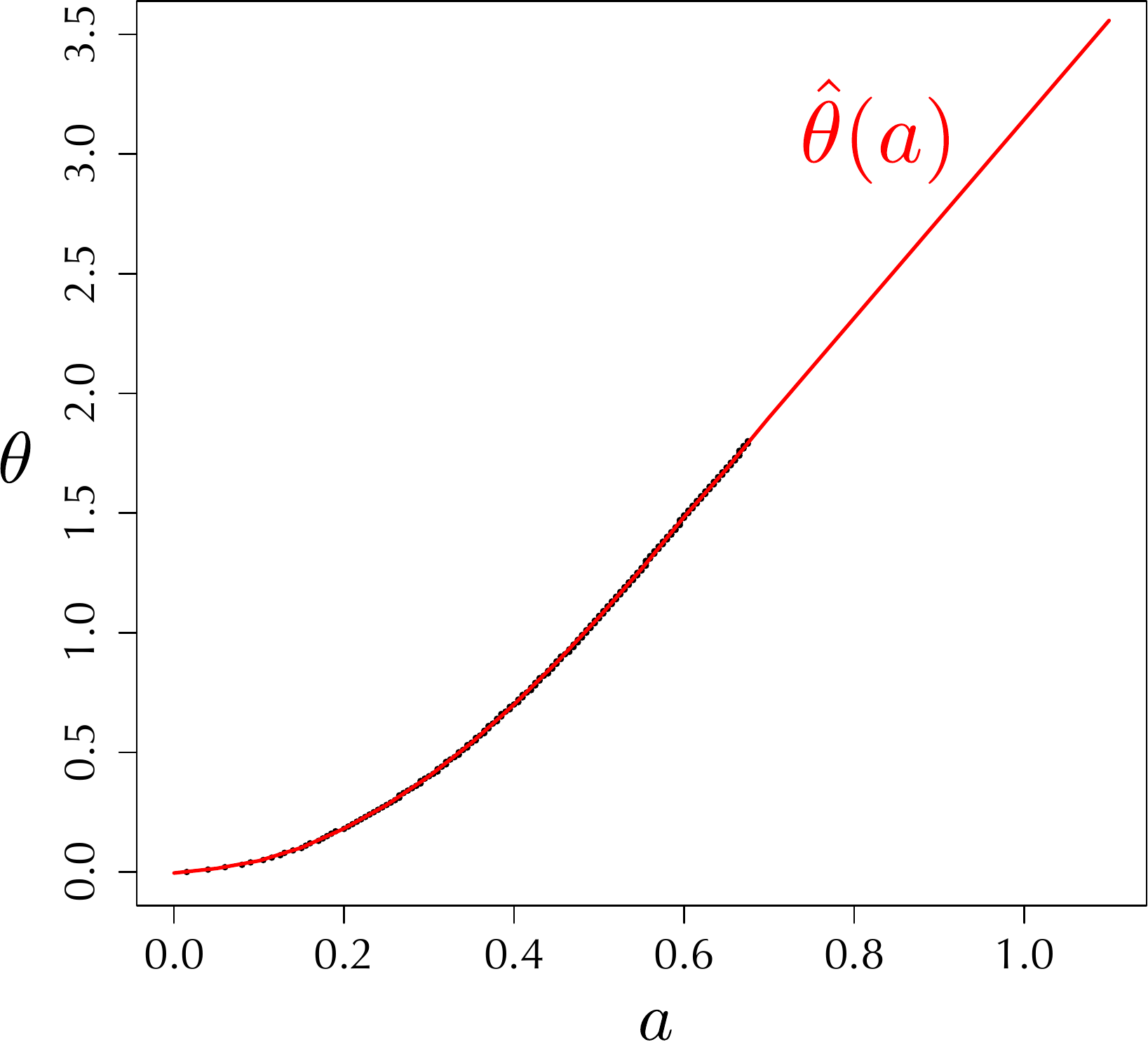}
   (c)
   %\subcaption{}%{Relationship between $\theta$ and $a$ }
    \end{minipage}
   \caption{Relationship between $\theta$ and $a$ through the selection probability.
  (a) Selection probability $P(Y\in S | \mu)$ vs. $\theta$ on the two-dimensional example 
  in which the hypothesis region is concave and  the  boundary surface is nonsmooth.
  (b) Selection probability $P(\{1,2\} | a)$ vs. $a$ on the pvclust simulation with
  its smoothing spline.
 % The red line is a estimated function using the smoothing spline.
  (c) The selection probabilities are matched by $\theta = \hat\theta(a)$.
   }
  \label{fig:sprob-theta-a}
\end{figure}

Here, we describe the details about the simulation of pvclust in Section~\ref{sec:pvclust-simulation}.
As described in Section~\ref{sec:pvclust-simulation}, 
we consider the clustering problem of three tissues.
There are three clusters $\{1,2\},\{1,3\},\{2,3\}$ with the exception of trivial clusters $\{1\},\{2\},\{3\}$ and $\{1,2,3\}$.
Let $d_{ij}=\sum_{s=1}^{1000}(x_{is}-x_{js})^2/1000$.
If $d_{12}<\max{d_{13},d_{23}}$, we observed the cluster $\{1,2\}$.
If $a=0$, the distribution of $(d_{12},d_{13},d_{23})$ is permutation invariant and thus the probability of each cluster is equal to $1/3$.
If $a>0$, $(d_{12},d_{13},d_{23})$ and $(d_{13},d_{12},d_{23})$ follow the same distribution 
and the probability of cluster $\{1,2\}$ is equivalent to one of cluster $\{1,3\}$.
Let $\Pb(\{i,j\} | a)$ be the probability that cluster $\{i,j\}$ occurs.
Since $\Pb(\{2,3\} | a)$ decreases as the value of $a$ increases, 
$\Pb(\{1,2\} | a)$ (or $\Pb(\{1,3\} | a)$) approaches $1/2$ as $a$ becomes large.
In pvclust, we test the null hypothesis that the cluster $\{i,j\}$ is not true only when the cluster $\{i,j\}$ is observed.
More formally, with the notation of \ref{sec:pvclust-fn}, 
the null hypothesis $\mu \in H=S_n(\{i,j\})^c$ is tested only when $y\in S=S_n(\{i,j\})$.
Hence, this pvclust example may correspond to the non-smooth and concave case of two dimensional examples 
in Section~\ref{sec:convex-concave-simulation}.
Specifically, 
$S_n(\{1,2\})^c$ (or $S_n(\{1,3\})^c$) is corresponding to the region $H=\{(u,v)\mid v\le |u|/3\}$ in the two dimensional example, 
and $S_n(\{1,2\})$ (or $S_n(\{1,3\})$) can be interpreted as the selective set $S=H^c$ in the two dimensional example.
In the two dimensional example, 
we chose $\theta=0.0,0.5,\dots,3.5$ and, for each parameter $\mu=(\theta,-h(\theta))\in \partial H$,
we computed the selection probability $\Pb(Y\in S \mid \mu)$ and the selective rejection probabilities with several $p$-values as shown in Table~\ref{table:concave}.
The subfigure (a) in Fig.~\ref{fig:sim_pvclust} is a visualization of Table~\ref{table:concave} 
and shows how the selective rejection probability of each method varies with $\theta$.
The value of $\theta$ represents the distance from the vertex, and $a$ in this simulation of pvclust is related to $\theta$.
The purpose of this simulation is to ensure that, even in more practical setting, 
the proposed method can reduce bias efficiently with distance from the vertex.
Note that the situation in this simulation (or in the two-dimensional example) approaches the flat case, 
in which the boundary surface $\partial H$ is flat, as $a$ (or $\theta$) becomes large.
To make it easier to compare with the result of two-dimensional example, 
first we need to find the values of $a$ which correspond to the values of $\theta$ in the two-dimensional example.
In this simulation, we focus on selective inferences for the null hypotheses that cluster $\{1,2\}$ (or $\{1,3\}$) is not true.
Here, for each $\theta$, 
we find the value of $a$ such that the selection probability, that is, $\Pb(\{1,2\} | a)$ (or $\Pb(\{1,3\} | a)$), 
is equivalent to $\Pb(Y\in S \mid \mu)$ in the two-dimensional example.
The subfigure (a) of Fig.~\ref{fig:sprob-theta-a} shows the relationship between $\theta$ and $\Pb(Y\in S \mid \mu)$.
Note that all probabilities in Table~\ref{table:concave} are computed accurately by numerical integration.
In this simulation of pvclust, it is difficult to use numerical integration and thus we employed the Monte-Carlo simulation.
For each value of $a=0.00,0.01,\dots,1.00$, we generate $10^5$ datasets, 
and the probability $\Pb(\{1,2\} | a)$ is estimated by the ratio of the number of times that cluster $\{1,2\}$ occurs to $10^5$, say $\hat{p}(a)$.
In the subfigure (b) of Fig.~\ref{fig:sprob-theta-a}, the black dots indicate pairs $(a,\hat{p}(a))$, 
and the red line is the estimated functional relationship between $a$ and the corresponding probability $\Pb(\{1,2\} | a)$ by the smoothing spline.
Based on the estimated relationship, in the sense of the selection probability, 
the values of $a$ corresponding to $\theta=0.0,0.5,\dots,3.5$ are given by $a=0.018, 0.337, 0.484, 0.604, 0.717, 0.822, 0.927, 1.031$.
In Section~\ref{sec:pvclust-simulation}, for each of these values, we compute the selective rejection probabilities for BP, AU, 2BP, 2AU, and SI.
The subfigure (c) of Fig.~\ref{fig:sprob-theta-a} shows the relationship between $a$ and $\theta$.

\section{Simulation details} \label{sec:convex-concave-simulation-details}

%convex
\begin{table}
\caption{Convex hypothesis regions : The results are in the same settings as in Table~\ref{table:concave}. }
\label{table:convex}
\scalebox{0.92}{
\begin{tabular}{lrrrrrrrr|r}
\hline
Smooth            & $\theta=0.0$  & $0.5$ & $1.0$ & $1.5$ & $2.0$ & $2.5$ & $3.0$ & $3.5$ & Bias\\
\hline
BP          		& 26.12		& 25.73		& 24.74 		& 23.57 		& 22.51		& 21.70		& 21.14 		& 20.76	& 13.27\\
AU ($k = 3$)     & 18.42  		& 18.44    		& 18.55    		& 18.81  		& 19.15   		& 19.46		& 19.69   		& 19.83	& 9.03\\
\hdashline
2BP          		& 13.82 		& 13.56		& 12.91 		& 12.17 		& 11.52 		& 11.03 		& 10.68 		& 10.46 	& 2.01\\
2AU ($k = 2$)     & 9.61    		& 9.55    		&   9.46    		& 9.42    		& 9.46    		& 9.55 		& 9.67    		&   9.76    & 0.46\\
2AU ($k = 3$)     &   9.22  		& 9.23    		&   9.28    		& 9.40    		& 9.57    		& 9.73 		& 9.85    		& $\bm{9.91}$ & 0.48\\
SDBP 		& 10.51		& 10.40		& $\bm{10.15}$	& $\bm{9.93}$	& 9.80		& 9.77		& 9.80		& 9.84 	& 0.23\\
SI ($k = 2$)      & $\bm{10.30}$ & $\bm{10.21}$	& $\bm{10.03}$	& $\bm{9.88}$	& $\bm{9.81}$	& $\bm{9.81}$	& $\bm{9.85}$ &   9.89    & $\bm{0.16}$\\
SI ($k = 3$)      & $\bm{9.99}$ 	& $\bm{9.95}$ 	& 9.89	 	& 9.86		& $\bm{9.89}$	& $\bm{9.94}$	& $\bm{9.98}$ &$\bm{10.00}$ & $\bm{0.07}$\\
\hdashline
$P(Y\in S\mid\mu)$    & 55.46 &   55.00   & 53.91 & 52.72 & 51.76 & 51.11 & 50.71 & 50.47 & -  \\
\hline\hline
Nonsmooth     & $\theta=0.0$  & $0.5$ & $1.0$ & $1.5$ & $2.0$ & $2.5$ & $3.0$ & $3.5$ & Bias\\
\hline
BP          		& 35.02		& 28.61		& 24.48		& 22.09		& 20.86		& 20.31		& 20.01 		& 20.03		& 13.29\\
AU ($k = 3$)     & 19.32  		& 16.49    		& 17.09    		& 18.66  		& 19.80  		& 20.26		& 20.29   		& 20.18		& 8.69\\
\hdashline
2BP          		& 20.09 		&  15.29 		&  12.57 		&  11.15 		&   10.46 		& 10.17 		& $\bm{10.05}$	& $\bm{10.01}$ & 2.08\\
2AU ($k = 2$)   &   11.48 		&   8.89    		&  8.39   		& 8.73    		&    9.22   		& 9.60    		& 9.82    		&  9.94   		& 0.80\\
2AU ($k = 3$)   & $\bm{9.91}$ 	&   8.11		&  8.46    		& $\bm{9.32}$ 	& $\bm{9.92}$ 	& $\bm{10.15}$ & 10.16 		&   10.10 		& 0.69\\
SDBP 		& 12.13		& 9.30		& 8.57		& 8.78		& 9.13		& 9.49		& 9.74		& 9.88		& 0.81\\
SI ($k = 2$)      &   12.76		 & $\bm{9.81}$ &$\bm{9.00}$  & 9.10    		& 9.42    		&   9.69    		& 9.86    		& $\bm{9.95}$	& $\bm{0.62}$\\
SI ($k = 3$)      & $\bm{11.34}$  & $\bm{9.05}$ & $\bm{8.95}$ & $\bm{9.48}$ 	& $\bm{9.91}$ 	&$\bm{10.09}$ & $\bm{10.10}$ &   10.06   	& $\bm{0.43}$\\
\hdashline
$P(Y\in S\mid\mu)$    & 66.67 & 58.17 & 53.13 & 50.92 & 50.20 & 50.03 & 50.00 & 50.00 & -\\
\hline
\end{tabular}
}
\end{table}

In this section, we describe the results of the convex case in Section~\ref{sec:convex-concave-simulation}.
First, we consider the following convex hypothesis regions $H$:
\ba
h(u)= \sqrt{a+u^2/3}\;\;(a=0,1),\quad H=\{(u,v)\mid v \le -h(u)\}.
\ea
Table~\ref{table:convex} is the results of the convex case, which is in parallel with Table~\ref{table:concave}.
From this results, the bias reduces effectively by using our method with distance from the vertex in the convex case.
Fig.~\ref{fig:sim_convex} shows the contour lines of $p$-values at level $\alpha=0.1$．
As in Fig.~\ref{fig:sim_concave}, as $\theta$ becomes large, the lines of 2BP, 2AU, and SI agree with each other．
From this, we confirm that the twice of non-selective $p$-values induce selective inference when $\partial H$ is flat.
In fact, $P(Y\in S|\mu)$ in Tables~\ref{table:convex} show that the selection probabilities are nearly $1/2$ at large $\theta$ values.
The selection probability, increasingly, approaches $2/3$ as $\theta$ goes to zero in the nonsmooth convex case.
The curve of $(3/2)p_{\mathrm{AU},3}=\alpha$ is very close to the curve of $p_{\mathrm{SI},3}=\alpha$ near the vertex.
Thus, also in the nonsmooth convex case, we can see that our method adjusts automatically the selection probability.
In addition, Fig.~\ref{fig:sim_convex_sdbp} shows that SI ($k=2$) and SDBP have similar rejection boundaries in both smooth and nonsmooth cases.
Actually, we can see that SI ($k=2$) and SDBP provide similar selective rejection probabilities in Table~\ref{table:convex}

\begin{figure}
 \begin{minipage}{0.49\textwidth}
  \begin{center}
   \includegraphics[width=0.9\textwidth]{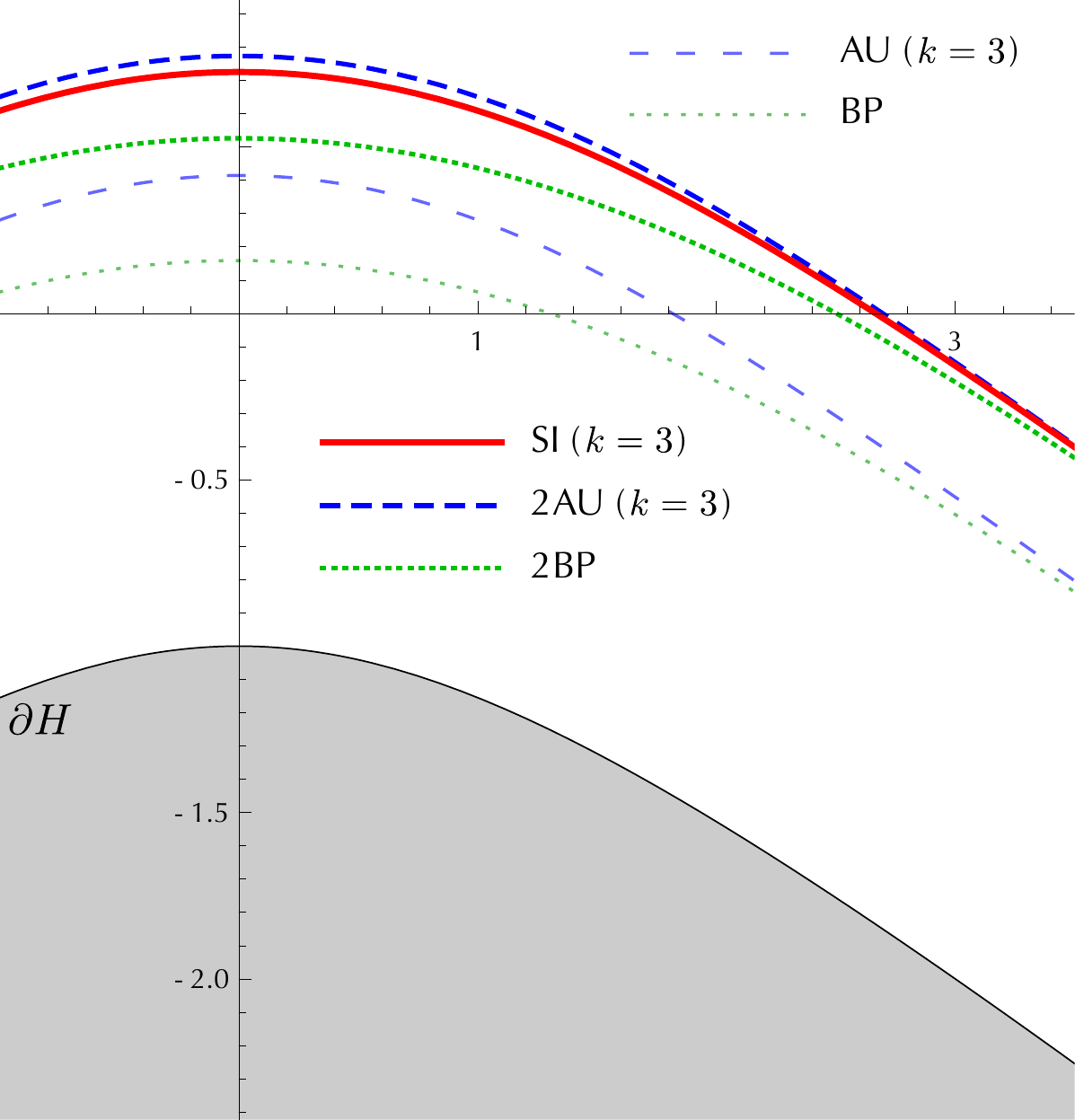}
   \hspace{1.6cm} (a) Smooth case : $a=1$
 \end{center}
 \end{minipage}
 \begin{minipage}{0.49\textwidth}
  \begin{center}
   \includegraphics[width=0.9\textwidth]{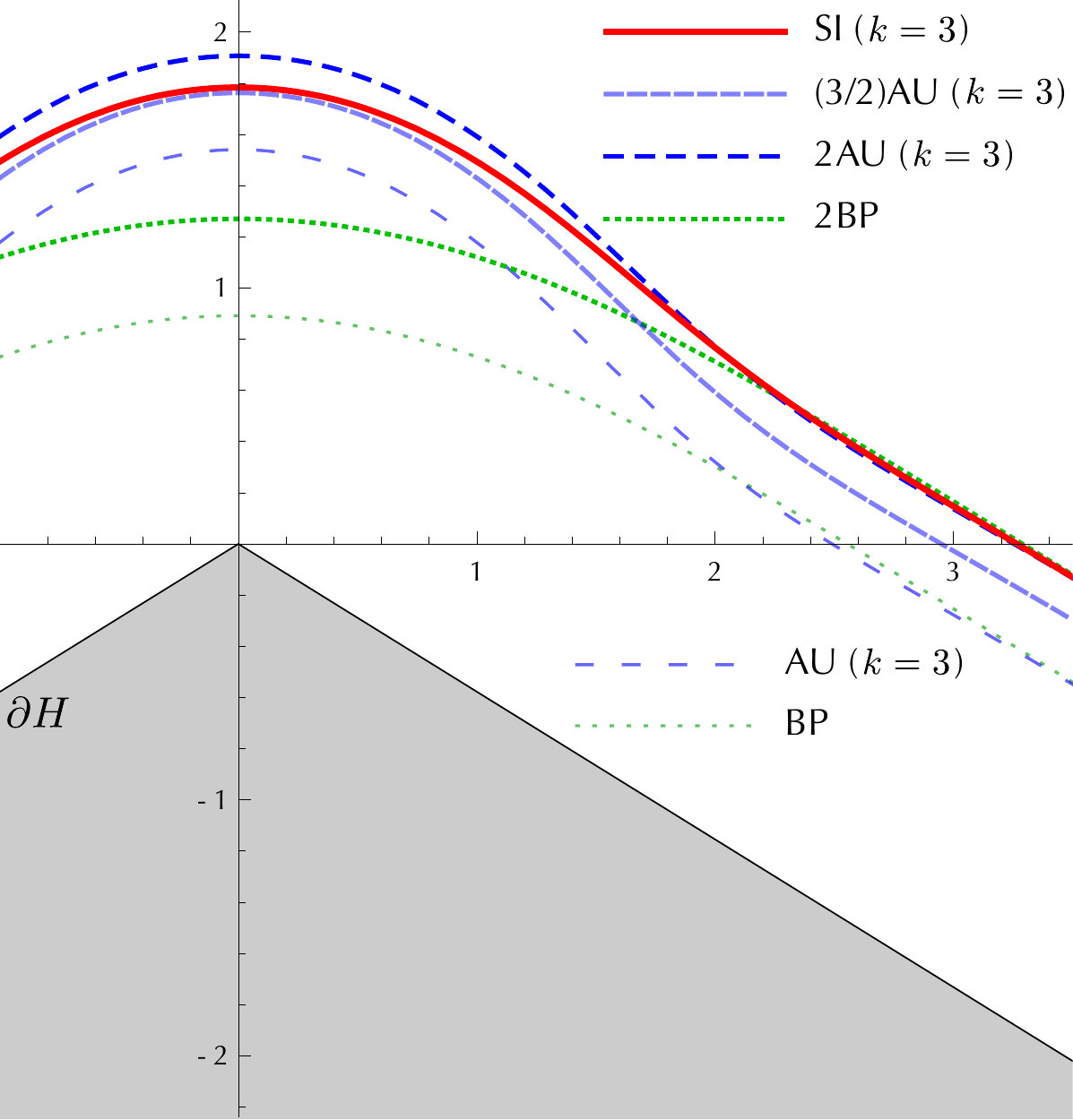}
   \hspace{1.6cm} (b) Nonsmooth case : $a=0$
 \end{center}
 \end{minipage}
   \caption{Convex hypothesis regions : Contour lines of $p$-values with $\alpha=0.1$.
   SI ($k=3$) : $p_{\mathrm{SI},3}(y)=\alpha$ (a solid line).
   (3/2)AU ($k=3$) : $3p_{\mathrm{BP},2}(y)/2=\alpha$  (a densely dashed line).
   2AU ($k=3$) : $2p_{\mathrm{AU},3}(y)=\alpha$ (a dashed line).
   2BP : $p_\mathrm{BP}=\alpha$ (a dotted line).
   AU ($k=3$) : $p_{\mathrm{AU},3}(y)=\alpha$ (a loosely dashed line).
   BP ($k=3$) : $p_\mathrm{BP}=\alpha$ (a loosely dotted line).}
  \label{fig:sim_convex}
\end{figure}

\begin{figure}
 \begin{minipage}{0.49\textwidth}
  \begin{center}
   \includegraphics[width=0.9\textwidth]{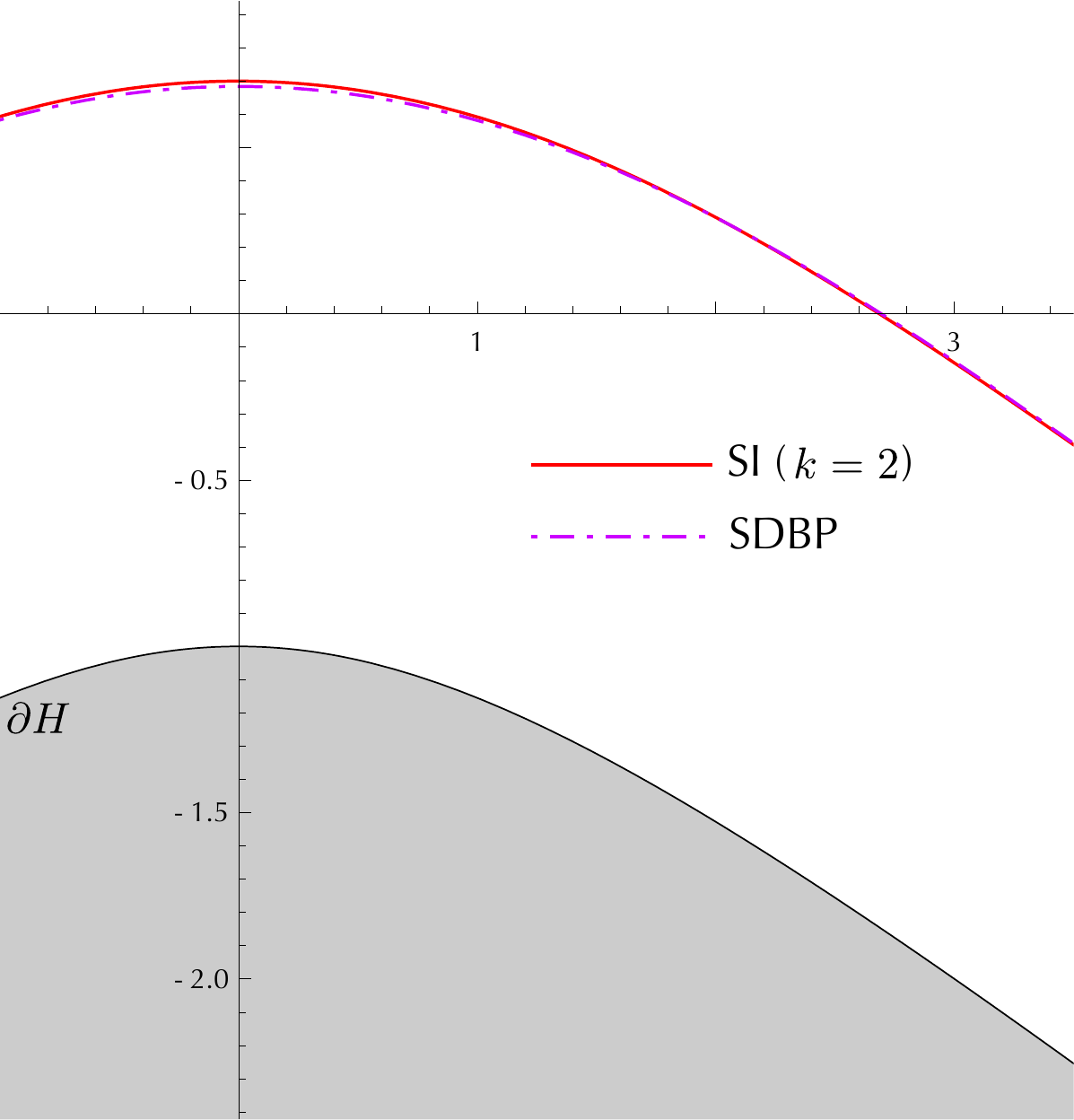}
   \hspace{1.6cm} (a) Smooth case : $a=1$
 \end{center}
 \end{minipage}
 \begin{minipage}{0.49\textwidth}
  \begin{center}
   \includegraphics[width=0.9\textwidth]{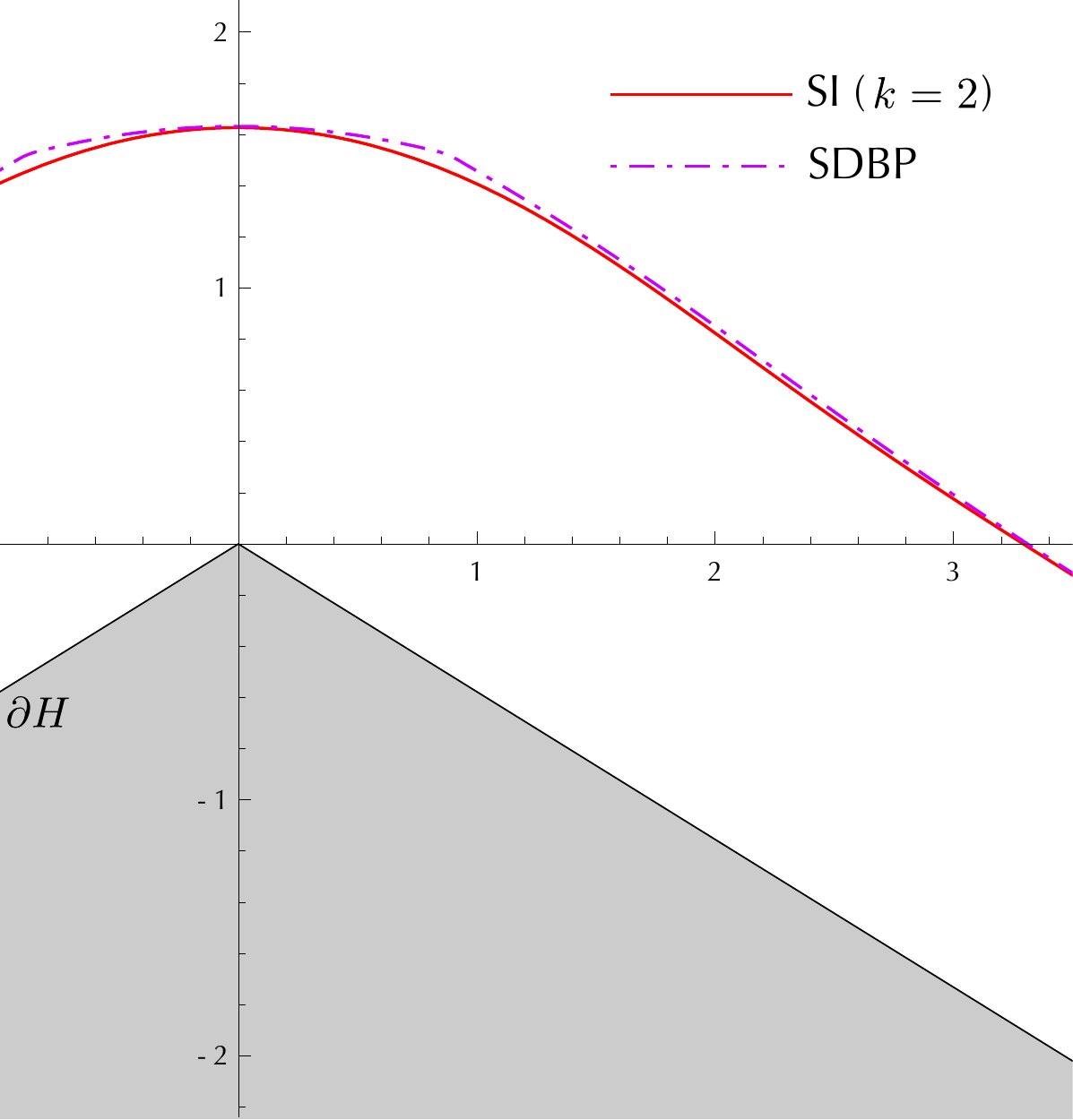}
   \hspace{1.6cm} (b) Nonsmooth case : $a=0$
 \end{center}
 \end{minipage}
   \caption{Relationship between SI ($k=2$) and SDBP : Contour lines of $p$-values with $\alpha=0.1$.
    SI ($k=2$) : $p_{\mathrm{SI},2}(y)=\alpha$ (a solid line).
    SDBP : $p_{\mathrm{BP},2}(y)=\alpha$  (a dashed-dotted line).}
  \label{fig:sim_convex_sdbp}
\end{figure}

\begin{figure}
 \begin{minipage}{0.32\textwidth}
  \centering
   \includegraphics[width=\textwidth]{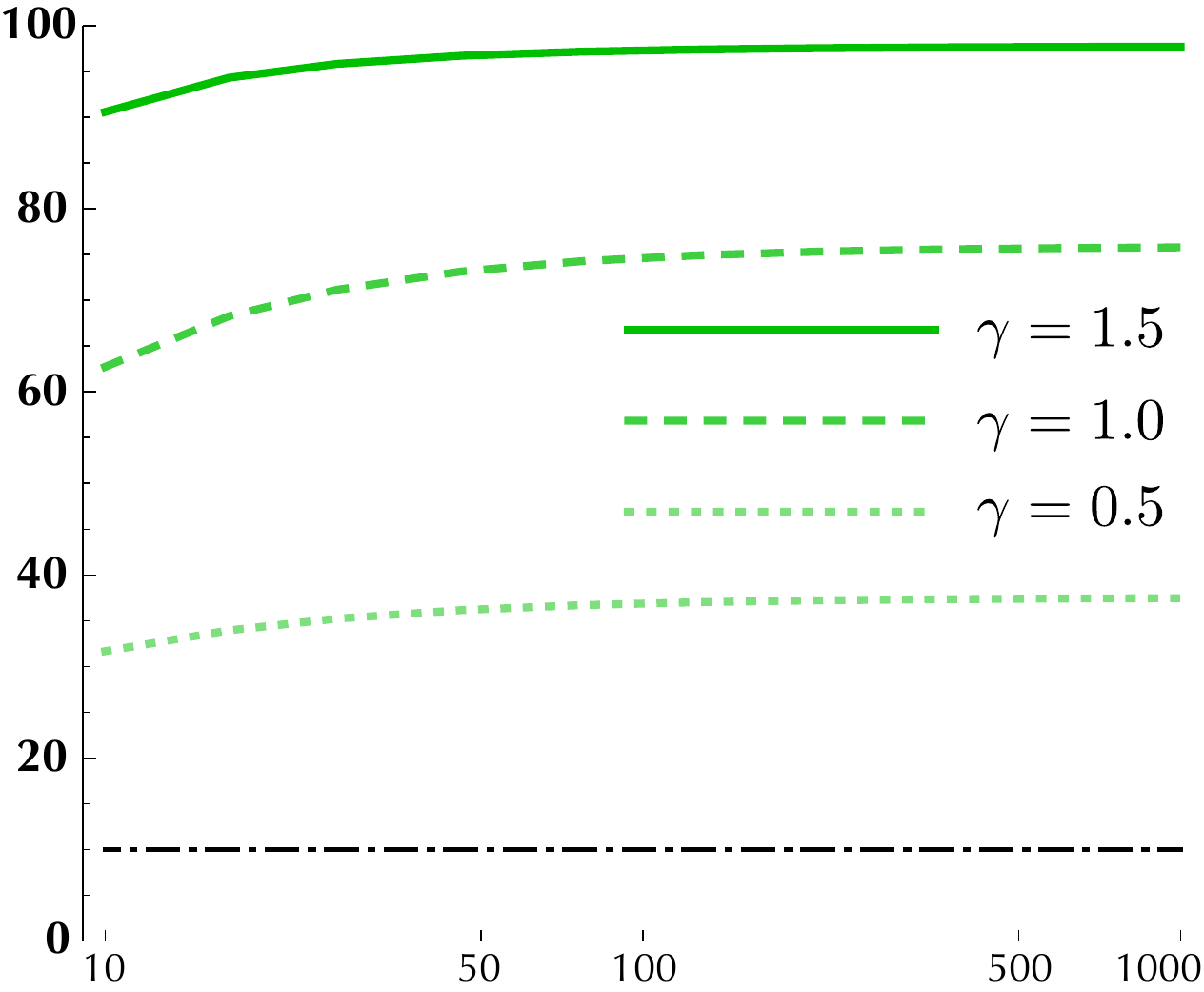}
   \subcaption{2BP}
 \end{minipage}
 \begin{minipage}{0.32\textwidth}
  \centering
   \includegraphics[width=\textwidth]{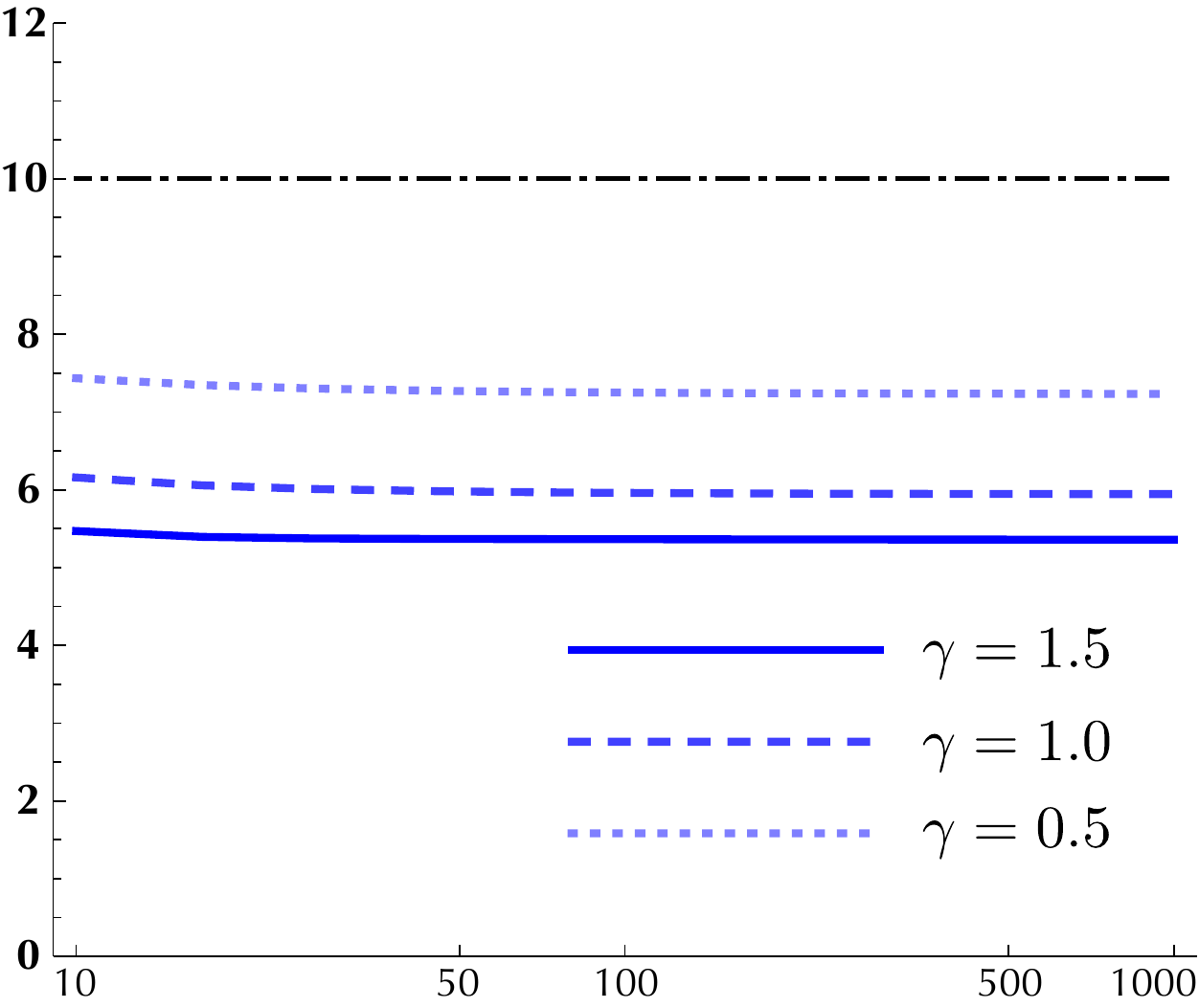}
   \subcaption{2AU ($k=3$)}
 \end{minipage}
  \begin{minipage}{0.32\textwidth}
  \centering
   \includegraphics[width=\textwidth]{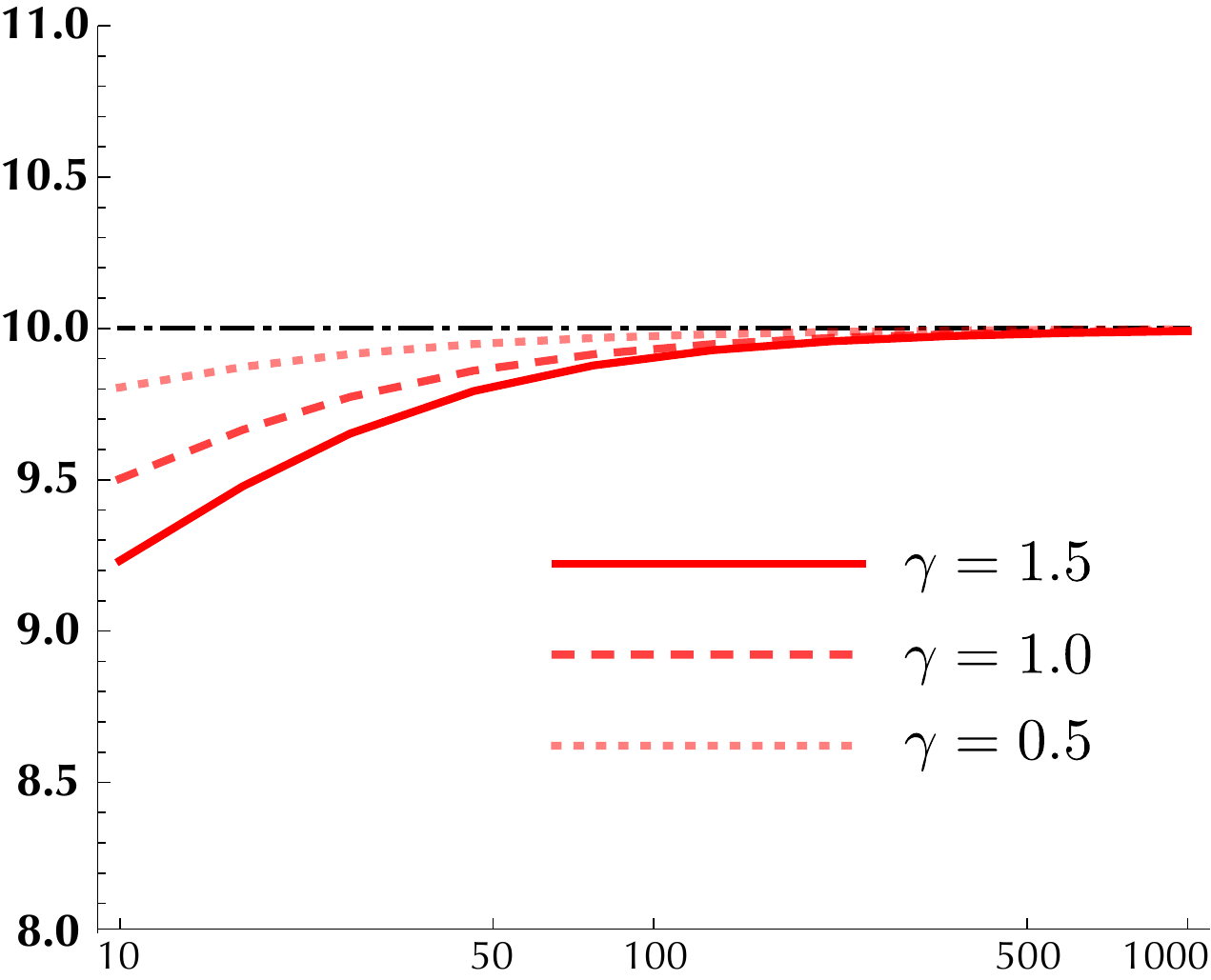}
   \subcaption{SI ($k=3$)}
    \end{minipage}
   \caption{Convex hypothesis regions : selective rejection probabilities as a function of the number of dimensions $m+1$. 
	The axes have the same meaning as in Fig.~\ref{fig:sim_sphere_concave}
   ($\gamma=0.5$ : a dotted line, $\gamma=1.0$ : a broken line, $\gamma=1.5$ : a solid line).
   The dashed-dotted lines is the line of the ideal unbiased selective test as with Fig.~\ref{fig:sim_sphere_concave}.
   }
  \label{fig:sim_sphere_convex}
\end{figure}

Next, we show the results for the convex case of the spherical example which is originally considered in Example~1 of \cite{Efron:Tibshirani:1998:PR}.
Suppose that $H=\{\mu \mid \|\mu\|\le \theta\}$ and $S=H^c$ as a convex hypothesis region.
That is, we consider the case that the hypothesis region is a sphere of radius $\theta$ in $\Rb^{m+1}$.
All the other setting is same as in the spherical example of Section~\ref{sec:convex-concave-simulation}.

Fig.~\ref{fig:sim_sphere_convex} shows
the change of the selective rejection probability as the number of dimensions increases.
Whereas 2BP and 2AU ($k$) have serious bias related to the magnitude of mean curvature,
the selective rejection probabilities of SI ($k$) approach $\alpha=10\%$ as the number of dimensions increases.

%
%%
%%%=================================================
\section{Proofs for the large sample theory} \label{sec:appendix-LST}

We give details of the large sample theory in this section.

\subsection{Preliminary} \label{sec:lemma-LST}

First we give Lemma~\ref{lemma:change-of-coordinates} and its proof, which provides the formula of change of coordinates for projections.
This result is shown in Lemma~3 of \cite{Shimodaira2014higherorder} with fourth order accuracy for class $\Sc$, but
the result is very much simplified here with the second order accuracy for class $\Tc$.
We consider the basis of local coordinates at $(u,-h(u))$: $f(h,u)$, $b_1(h,u)$, $\ldots$, $b_m(h,u) \in \mathbb{R}^{m+1}$.
$f(h, u)$ is the normal vector to $\Bc(h)$ with $i$-th element $\partial h/\partial u_i$, $i=1,\ldots,m$ and
$(m+1)$-th element 1.
$b_i(h,u)$ is the tangent vector to  $\Bc(h)$ with $i$-th element 1, $(m+1)$-th element $-\partial h/\partial u_i$, and all other elements zero.
$f(h,u)$ and $b_i(h,u)$ are orthogonal to each other, and the inner product is $f(h,u) \cdot b_i(h,u) = 0$.

\begin{lemma} \label{lemma:change-of-coordinates}
For any $h,\eta \in \Tc$, we consider a shift of a point $(u, -h(u)) \in \Bc(h)$ to the normal direction $f(h,u)$ with signed distance $\eta(u)$. 
The new point is $(\theta, -s(\theta))$. The function $s(\theta)$, $\theta\in\mathbb{R}^m$ is defined by
\begin{equation} \label{eq:htos}
 (\theta, -s(\theta)) = (u, -h(u)) + f(h,u) \| f(h,u) \|^{-1} \eta(u).
\end{equation}
The change of coordinates $u\leftrightarrow \theta$ is given by
$\theta_i(u) \doteq u_i + \eta_0 h_i + 2 \eta_0 h_{ij} u_j = u_i + O(n^{-1/2})$ and
$u_i(\theta) \doteq \theta_i - \eta_0 h_i - 2 \eta_0 h_{ij} \theta_j = \theta_i + O(n^{-1/2})$.
$s$ is given by
$s(\theta) \doteq h_0 - \eta_0 +(h_i - \eta_i) \theta_i + (h_{ij} - \eta_{ij}) \theta_i \theta_j$
with coefficients
\begin{equation} \label{eq:hlambdatos}
	s_0 = h_0 - \eta_0,\quad s_i = h_i - \eta_i, \quad s_{ij} = h_{ij} - \eta_{ij},
\end{equation}
so $s\in \Tc$. Conversely, for any $h, s\in\Tc$, $\eta(u)$ with coefficients
\begin{equation} \label{eq:hstolambda}
	\eta_0 = h_0 - s_0,\quad \eta_i = h_i - s_i, \quad  \eta_{ij} = h_{ij} - s_{ij}
\end{equation}
satisfies (\ref{eq:htos}), and $\eta\in \Tc$.
Therefore, shift of surfaces in $\Tc$ can formally be treated as simple differences
$s(\theta) \doteq h(\theta) -\eta(\theta)$, $\eta(u) \doteq h(u)-s(u)$
by ignoring $O(n^{-1})$ terms of tilting of the normal vector.

\end{lemma}
\begin{proof}
Since $\partial h/\partial u_i = h_i + 2h_{ij} u_j = O(n^{-1/2})$, and $(\partial h/\partial u_i)^2 \doteq 0$,
we have 
$\| f \|^2 = 1 + \sum_{i=1}^m (\partial h/\partial u_i)^2 \doteq 1 $ and $\|f\|^{-1} \doteq 1$.
Then $\theta_i(u)$ is obtained from the $\theta_i$ element of (\ref{eq:htos}) as
$\theta_i = u_i + \eta(u) \|f\|^{-1} \partial h/\partial u_i \doteq u_i + (\eta_0 + O(n^{-1/2})) (h_i + 2h_{ij} u_j)$.
Conversely, substituting $u_i(\theta)$ into it, we verify that $u_i(\theta)$ is correct.
The $v$ element of (\ref{eq:htos}) gives
$s(\theta) = h(u) - \eta(u) \|f\|^{-1} \doteq h(u) - (\eta_0 + \eta_i u_i + \eta_{ij} u_i u_j) (1+O(n^{-1})) 
\doteq (h_0 - \eta_0) + (h_i - \eta_i) u_i + (h_{ij} -\eta_{ij}) u_i u_j
$. Substituting $u_i = \theta_i + O(n^{-1/2})$ into it, we get (\ref{eq:hlambdatos}).
Conversely, given $h, s\in \Tc$, we substitute $\eta(u)$ of (\ref{eq:hstolambda}) into (\ref{eq:htos}) and follow the calculation so far, we verify that $\eta(u)$ is the solution.
\end{proof}

The following trivial lemma will be repeatedly used.
\begin{lemma}\label{lemma:phi-phi-taylor}
$x_1,x_2\in\mathbb{R}$, and for sufficiently small $\epsilon_1, \epsilon_2\in\mathbb{R}$, we have
\begin{equation}
  \frac{\bar\Phi(x_1 + \epsilon_1)} {\bar\Phi(x_2 + \epsilon_2)} = \frac{\bar\Phi(x_1 + \epsilon_3)}{\bar\Phi(x_2)} + O(\epsilon_1^2 + \epsilon_2^2),
\end{equation}
where $\epsilon_3 = \epsilon_1 - A(x_1,x_2) \epsilon_2$ and
%$  A(x_1,x_2) = \frac{\bar\Phi(x_1)}{\phi(x_1)} \frac{\phi(x_2)}{\bar\Phi(x_2)} $.
$  A(x_1,x_2) = (\bar\Phi(x_1)\phi(x_2))/(\phi(x_1)\bar\Phi(x_2))$.

\end{lemma}
\begin{proof}
By applying the Taylor expansion (\ref{eq:phi-taylor}) to the both sides, and arranging the formula, we immediately get the result.
\end{proof}

\subsection{Proof of Lemma~\ref{lemma:geometry-of-bp}} \label{sec:proof-lemma-geometry-of-bp}
The change of coordinates is obtained from Lemma~\ref{lemma:change-of-coordinates}
by letting $y=(\theta, -s(\theta))$, $\proj(H|y) = (u,-h(u))$.
Substituting $u_i = \theta_i + O(n^{-1/2})$ into $\hat\eta(\theta) = \eta(u(\theta))$, we have
$\hat\eta(\theta) \doteq \eta_0 + \eta_i \theta_i + \eta_{ij} \theta_i \theta_j$, so
we get (\ref{eq:hat-lambda}) from (\ref{eq:hstolambda}).
Next, for showing (\ref{eq:hat-gamma}),
we consider the local coordinates $(\Delta \tilde u, \Delta \tilde v)$ at $(u,-h(u))$
with the basis $b_1(h,u),\ldots,b_m(h,u),f(h,u)$.
The surface $\Bc(h)$ is expressed as $\Delta \tilde v = -\tilde h(\Delta \tilde u)$.
The same argument is given in Lemma~1 and Lemma~2 of \cite{Shimodaira2014higherorder} but symbols $\Delta u$ and $\Delta \tilde u$ are exchanged.

%\begin{figure}
%\centering
%\includegraphics[width=0.75\textwidth]{./fig_draft/local-coord.pdf}
%\caption{Local coordinates $(\Delta \tilde u, \Delta \tilde v)$ at $(u,-h(u))$}
%\label{fig:local-coord}
%\end{figure}

For solving the equation
\[
(u+\Delta u, -h(u+\Delta u)) = (u,-h(u)) + b_i(h,u) \Delta \tilde u_i - \tilde h(\Delta \tilde u) \|f(h,u)\|^{-1} f(h,u),
\]	
we note that $ \|f(h,u)\|^{-1}  \doteq 1$, and then we get
$\Delta u_i \doteq \tilde\Delta u_i$, $\tilde h(\Delta \tilde u) \doteq h_{ij} \Delta \tilde u_i \Delta \tilde u_j$
by comparing each element.
Since $b_i(h,u)\cdot b_j(h,u) \doteq \delta_{ij}$,
$b_1(h,u),\ldots, b_m(h,u)$ form the orthonormal basis of the tangent space with the second order accuracy.
Therefore, the mean curvature of $\Delta \tilde v \doteq  h_{ij} \Delta \tilde u_i \Delta \tilde u_j$ is $h_{ii}$, which
proves (\ref{eq:hat-gamma}).

\subsection{Proof of Theorem~\ref{theorem:au-surface}} \label{sec:proof-theorem-au-surface}
We show the rest of the proof here.
Applying the Taylor expansion (\ref{eq:phi-taylor}) to the both sides of (\ref{eq:au-surface-proof1}),
\begin{equation} \label{eq:au-surface-proof2}
\begin{split}
	& \bar\Phi(h_0 - r_0) - \phi(h_0 - r_0) \{ (h_i - r_i) \theta_i + (h_{ij} - r_{ij}) \theta_i \theta_j  - r_{ii}\} \\
   \doteq
	\alpha & \bar\Phi(h_0 - s_0) - \alpha \phi(h_0 - s_0) \{ (h_i - s_i) \theta_i + (h_{ij} - s_{ij}) \theta_i \theta_j  - s_{ii}\}.
\end{split}
\end{equation}
By comparing the coefficients of the terms of $\theta_i$ and $\theta_i \theta_j$, we get
\[
	r_i \doteq h_i - \alpha C^\prime (h_i - s_i),\quad
	r_{ij} \doteq h_{ij} - \alpha C^\prime (h_{ij} - s_{ij}),
\]
where $C^\prime = \phi(h_0 - s_0)/\phi(h_0 - r_0)$.
From the constant term,
\begin{equation} \label{eq:au-surface-proof3}
	 h_0 - r_0 = \bar\Phi^{-1}(\alpha \bar\Phi(h_0 - s_0)) + O(n^{-1/2}),
\end{equation}
which implies $C^\prime=C+O(n^{-1/2})$.
Thus we get the formula of $r_i$ and $r_{ij}$ in the theorem.
For showing $R\subset S$, note that $r_0 < s_0$ for $0<\alpha<1$ by ignoring $O(n^{-1/2})$ in (\ref{eq:au-surface-proof3}).
Since $s(u)-r(u) = s_0 - r_0 + O(n^{-1/2})$, we have $\lim_{n\to\infty} (s(u)-r(u))>0$.
For showing (\ref{eq:au-surface-alpha}), substitute $r_{ii} \doteq h_{ii} - \alpha C (h_{ii} - s_{ii})$ into (\ref{eq:au-surface-proof2}), and rearranging the constant term, we get
$\bar \Phi(h_0 - r_0) + \phi(h_0 - r_0) h_{ii} \doteq
	\alpha  \bar\Phi(h_0 - s_0) +  \alpha \phi(h_0 - s_0) h_{ii}$.
By applying (\ref{eq:phi-taylor}) to it, we get $\bar\Phi(h_0 - r_0 - h_{ii}) \doteq \alpha \bar\Phi(h_0 - s_0 - h_{ii}) $,
 proving (\ref{eq:au-surface-alpha}), and also the formula of $r_0$ as well.
We had assumed that $r\in \Tc$ in the beginning, and the obtained $r$ is in fact $r\in \Tc$.
By substituting this $r$ into (\ref{eq:selective-rejection-probability}) and follow the calculation so far,
we verify that it is the solution of (\ref{eq:selective-rejection-probability}).

\subsection{Proof of Theorem~\ref{theorem:au-pvalue}} \label{sec:proof-theorem-au-pvalue}
We show the rest of the proof here.
Let us define $r$ by considering the surface $\Bc(r) = \{ y \mid p_\mathrm{SI}(H | S,y)=\alpha\}$ for the $p$-value given in (\ref{eq:au-pvalue}).
We will verify that this $r$ coincides with the $r$ of $R$ in Theorem~\ref{theorem:au-surface}.
$\Bc(r)$ is interpreted as the surface obtained by shifting points $(u,-h(u))$, $u\in\mathbb{R}^m$ on $\Bc(h)$ to the normal direction by a signed distance $\eta(u)$, which is defined below.
From (\ref{eq:au-surface-alpha}), $y\in \partial R$ at $u=0$, so 
$\eta(0) \doteq h_0 - r_0 \doteq \bar\Phi^{-1} (\alpha \bar\Phi(h_0 - s_0 - h_{ii})) + h_{ii}$.
$\eta(u)$, $u\in\mathbb{R}^m$, is obtained by replacing the geometric quantities in $\eta(0)$ by those at $(u,-h(u))$.
$h_0-s_0$ is the signed distance from $\Bc(h)$ to $\Bc(s)$, so it is replaced by $h(u)-s(u)$ according to Lemma~\ref{lemma:change-of-coordinates}.
$h_{ii}$ is the mean curvature, and it is replaced by $h_{ii}+O(n^{-1})$ according to Lemma~\ref{lemma:geometry-of-bp}.
We then have $\eta(u) \doteq \bar\Phi^{-1} (\alpha \bar\Phi(h(u) - s(u) - h_{ii})) + h_{ii}$.
This is rearranged as $\eta(u) \doteq \eta(0) + \alpha C(h_i - s_i)u_i + \alpha C (h_{ij} - s_{ij}) u_i u_j$ by Taylor expansion. 
Since $r(u) \doteq h(u) - \eta(u)$ from Lemma~\ref{lemma:change-of-coordinates},
by comparing the coefficients, we verify that this $r(u)$ coincides with that in
Theorem~\ref{theorem:au-surface} with error $O(n^{-1})$.

\subsection{Proof of Theorem~\ref{theorem:double-bootstrap}} \label{sec:proof-theorem-double-bootstrap}

First we give the expression for $  p_{\mathrm{BP},1}(H|S,y) $.
From (\ref{eq:psi-hsq}),  the numerator is $\bar\Phi(\psi_{\sigma^2}(H|y)) \doteq \bar\Phi(\psi_{\sigma^2}(h|r,0)) = \bar\Phi(h_0 - r_0 + h_{ii}\sigma^2$).
Since $h(u_0) \doteq h_0$, the denominator is $P_1(Y^* \in S \mid \hat \mu) \doteq \alpha_1(\Rc(s)^c | (u_0, -h_0)) \doteq \bar\Phi(-\psi_1(s|h,u_0)) \doteq \bar\Phi(h_0 - s_0 - s_{ii})$.
Thus we get $  p_{\mathrm{BP},1}(H|S,y) $ in (\ref{eq:p-bp-1-2-lst}).

Next we give the expression for $  p_{\mathrm{BP},2}(H|S,y) $.
Let $y=(\theta,-r(\theta))$ and derive $r(\theta)$ so that $p_{\mathrm{BP},1}(H|S,y)$ takes a constant value.
The numerator is $\bar\Phi(\psi_{\sigma^2}(H|y)) \doteq \bar\Phi(\psi_{\sigma^2}(h|r,\theta)) = \bar\Phi(h_0 - r_0 + (h_i - r_i)\theta_i + (h_{ij} - r_{ij})\theta_i \theta_j + h_{ii}\sigma^2$).
By adding the error of $O(n^{-1/2})$ to the $u$-axis of $\proj(H|y)$,
we write $\hat\mu' \doteq (\theta+O(n^{-1/2}), -h(\theta)+O(n^{-1}))$.
Then the denominator is
$P_1(Y^* \in S \mid \hat \mu') = \alpha_1(\Rc(s)^c | \hat\mu') \doteq \bar\Phi(-\psi_1(s|h,\theta+O(n^{-1/2}))) \doteq \bar\Phi(h_0 - s_0 + (h_i - s_i)\theta_i + (h_{ij} - s_{ij}) \theta_i \theta_j - s_{ii})$.
We then rearrange the expression of $p_{\mathrm{BP},1}(H|S,y)$  using Lemma~\ref{lemma:phi-phi-taylor}
with $x_1 = h_0-r_0$, $x_2=h_0 - s_0$, $\epsilon_1 =  (h_i - r_i)\theta_i + (h_{ij} - r_{ij})\theta_i \theta_j + h_{ii}\sigma^2$, 
and $\epsilon_2 = (h_i - s_i)\theta_i + (h_{ij} - s_{ij}) \theta_i \theta_j - s_{ii}$.
We have 
$\epsilon_3 = h_{ii}\sigma^2 + A s_{ii} + (h_i - r_i - A(h_i - s_i)) \theta_i + (h_{ij} - r_{ij} - A(h_{ij} - s_{ij}) \theta_i \theta_j$ with $A = A(h_0-r_0,h_0-s_0)$.
Then we get
\[
  p_{\mathrm{BP},1}(H|S,(\theta,-r(\theta))) \doteq 
  \frac{\bar\Phi(h_0 - r_0 + \epsilon_3)}{\bar\Phi(h_0 - s_0)},
\]
which should be a constant value. Thus $\epsilon_3$ should take a constant value for any $\theta\in\mathbb{R}^m$.
By letting the coefficients of $\theta_i$ and $\theta_i \theta_j$  be zero in $\epsilon_3$,
we get
$r_i \doteq (1-A)h_i + As_i$ and $r_{ij} \doteq (1-A) h_{ij} + A s_{ij}$.

Using this $r$, we compute the numerator of $p_{\mathrm{BP},2}(H|S,y)$ at $y=(0,-r_0)$.
Then, $P_1( p_{\mathrm{BP},1}(H|S,Y^*) <  p_{\mathrm{BP},1}(H|S,y)  \mid \hat\mu) \doteq \alpha_1(\Rc(r)^c | (u_0,-h_0)) \doteq \bar\Phi(-\psi_1(r|h,u_0)) \doteq \bar\Phi(h_0 - r_0 - r_{ii}) \doteq \bar\Phi(h_0 - r_0 - (1-A)h_{ii} - As_{ii})$.
The denominator is the same as $p_{\mathrm{BP},1}(H|S,y)$.
Therefore, by applying Lemma~\ref{lemma:phi-phi-taylor}, we have
\[
    p_{\mathrm{BP},2}(H|S,y) \doteq \frac{\bar\Phi(h_0 - r_0 - (1-A)h_{ii} - As_{ii})}{\bar\Phi(h_0 - s_0 - s_{ii})}
    \doteq  \frac{\bar\Phi(h_0 - r_0 - h_{ii})}{\bar\Phi(h_0 - s_0 - s_{ii} + (s_{ii}-h_{ii}))},
\]  
proving $  p_{\mathrm{BP},2}(H|S,y) $ in (\ref{eq:p-bp-1-2-lst}).

The last statement comes from the fact that
$p_{\mathrm{BP},1}(H|S,y) \doteq p_{\mathrm{BP},2}(H|S,y)$ when $s_{ii} \doteq h_{ii}$ and $\sigma^2=-1$.

%%%=================================================
%%
%

%
%%
%%%=================================================
\section{Proofs for the theory of nearly flat surfaces}

In this section, we provide the remaining proofs of Section~\ref{sec:NFT}.

\subsection{Proof of Lemma~\ref{lemma:geometry-of-bp:NFS}}\label{sec:proof-geometry-of-bp:NFS}

The argument is parallel to that of Section~\ref{sec:scaling-law}.
This is shown in Section~5.3 of \cite{Shimodaira08}. 
From the definition and (\ref{eq:bpmodel}),
\[
\alpha_{\sigma^2}(H | y)
=P_{\sigma^2}(V^\ast \le v_h-h(U^\ast)\mid y)
= E_{\sigma^2}\biggl[ \bar\Phi\biggl(\frac{v-v_h+h(U^\ast)}{\sigma}\biggr)\,\biggm| u \biggr].
\]
Let $x=(v-v_h+E_{\sigma^2}h(u))/\sigma$ and $\epsilon=(h(U^\ast) - E_{\sigma^2}h(u))/\sigma$.
We have $\alpha_{\sigma^2}(H | y)=E_{\sigma^2}(\bar\Phi(x+\epsilon) | u)$.
Since $E_{\sigma^2}(\epsilon | u)=0$, considering the Taylor expansion (\ref{eq:phi-taylor}),
we have
$
\alpha_{\sigma^2}(H | y)
=\bar\Phi( (v- v_h+E_{\sigma^2}h(u))/\sigma ) + O(\lambda^2).
$

\subsection{Proof of Lemma~\ref{lemma:rej-surface-nfs}}\label{sec:proof-rej-surface-nfs}

For $\mu=(\theta,-h(\theta))\in \partial H$, by Lemma~\ref{lemma:geometry-of-bp:NFS},
we have
$
P_1(Y\in S^c\mid \mu) =\alpha_1(\Rc(s,v_s)\mid \mu) \simeq \bar\Phi(-h(\theta)-v_s+E_1s(\theta))
\simeq  \Phi(v_s)+\phi(v_s)\{h(\theta)-E_1s(\theta)\}.
$
We proceed by assuming  $r$ is nearly flat. Then,
$
P_1(Y\in R^c\mid \mu)
\simeq \bar\Phi(-h(\theta)-v_r+E_1r(\theta))
\simeq  \Phi(v_r)+\phi(v_r)\{h(\theta)-E_1r(\theta)\}.
$
Thus, from (\ref{eq:selec-rej-prob-nfs}), we obtain 
$
\bar\Phi(v_r)+\phi(v_r)\{E_1r(\theta)-h(\theta)\}
\simeq \alpha[ \bar\Phi(v_s)+\phi(v_s)\{E_1s(\theta)-h(\theta)\} ].
$
Subtracting $\bar\Phi(v_r)=\alpha \bar\Phi(v_s)$ from both sides, we have
$\phi(v_r)\{E_1r(\theta)-h(\theta)\}
\simeq
\alpha\phi(v_s)\{E_1s(\theta)-h(\theta)\}$.
Using the notation $C=\phi(v_s)/\phi(v_r)$, we have
$
E_1r(\theta) \simeq h(\theta) + \alpha C\{E_1s(\theta)-h(\theta)\}.
$
Applying  the inverse operator $E_{-1}$ to both sides, we get
$
r(u) \simeq E_{-1}h(u) + \alpha C\{s(u)-E_{-1}h(u)\}.
$
Since $E_{-1}h$ and $s$ are nearly flat, $r$ is also nearly flat.
By following the calculation so far, we can verify that this $r$ in fact satisfies (\ref{eq:selec-rej-prob-nfs}),
and therefore (\ref{eq:unbias-rej-surface}) is the solution of (\ref{eq:selec-rej-prob-nfs}).

For proving (\ref{eq:rej-surface-alpha}), 
first note that 
$
\phi(v_r)\{r(u)-E_{-1}h(u)\} \simeq \alpha \phi(v_s)\{s(u)-E_{-1}h(u)\}
$
from (\ref{eq:unbias-rej-surface}).
Combining this with $\bar\Phi(v_r)=\alpha \bar\Phi(v_s)$,  we obtain
$
\bar\Phi(v_r)+\phi(v_r)\{r(u)-E_{-1}h(u)\} \simeq \alpha [\bar\Phi(v_s)+\phi(v_s)\{s(u)-E_{-1}h(u)\}].
$
Thus, considering Taylor expansion (\ref{eq:phi-taylor}), we conclude
$
\bar\Phi(v_r-r(u)+E_{-1}h(u)) \simeq \alpha \bar\Phi(v_s-s(u)+E_{-1}h(u))
$.

%---
\subsection{Proof of Theorem~\ref{theo:gen-pval-si}}\label{sec:proof-gen-pval-si}

We first prove the existence of $r_k$.
Let $\tilde{r}_k$ denote the function before applying $\Fc^{-1}$.
It is sufficient to prove that $\|\tilde{r}_k\|_1<\infty$.
From the condition (iii), we have
\begin{align*}
\|\tilde{r}_k(\omega)\|_1
%&\le
%\|\{1-A(v_r)-J_k(\omega)\}e^{\|\omega\|^2/2}\tilde{h}(\omega) \|_1 
%+ \|\{A(v_r)-e^{\|\omega\|^2/2}I_k(\omega)\}\tilde{s}(\omega)\|_1\\
  & \le \|\{1-A(v_r)-J_k(\omega)\}e^{\|\omega\|^2/2}\|_\infty \, \|\tilde{h}(\omega) \|_1 \\
  &\quad + \{A(v_r)+\|e^{\|\omega\|^2/2}I_k(\omega)\|_\infty\} \, \|\tilde{s}(\omega)\|_1\\
  & <\infty.
\end{align*}
Thus, $r_k$ exists for each $k$.

Next, for a nearly flat function $r$, we consider a general $p$-value $p(H | S, (u,v))$ which has the following representation:
$
p(H | S, (u,v))=\bPhi(v+r(u))/\bPhi(v_s).
$
From (\ref{eq:bp-approx}), we have 
$
P_1(p(H | S,Y)<\alpha\mid \theta,-h(\theta))
\simeq
1-P_1(V\le v_r -r (U)\mid \theta,-h(\theta))
\simeq 
\bPhi(h(\theta)+v_r-E_1r(\theta))
$ 
and 
$
P_1(Y\in S\mid \theta,-h(\theta))
\simeq
\bPhi(h(\theta)+v_s-E_1s(\theta)).
$
For $\bPhi(v+a)/\bPhi(v_s+b)$ with terms $a$ and $b$ of order $O(\lambda)$, 
in the same manner as Lemma~\ref{lemma:phi-phi-taylor}, we have 
\banum\label{eq:rep-trans}
\frac{\bPhi(v+a)}{\bPhi(v_s+b)}\simeq \frac{\bPhi(v+r(u))}{\bPhi(v_s)}
\eanum
where $r(u)=a-A(v)b$.
In (\ref{eq:rep-trans}), letting $v=v_r$, $a=h(\theta)-E_1r(\theta),\;b=h(\theta)-E_{1}s(\theta)$,
we have that, for $\mu=(\theta,-h(\theta)) \in \partial H$,
\ba
\frac{P_1(p(H |S, Y)<\alpha\mid \mu)}{P_1(Y\in S\mid \mu)}
&\simeq
\frac{\bar{\Phi}(h(\theta)+v_r -E_1r(\theta))}{\bar{\Phi}(h(\theta)+v_s -E_1s(\theta))}\\
&\simeq
\frac{\bPhi\lsb v_r+h(\theta)-E_1r(\theta)
-A(v_r)\{ h(\theta) -E_{1}s(\theta) \}  \rsb}{\bPhi(v_s)}.
%\frac{P_1(v_r + \mathrm{bias}(\theta))}{\bar{\Phi}(v_s)}
\ea
If $h(\theta)-E_1r(\theta)-A(v_r)\{h(\theta)-E_1s(\theta)\} =0$, 
then the test using $p(H |S, y)$ is unbiased erring only $O(\lambda^2)$.
Thus, we will denote it by $\mathrm{bias}(\theta)$.
The function $r=r_\mathrm{SI}$ satisfying $\mathrm{bias}(\theta)= 0$ is given by
\banum\label{eq:unbiased-surf}
r_{\mathrm{SI}}(u) 
&= E_{-1}h(u)-A(v_r)\{E_{-1}h(u)-s(u)\}\nonumber \\
&=(1-A(v_r))E_{-1}h(u)+A(v_r)s(u).
\eanum
Note that $r_{\mathrm{SI}}(u) \simeq r(u)$ of (\ref{eq:unbias-rej-surface}) since $A(v_r)=\alpha C$.
%For $v=v_r+O(\lambda)$, 
%this function is equivalent to (\ref{eq:unbias-rej-surface}) ignoring $O(\lambda^2)$ terms.
%Let $\tilde{r}$ and $\tilde{r}_\mathrm{SI}$ denote Fourier transforms of $r$ and $r_\mathrm{SI}$,
%respectively. 
%The function $\mathrm{bias}(\theta)$ can be represented by
Since
$
\mathrm{bias}(\theta) 
= E_1r_{\mathrm{SI}}(\theta)-E_1r(\theta) 
%=
%\Fc^{-1}\lsb e^{-\|\omega\|^2/2}\{\tilde{r}_\mathrm{SI}(\omega)-\tilde{r}(\omega) \}\rsb (\theta).
$,
%
%Let $\widetilde{\mathrm{bias}}(\omega)$ be 
the Fourier transform of $\mathrm{bias}(\theta)$ is
%We have
\banum\label{eq:bias-fourier}
\widetilde{\mathrm{bias}}(\omega)=e^{-\|\omega\|^2/2}\{\tilde{r}_\mathrm{SI}(\omega)-\tilde{r}(\omega) \},
\eanum
where 
 $\tilde{r}$ and $\tilde{r}_\mathrm{SI}$ denote Fourier transforms of $r$ and $r_\mathrm{SI}$, respectively.
Note that the Fourier transform of $r_\mathrm{SI}$ is given by
\[
\tilde{r}_\mathrm{SI}(\omega)=(1-A(v_r))e^{\|\omega\|^2/2}\tilde{h}(\omega)+A(v_r)\tilde{s}(\omega).
\]

In (\ref{eq:bias-fourier}), replacing $\tilde{r}$ with $\tilde{r}_k$,
$\widetilde{\mathrm{bias}}_k(\omega)$ for $p_{k}(H | S, y)$ is represented by
$$
\widetilde{\mathrm{bias}}_k(\omega)
=J_k(\omega)\tilde{h}(\omega)+I_k(\omega)\tilde{s}(\omega).
$$
The condition (i) implies $\lim_{k\rightarrow \infty}\widetilde{\mathrm{bias}}_k(\omega)=0$ for each $\omega \in \Rb$. From the condition (ii), we have
\[
\left|e^{i\omega\cdot u}\widetilde{\mathrm{bias}}_k(\omega) \right|\le |J_k(\omega)\tilde{h}(\omega)|+|I_k(\omega)\tilde{s}(\omega)|
\le C\{ |\tilde{h}(\omega)|+|\tilde{s}(\omega)|\}.
\]
Moreover, $C\{ |\tilde{h}(\omega)|+|\tilde{s}(\omega)|\}$ is integrable 
since $\|\tilde{h}(\omega)\|_1,\|\tilde{s}(\omega)\|_1<\infty$.
Combining these results with Lebesgue's dominated convergence theorem,
\ba
\lim_{k\rightarrow \infty} \Fc^{-1}\lsb \widetilde{\mathrm{bias}}_k(\omega) \rsb(\theta)
&=\lim_{k\rightarrow \infty} \frac{1}{(2\pi)^m}\int e^{i\omega\cdot\theta}\widetilde{\mathrm{bias}}_k(\omega) \,d\omega\\
&= \frac{1}{(2\pi)^m}\int e^{i\omega\cdot\theta}\lim_{k\rightarrow \infty}\widetilde{\mathrm{bias}}_k(\omega) \,d\omega
=0.
\ea
Hence, we conclude that $\lim_{k\rightarrow \infty} {\mathrm{bias}}_k(\theta)=0$ for each $\theta$, 
which proves (\ref{eq:conv-gen-pval-si}).

Next, we consider the case that
$h$ and $s$ can be represented by polynomials of degree less than or equal to $2k-1$.
Let the condition (iv) holds. 
Then, according to the argument in Section~A.7 of \cite{Shimodaira08},
we have $\Fc^{-1}[J_k(\omega)\tilde{h}(\omega)]=0$ and $\Fc^{-1}[I_k(\omega)\tilde{s}(\omega)]=0$, 
which implies ${\mathrm{bias}}_k(\theta)=0$.
The last part comes from the idea that, for $h(u)=u_1^{b_1} \cdots u_m^{b_m}$, $\tilde h(\omega) \propto \delta^{(b_1)}(\omega_1)\cdots \delta^{(b_m)}(\omega_m)$, where $\delta^{(k)}$ is the $k$-th derivative of the Dirac delta function, so
$\int \tilde h(u) \omega_1^{c_1}\cdots \omega_m^{c_m} d\omega =0$ for 
$b_1+\cdots+b_m\le 2k-1$, $c_1+\cdots+c_m\ge2k$.

%---

%---
\subsection{Proof of Lemma~\ref{lemma:taylor-pval-si}}\label{sec:proof-taylor-pval-si}

For a general region $H=\Rc(h,v_{h})$, 
noting the scaling-law $\psi_{\sigma^2}(H|y)\simeq v-v_h + \mathcal{F}^{-1} [e^{-\sigma^2\|\omega\|^2/2} \tilde h(\omega)](u)$,
we have
\ba
\psi_{\sigma_a^2,k}(H | y,\sigma_{b}^2)
&\simeq v-v_h+\sum_{j=0}^{k-1}
\frac{(\sigma_a^2-\sigma_{b}^2)^j}{j!}
\frac{\partial^j}{\partial (\sigma^2)^j}\biggm|_{\sigma^2=\sigma_{b}^2}
\mathcal{F}^{-1}\left[\tilde{h}(\omega)e^{-\sigma^2\|\omega\|^2/2} \right](u)\\
&=v-v_h+
\mathcal{F}^{-1}\Biggl[
\tilde{h}(\omega)e^{-\sigma_{b}^2\|\omega\|^2/2}
\Biggl\{ \sum_{j=0}^{k-1}\frac{(\sigma_a^2-\sigma_{b}^2)^j}{j!}\left(-\frac{\|\omega\|^2}{2}\right)^j 
\Biggr\}\Biggr](u)\\
&=v-v_h+
\mathcal{F}^{-1}\left[\tilde{h}(\omega)e^{-\sigma_a^2\|\omega\|^2/2}
\left\{ 1- G_k(\omega | \sigma_a^2,\sigma_{b}^2)\right\}\right](u),
\ea
where the last equation comes from the definition of $G_k$ and the identity
\[
	\frac{\gamma(k,x)}{\Gamma(k)} = 1 - e^{-x} \sum_{j=1}^{k-1} \frac{x^j}{j!}
\]
with $x = (\sigma_b^2 - \sigma_a^2 ) \|\omega\|^2/2$.
By defining 
\[
  \tilde{h}_{\sigma_a^2;\sigma_b^2,k}(\omega)
=\tilde{h}(\omega)\{ 1-G_{k}(\omega | \sigma_a^2,\sigma_b^2) \},\quad
h_{\sigma_a^2;\sigma_b^2,k}(u)=\Fc^{-1}[\tilde{h}_{\sigma_a^2;\sigma_b^2,k}(\omega) ](u),
\]
this is rearranged as
$\psi_{\sigma_a^2,k}(H | y,\sigma_{b}^2) 
\simeq 
v-v_h+\Fc^{-1}[ e^{-\sigma_a^2\|\omega\|^2/2} \tilde{h}_{\sigma_a^2;\sigma_b^2,k}(\omega) ](u)
= %\simeq 
v-v_h+E_{\sigma_a^2} h_{\sigma_a^2;\sigma_b^2,k}(u)$.
%\simeq 
%v-v_h+\Fc^{-1}\lsb e^{-\sigma_a^2\|\omega\|^2/2} \tilde{h}_{\sigma_a^2;\sigma_b^2,k}(\omega) \rsb(u).
%\ea

For the hypothesis region $H=\Rc(h,0)$ and the selective region $S=\Rc(s,v_s)^c$, 
we define 
$
h_k(u) 
=h_{-1;\sigma_{-1}^2,k}(u)=\Fc^{-1}[\tilde{h}(\omega) ( 1 - G_{k}(\omega | -1,\sigma_{-1}^2) ) ](u)
$
and
$
s_k(u) 
= s_{0;\sigma_{0}^2,k}(u)=\Fc^{-1}[ \tilde{s}(\omega) (1 - G_{k}(\omega | 0,\sigma_{0}^2) )](u).
$
Using these notations, we have
$
\psi_{-1,k}(H | y,\sigma_{-1}^2) \simeq v+E_{-1}h_k(u)
$
and 
$ 
\psi_{0,k}(S | y,\sigma_0^2) \simeq -v + v_s -s_k(u).
$
Combining these, we obtain an expression
\[
p_{\mathrm{SI},k}(H | S,y)\simeq \frac{\bPhi(v+E_{-1}h_k(u))}{\bPhi(v_s-s_k(u)+E_{-1}h_k(u))}.
\]

%We next show the representation (\ref{eq:gen-pval-si}) for $p_{\mathrm{SI},k}(H | S,y)$.
%Similar with Lemma~\ref{lemma:phi-phi-taylor},
We recall that the function $r$ satisfying (\ref{eq:rep-trans}) can be represented by
$r(u) = a-A(v)b$. In this case, $a=E_{-1}h_k(u)$ and $b=-s_k(u)+E_{-1}h_k(u)$.
Thus, the function $r_k$ corresponding to $p_{\mathrm{SI},k}(H | S,y)$ is given by
\ba
r_k(u,v)
&= E_{-1}h_k(u)-A(v)\lpar -s_k(u)+E_{-1}h_k(u) \rpar\\
&= (1-A(v))E_{-1}h_k(u)+A(v)s_k(u),
\ea
and $r_k(u)=r_k(u,v_r)$.
Let $\tilde{r}_k(\omega)$ denote the Fourier transform of $r_k$.
We have
$
\tilde{r}_k(\omega)
= (1-A(v_r))e^{\|\omega\|^2/2}\tilde{h}_k(\omega)+A(v_r)\tilde{s}_k(\omega)
$.
By comparing this with $r_k(u)$ of Theorem~\ref{theo:gen-pval-si}, we obtain
\[
  J_k(\omega)=(1-A(v_r))G_{k}(\omega | -1,\sigma_{-1}^2),\quad
  I_k(\omega)=A(v_r)e^{-\|\omega\|^2/2}G_{k}(\omega | 0,\sigma_{0}^2).
\]
%Moreover, $\widetilde{\mathrm{bias}}_k(\omega)$ is represented as
%\ba
%\widetilde{\mathrm{bias}}_k(\omega)
%&=e^{-\|\omega\|^2/2}\lpar \tilde{r}_{\mathrm{SI}}(\omega)-\tilde{r}_k(\omega) \rpar\\
%&=e^{-\|\omega\|^2/2}\lsb  
%\lbr (1-A(v))e^{\|\omega\|^2/2}\tilde{h}(\omega) + A(v)\tilde{s}(\omega) \rbr  
%-\lbr (1-A(v))e^{\|\omega\|^2/2}\tilde{h}_k(\omega) + A(v)\tilde{s}_k(\omega) \rbr  
%\rsb\\
%&=(1-A(v))(\tilde{h}(\omega)-\tilde{h}_k(\omega)) + e^{-\|\omega\|^2/2}A(v)(\tilde{s}(\omega)-\tilde{s}_k(\omega)).
%\ea

We next show that $I_k$ and $J_k$ satisfy the conditions (i), (ii), (iii), (iv) 
in Theorem~\ref{theo:gen-pval-si}.
From $0\le \gamma (k,z)\le \Gamma(k)$, it is easy to check the condition (ii).
Based on (2.133) in \cite{GilEtAl07} and 
Stirling's approximation (i.e., $\Gamma(k+1)\sim \sqrt{2\pi k}(k/e)^k$),
it follows that 
\[
\frac{\gamma(k,z)}{\Gamma(k)}
=\frac{e^{-z}z^k}{\Gamma(k+1)}\{1+O(k^{-1})\}
\sim 
\frac{1}{\sqrt{2\pi}}\frac{e^{k-z}z^k}{k^{k+1/2}}\rightarrow 0\;\text{ as }\; k \rightarrow \infty.
\]
That is, for each $z$, $\lim_{k\rightarrow \infty} \gamma(k,z)/\Gamma(k)=0$, 
and the condition (i) is confirmed.
From the definition of $J_k$ and $G_k$, we have
$
1-A(v_r)-J_k(\omega)=(1-A(v_r))(1-G_k(\omega | -1,\sigma_{-1}^2))
=(1-A(v_r))e^{-x}\sum_{j=0}^{k-1}x^j/j!$,
where $x=(1+\sigma_{-1}^2)\|\omega\|^2/2$.
Then
$
(1-A(v_r)-J_k(\omega))e^{\|\omega\|^2/2}
=(1-A(v_r))\exp((1-(1+\sigma_{-1}^2))\|\omega\|^2/2)\sum_{j=0}^{k-1}\frac{x^j}{j!}
$
with the coefficient of the exponent $1-(1+\sigma_{-1}^2) = -\sigma_{-1}^2<0$.
So we can see that the condition (iii) of $J_k(\omega)$ is satisfied.
Similarly, since $|1-G_k(\omega | \sigma_a^2,\sigma_b^2)|<\infty$ for $\sigma_a^2< \sigma_b^2$,
 the condition (iii) for $I_k(\omega)$ is also satisfied.
From the last expression of $G_k(\omega)$ given in this lemma, it is represented as $\sum_{j=k}^\infty c_{k,j}\|\omega\|^{2j}$ with some coefficients $c_{k,j}$.
Hence, $I_k$ and $J_k$ also satisfy the condition (iv).

%---

%---
\subsection{Proof of Lemma \ref{lemma:ibp-pval-si}}\label{sec:proof-ibp-pval-si}

For $y=(u,v)\in \Rb^{m+1}$, let $\proj(H|y)\simeq(u^\prime,v^\prime)$.
From the Lipschitz continuity of $h$, it follows that $u^\prime=u+O(\lambda)$, $v^\prime\simeq -h(u)$.
Combining this fact with (\ref{eq:bp-approx}), we have 
$\bar\Phi(\psi_{\sigma^2}(H | y))\simeq \bPhi(v+E_{\sigma^2}h(u))$
and
$P_1(Y^\ast \in S \mid \proj(H|y))\simeq \bar{\Phi}(h(u)+v_s-E_1s(u))$.
This gives
\ba
p_{\mathrm{BP},1}(H\mid S, y)
\simeq
\frac{\bPhi(v+E_{\sigma^2}h(u))}{\bar{\Phi}(h(u)+v_s-E_1s(u))}.
\ea
From (\ref{eq:rep-trans}), we thus get
$r_1(u) \simeq E_{\sigma^2}h(u)-A(v_r)[h(u)-E_1s(u)]$ for $v=v_r+O(\lambda)$.
Then
$
\tilde{r}_1(\omega)
\simeq
e^{-\sigma^2\|\omega\|^2/2}\tilde{h}(\omega)-A(v_r)[ \tilde{h}(\omega)-e^{-\|\omega\|^2/2}\tilde{s}(\omega) ].
$
  
Let us assume the form of (\ref{eq:gen-pval-si}) as
$p_{\mathrm{BP},k}(H | S,y) \simeq \bar\Phi(v + r_k(u))/\bar\Phi(v_s)$ for $v=v_r+O(\lambda)$.
For $Y^*=(U^*,V^*)$,
$
P_1( p_{\mathrm{BP},k}(H | S,Y^*) \le p_{\mathrm{BP},k}(H | S,y) \mid \proj(H|y))
\simeq P_1(V^\ast + r_k(U^\ast)\ge v + r_k(u)\mid \proj(H|y))
\simeq \bar{\Phi}(h(u)+v+r_k(u)-E_1r_k(u)).
$
Thus, we have 
\[
p_{\mathrm{BP},k+1}(H | S,y)  \simeq \frac{\bar{\Phi}(h(u)+v+r_k(u)-E_1r_k(u))}{\bar{\Phi}(h(u)+v_s-E_1s(u))},
\]
showing the form  (\ref{eq:gen-pval-si}) is correct by induction.
From (\ref{eq:rep-trans}) again,
$
r_{k+1}(u) \simeq r_k(u)-E_1r_k(u) + h(u) - A(v_r)\{h(u)-E_1s(u)\}.
$
Applying the Fourier transform to the both sides, we have
$
\tilde{r}_{k+1}(\omega) 
\simeq 
(1-e^{-\|\omega\|^2/2})\tilde{r}_k(\omega)+\tilde{h}(\omega)
-A(v_r)\{\tilde{h}(u)-e^{-\|\omega\|^2/2}\tilde{s}(u)\}.
$
By solving this recurrence relation,
we have the specific forms of  $I_k$ and $J_k$ given in this lemma.

We next show that $I_k$ and $J_k$ satisfy the conditions (i), (ii), (iii), (iv) in Theorem~\ref{theo:gen-pval-si}.
Let $x=e^{-\|\omega\|^2/2}$ and $A=A(v_r)$ for short, and write $I_k = (1-x)^{k} A x$, $J_k=(1-x)^{k-1}(1-x^{1+\sigma^2} - A(1-x))$. Noting $0<x\le 1$ and
\[
  I_k = (1-x)^{k-1} I_1,\quad J_k= (1-x)^{k-1} J_1,
\]
we have $\lim_{k\to\infty} (1-x)^{k-1}=0$, showing (i) holds for $I_k$ and $J_k$.
The condition (ii) also holds because $||I_k||_\infty \le  \| I_1  \|_\infty < \infty$ and
$||J_k||_\infty \le  \| J_1  \|_\infty < \infty$ for $\sigma^2>-1$.
Next consider (iii). $(1-A-J_k)x^{-1} = (1-A)x^{-1}(1 - (1-x)^{k-1}) + (1-x)^{k-1}(x^{\sigma^2}-A )$, from which only terms of $x^j$ and $x^{j+\sigma^2}$, $j=0,1,\ldots,k-1$ appear. $I_k x^{-1} = (1-x)^k A$ has only terms of $x^j$, $j=0,1,\ldots,k$.
All these terms are bounded for $\sigma^2>0$, and so (iii) holds for $I_k$ and $J_k$.
For (iv), first note that $1-x^a = -\sum_{j=1}^\infty (-a \|\omega\|^2/2)^j/j!$ has only terms of $\|\omega\|^{2j}$, $j\ge1$, and so $(1-x)^{k-1}$ has only those of $j\ge{k-1}$. Since $I_1$ and $J_1$ have only those of $j\ge 1$, we conclude that $I_k$ and $J_k$ have only those of $j\ge k$, which shows (iv).

%%%=================================================
%%
%

% %\clearpage
% \bibliographystyle{imsart-nameyear}
% \bibliography{simbp_bib,stat2017}

\end{document}